\documentclass[10pt]{article}

\usepackage[english,plain]{fancyref}
\usepackage{hyperref}
\usepackage{amssymb}
\usepackage{amsmath}
\usepackage{mathtools}
\usepackage{paralist}
\input{liemacs10.sty}
\newcommand{\sV}{{\tt V}}

\renewcommand{\phi}{\varphi}

\renewcommand{\phi}{\varphi}
\renewcommand{\epsilon}{\varepsilon}

\newcommand*{\fancyrefproplabelprefix}{prop}
\newcommand*{\frefpropname}{Proposition}
\frefformat{plain}{\fancyrefproplabelprefix}{\frefpropname\fancyrefdefaultspacing#1}
\newcommand*{\fancyrefthmlabelprefix}{thm}
\newcommand*{\frefthmname}{Theorem}
\frefformat{plain}{\fancyrefthmlabelprefix}{\frefthmname\fancyrefdefaultspacing#1}
\newcommand*{\fancyreflemlabelprefix}{lem}
\newcommand*{\freflemname}{Lemma}
\frefformat{plain}{\fancyreflemlabelprefix}{\freflemname\fancyrefdefaultspacing#1}
\newcommand*{\fancyrefcorlabelprefix}{cor}
\newcommand*{\frefcorname}{Corollary}
\frefformat{plain}{\fancyrefcorlabelprefix}{\frefcorname\fancyrefdefaultspacing#1}
\renewcommand*{\fancyrefeqlabelprefix}{eq}

\frefformat{plain}{\fancyrefeqlabelprefix}{(#1)}
\newcommand*{\fancyrefproblabelprefix}{prob}
\newcommand*{\frefprobname}{Problem}
\frefformat{plain}{\fancyrefproblabelprefix}{\frefprobname\fancyrefdefaultspacing#1}
\newcommand*{\fancyrefdeflabelprefix}{def}
\newcommand*{\frefdefname}{Definition}
\frefformat{plain}{\fancyrefdeflabelprefix}{\frefdefname\fancyrefdefaultspacing#1}
\newcommand*{\fancyrefapplabelprefix}{app}
\newcommand*{\frefappname}{Appendix}
\frefformat{plain}{\fancyrefapplabelprefix}{\frefappname\fancyrefdefaultspacing#1}
\renewcommand*{\fancyrefseclabelprefix}{sec}
\renewcommand*{\frefsecname}{Section}
\frefformat{plain}{\fancyrefseclabelprefix}{\frefsecname\fancyrefdefaultspacing#1}
\newcommand*{\fancyrefsubseclabelprefix}{subsec}
\newcommand*{\frefsubsecname}{Subsection}
\frefformat{plain}{\fancyrefsubseclabelprefix}{\frefsubsecname\fancyrefdefaultspacing#1}
\newcommand*{\fancyrefconlabelprefix}{con}
\newcommand*{\frefconname}{Conjecture}
\frefformat{plain}{\fancyrefconlabelprefix}{\frefconname\fancyrefdefaultspacing#1}
\newcommand*{\fancyrefremlabelprefix}{rem}
\newcommand*{\frefremname}{Remark}
\frefformat{plain}{\fancyrefremlabelprefix}{\frefremname\fancyrefdefaultspacing#1}
\newcommand*{\fancyrefexlabelprefix}{ex}
\newcommand*{\frefexname}{Example}
\frefformat{plain}{\fancyrefexlabelprefix}{\frefexname\fancyrefdefaultspacing#1}

\DeclarePairedDelimiterX\braket[2]{\langle}{\rangle}{#1 \delimsize\vert #2}

\begin{document} 

\title{Outgoing monotone geodesics of standard subspaces} 
\author{\textbf{Jonas Schober}
\\ \small Department of Mathematics
\\ \small Brigham Young University, Provo, Utah, USA
\\ \small E-mail: schober@math.byu.edu}

\date{}

\maketitle

\abstract
We prove a real version of the Lax--Phillips Theorem and classify outgoing reflection positive ortho\-gonal one-parameter groups. Using these results, we provide a normal form for outgoing monotone geodesics in the set \(\mathrm{Stand}(\cH)\) of standard subspaces on some complex Hilbert space~\(\cH\). As the modular operators of a standard subspace are closely related to positive Hankel operators, our results are obtained by constructing some explicit symbols for positive Hankel operators. We also describe which of the monotone geodesics in \(\mathrm{Stand}(\cH)\) arise from the unitary one-parameter groups described in Borchers' Theorem and provide explicit examples of monotone geodesics that are not of this type.

\tableofcontents


\newpage
\section*{Introduction}
Standard subspaces are much-studied objects in algebraic quantum field theory (AQFT). Here, given a complex Hilbert space \(\cH\), a \textit{standard subspace} is a real subspace \(\sV \subeq \cH\) with
\[\sV \cap i\sV = \{0\} \qquad \text{and} \qquad \oline{\sV + i\sV} = \cH\]
(\cite{Lo08}, \cite{Ne21}, \cite{MN21}, \cite{NO21}, \cite{NOO21}). Given a standard subspace \(\sV \subeq \cH\), one has the densely defined \textit{Tomita operator}
\[T_\sV : \sV+i\sV \to \cH, \quad x+iy \mapsto x-iy.\]
Then, by polar decomposition
\[T_\sV = J_\sV \Delta_\sV^{\frac 12}\]
one obtains an anti-unitary involution \(J_\sV\) and the positive densely defined operator \(\Delta_\sV\). This pair \((\Delta_\sV,J_\sV)\) is referred to as the \textit{pair of modular objects} of \(\sV\).

Now, given a complex Hilbert space \(\cH\), one is interested in the geometric structure of the set \(\mathrm{Stand}(\cH)\) of standard subspaces of \(\cH\) (\cite{Ne18}). \(\mathrm{Stand}(\cH)\) is a so-called reflection space (see \cite[Def.~1.1]{Ne18}), which brings up the natural question of classifying \textit{geodesics}, i.e. morphisms of reflection spaces, from \(\R\) to \(\mathrm{Stand}(\cH)\). Concretely, one would like to understand geodesics \(\gamma: \R \to \mathrm{Stand}(\cH)\) that are \textit{monotone} in the sense that, for \(t_1 \leq t_2\), one has \(\gamma(t_1) \subeq \gamma(t_2)\). By \cite[Prop.~2.9]{Ne18}, under an additional strong continuity assumption, such monotone geodesics are of the form \(\gamma(t) = U_t \sV\), where \(\sV \subeq \cH\) is a standard subspace and \((U_t)_{t \in \R}\) is a strongly continuous unitary one-parameter group satisfying
\begin{equation}\label{eq:monGeoCond}
J_\sV U_t = U_{-t} J_\sV \quad \forall t \in \R
\end{equation}
and
\[\qquad U_t \sV \subeq \sV \quad \forall t \in \R_+.\]
In general, a description of all monotone geodesics in \(\mathrm{Stand}(\cH)\) is an open problem. In the present paper, we will give a normal form for these monotone geodesics under the additional assumption that
\[\bigcap_{t \in \R} U_t \sV = \{0\} \qquad \text{and} \qquad \overline{\bigcup_{t \in \R} U_t \sV} = \cH.\]
Triples \((\cH,\sV,U)\) satisfying these conditions are called \textit{outgoing} (\cite[Sec.~XI.11]{RS79}). Outgoing subspaces are well-studied objects in the area of optical and acoustical scattering theory (\cite[Chap.~XI.11]{RS79}) and are usually considered in the context of complex Hilbert spaces, where a normal form is provided by the Lax--Phillips Theorem (\cite[Thm.~XI.82]{RS79}). The Lax--Phillips Theorem states that, given a complex Hilbert space \(\cE\), a subspace \(\cE_+ \subeq \cE\) and a unitary one-parameter group \((U_t)_{t \in \R}\), one has
\[\bigcap_{t \in \R} U_t \cE_+ = \{0\} \qquad \text{and} \qquad \overline{\bigcup_{t \in \R} U_t \cE_+} = \cE,\]
if and only if there exists a complex Hilbert space \(\cK\) such that
\[\cE \cong L^2(\R,\cK) \qquad \text{and} \qquad \cE_+ \cong L^2(\R_+,\cK)\]
with the operators \((U_t)_{t \in \R}\) being realized by
\[(U_tf)(x) = f(x-t).\]
We, on the other hand, are interested in triples where \(\cE_+\) is a standard subspace and therefore, in particular, a real Hilbert space. A real version of the Lax--Phillips Theorem, as we need it, was proven in the unpublished Bachelor thesis \cite{Fr20}, written in German. To make this result accessible to a broader audience, in the present paper, we provide an alternative proof for this real version of the Lax--Phillips Theorem.

\newpage
Considering outgoing triples of the form \((\cH,\sV,U)\), the real version of the Lax--Phillips Theorem together with condition \fref{eq:monGeoCond} can be used to relate \(J_\sV\) to a Hankel operator, which can be shown to be positive. Here, given a complex Hilbert space \(\cK\), Hankel operators are bounded operators \(H\) on the \(\cK\)-valued Hardy space \(H^2(\C_+,\cK)\) satisfying
\[HS_t = S_t^*H \qquad \forall t \in \R_+,\]
where the operators \(S_t \in B(H^2(\C_+,\cK))\) are given by
\[(S_t f)(x) = e^{itx} f(x).\]
This relation between \(J_\sV\) and positive Hankel operators has been explored in \cite[Sec.~4, Sec.~5]{Sc23} and \cite{ANS22}. In the scalar-valued case \(\cK = \C\), in \cite[Sec.~4]{ANS22}, given a positive Hankel operator \(H \in B(H^2(\C_+))\), explicit symbols, i.e. functions \(h \in L^\infty(\R,\C)\) with
\[\braket*{f}{Hg} = \int_\R \overline{f(x)}h(x) g(-x) \,dx,\]
are constructed. We will adapt the ideas given there to the vector-valued case and construct a much broader class of concrete symbols (\fref{thm:SchoberWithProjections}).

Using this real Lax--Phillips Theorem and these symbols for positive Hankel operators, we will prove the main result of this paper (\fref{thm:apperingClassi}), which provides a normal form for outgoing monotone geodesics of standard subspaces realized on some \(L^2\)-space over the real line.

A class of examples of such outgoing monotone geodesics is given by Borchers' Theorem (\cite[Thm.~II.9]{Bo92}, \cite[Thm.~2.2.1]{Lo08}). Borchers' Theorem states that if the self-adjoint infinitesimal generator
\[\partial U \coloneqq \lim_{t \to 0} \frac{1}{it}(U_t-\textbf{1})\]
is either strictly positive or negative, then one has the commutation relations
\[\Delta_\sV^{is} U_t \Delta_\sV^{-is} = U_{e^{\mp 2\pi s} t}, \qquad J_\sV U_t J_\sV = U_{-t} \qquad \forall s,t \in \R.\]
This yields an anti-unitary representation of the affine group \(\mathrm{Aff}(\R)\) on the Hilbert space \(\cH\). By the representation theory of \(\mathrm{Aff}(\R)\) (\fref{app:AffineGroup}), one then knows that there exists a real Hilbert space \(\cM\) such that
\[\cH \cong L^2(\R,\cM) \qquad \text{and} \qquad \sV \cong L^2(\R_+,\cM)\]
with the operators \((U_t)_{t \in \R}\) being realized by
\[(U_tf)(x) = f(x-t).\]
It then is easy to see that \((\cH,\sV,U)\) is outgoing. A central result of this paper (\fref{thm:BorchersClassi}) is to show how one can identify monotone geodesics that arise from Borchers' Theorem in the normal form for general outgoing monotone geodesics of standard subspaces provided before.

~

\noindent \textbf{This paper is structured as follows:}

In \fref{subsec:RPOG}, we provide a real version of the Lax--Phillips Theorem both in the picture of position and momentum space. The momentum space version of the real Lax--Phillips Theorem (\fref{thm:LaxPhillips}) then states that, if \((\cE,\cE_+,U)\) is outgoing, then there exists a real Hilbert space \(\cM\) such that
\[\cE \cong L^2(\R,\cM_\C)^\sharp, \qquad \cE_+ \cong H^2(\C_+,\cM_\C)^\sharp \qquad \text{and} \qquad U_t \cong S_t, \quad t \in \R\]
with
\[(S_tf)(x) = e^{itx} f(x).\]
Here, defining
\[\cC_\cM: \cM_\C \cong \C \otimes_\R \cM \to \cM_\C, \quad z \otimes v \mapsto \overline{z} \otimes v,\]
we denote
\[L^2(\R,\cM_\C)^\sharp \coloneqq \{f \in L^2(\R,\cM_\C) : (\forall x \in \R) \,f(-x) = \cC_\cM f(x)\}\]
and
\[H^2(\C_+,\cM_\C)^\sharp \coloneqq H^2(\C_+,\cM_\C) \cap L^2(\R,\cM_\C)^\sharp\]
with \(H^2(\C_+,\cM_\C)\) being the Hardy space on the upper half-plane with values in \(\cM_\C\).

In \fref{subsec:RefPos}, we will extend this framework by an involution \(\theta\), representing the modular operator \(J_\sV\). We will show that these involutions are of the form \(\theta=\theta_h\) for some function \(h \in L^\infty(\R,B(\cM_\C))\) where
\[(\theta_h f)(x) \coloneqq h(x) f(-x)\]
and that, denoting by \(P_+\) the orthogonal projection onto \(H^2(\C_+,\cM_\C)^\sharp\), the operator
\[H_h \coloneqq P_+ \theta_h P_+\]
is a positive Hankel operator (\fref{thm:refPos}).

In \fref{sec:Hankel}, given a complex Hilbert space \(\cK\) and a positive Hankel operator \(H\) on \(H^2(\C_+,\cK)\), we will construct \textit{symbols} for this Hankel operator, i.e. functions \(h \in L^\infty(\R,B(\cK))\) such that \(H = H_h\). We use the fact that, for every positive Hankel operator \(H\), there exists an operator valued measure \(\mu\) on \(\R_+\), called the \textit{Carleson measure} of \(H\), such that \(H = H_\mu\) with
\[\braket*{f}{H_\mu g} \coloneqq \int_{\R_+} \braket*{f(i\lambda)}{d\mu(\lambda) g(i\lambda)} \qquad \forall f,g \in H^2\left(\C_+,\cK\right)\]
(\fref{prop:Carleson}). Then, in \fref{subsec:moreSymbols}, given a projection \(p \in B(\cK)\) and some operator \(C \in B((1-p)\cK,p\cK)\), we construct symbols \(\beta(\mu,p,C) \in L^\infty(\R,B(\cK))\) for the Hankel operator \(H_\mu\) (\fref{thm:SchoberWithProjections}) and show that they are the unique symbols \(h\) of \(H_\mu\) satisfying
\[h(-x) = h^*(x) = -(2p-\textbf{1})h(x)(2p-\textbf{1})\]
(\fref{thm:SchoberProjUnique}).

In \fref{subsec:standardQuad} we show that a quadruple \(\left(L^2(\R,\cM_\C)^\sharp,H^2(\C_+,\cM_\C)^\sharp,S,\theta_h\right)\) is \textit{standard}, i.e. orthogonally equivalent to \((\cH,\sV,U,J_\sV)\) for some standard subspace \(\sV\) on some complex Hilbert space \(\cH\) and a strongly continuous unitary one-parameter group \((U_t)_{t \in \R}\) satisfying
\[J_\sV U_t = U_{-t} J_\sV \quad \forall t \in \R \qquad \text{and} \qquad U_t \sV \subeq \sV \quad \forall t \in \R_+,\]
if and only if \(h = \beta(\mu,p,C)\) for some Carleson measure \(\mu\), some projection \(p \in B(\cM_\C)\) and some operator \(C \in B((1-p)\cM_\C,p\cM_\C)\) (\fref{thm:apperingClassi}). Further, in \fref{subsec:BorchersQuad}, we show that among these symbols \(\beta(\mu,p,C)\), the ones arising from Borchers' Theorem, i.e. the ones where the infinitesimal generator of the unitary one-parameter group \((U_t)_{t \in \R}\) is a direct sum of a strictly positive and a strictly negative operator, are precisely the ones with \(\mu\) being twice the Lebesgue measure on \(\R_+\) and \(C = 0\), i.e.
\[h = \beta(2\lambda_1,p,0) = i \cdot \sgn \cdot \textbf{1}\] (\fref{thm:BorchersClassi}, \fref{ex:2LebesgueBeta}).

Finally, in \fref{subsec:exampleQuad}, we construct an explicit class of functions \(h = \beta(\mu,p,C)\) such that the corresponding quadruple \(\left(L^2(\R,\cM_\C)^\sharp,H^2(\C_+,\cM_\C)^\sharp,S,\theta_h\right)\) is standard but not of Borchers-type (\fref{thm:standardNotBorchers}).

\newpage
\section{Outgoing one-para\-meter groups}\label{sec:RPOG}

In this section, we will provide a normal form for outgoing subspaces in complex and real Hilbert spaces. In the complex case, this is provided by the classical Lax--Phillips Theorem, whereas in the real case, we will prove a real version of the Lax--Phillips Theorem. Later, we will include orthogonal involutions into the picture and provide a normal form for outgoing reflection positive orthogonal one-parameter groups (cf. \fref{def:RefPosDef}).

\subsection{The Lax--Phillips Theorem}\label{subsec:RPOG}
We start with the definition of an outgoing subspace:
\begin{definition}
Let \(\cE\) be a real or complex Hilbert space, \(\cE_+ \subeq \cE\) be a closed subspace and \((U_t)_{t \in \R} \subeq \U(\cE)\) be a orthogonal/unitary one-parameter group. We call \((\cE,\cE_+,U)\) \textit{outgoing} if
\[U_t \cE_+ \subeq \cE_+ \qquad \forall t \geq 0\]
and
\[\bigcap_{t \in \R} U_t \cE_+ = \{0\} \qquad \text{and} \qquad \overline{\bigcup_{t \in \R} U_t \cE_+} = \cE.\]
\end{definition}
A classical result providing a normal form for outgoing subspaces on a complex Hilbert space is the Lax--Phillips Theorem:
\begin{theorem}\label{thm:LaxPhillipsComplex}{\rm \textbf{(Lax--Phillips)}(\cite[Thm.~XI.82]{RS79})}
Let \(\cE\) be a complex Hilbert space, \(\cE_+ \subeq \cE\) be a closed subspace and \((U_t)_{t \in \R} \subeq \U(\cE)\) be a unitary one-parameter group. Then \((\cE,\cE_+,U)\) is outgoing, if and only if there exists a complex Hilbert space \(\cM\) and a unitary map
\[\psi: \cE \to L^2(\R,\cM)\]
such that
\[\psi(\cE_+) = L^2(\R_+,\cM) \qquad \text{and} \qquad \psi \circ U_t = \hat S_t \circ \psi \quad \forall t \in \R,\]
where the orthogonal shift operators \({\hat S_t: L^2(\R,\cM) \to L^2(\R,\cM)}\) are given by
\[(\hat S_tf)(x) = f(x-t), \qquad t \in \R, f \in L^2(\R,\cM).\]
\end{theorem}
We now want to prove a real version of the Lax-Phillips Theorem. For this, we need the following definition:
\begin{definition}\label{def:complexification}
Let \(\cM\) be a real Hilbert space. We set \(\cM_\C \coloneqq \C \otimes_\R \cM\) and define the complex conjugation
\[\cC_\cM: \cM_\C \to \cM_\C, \quad z \otimes v \mapsto \overline{z} \otimes v.\]
Further, given an operator \(A \in B(\cM)\), we define an operator \(A_\C \in B(\cM_\C)\) by
\[A_\C (z \otimes v) = z \otimes (A v), \qquad z \in \C, v \in \cM.\]
\end{definition}
Now, the real Lax--Phillips Theorem states the following:
\begin{theorem}{\rm \textbf{(Lax--Phillips, real version)}(\cite[Thm.~3.3]{Fr20})}\label{thm:frankenbach}
Let \(\cE\) be a real Hilbert space, \(\cE_+ \subeq \cE\) be a closed subspace and \((U_t)_{t \in \R} \subeq \U(\cE)\) be an orthogonal one-parameter group. Then \((\cE,\cE_+,U)\) is outgoing, if and only if there exists a real Hilbert space \(\cM\) and an orthogonal map
\[\psi: \cE \to L^2(\R,\cM)\]
such that
\[\psi(\cE_+) = L^2(\R_+,\cM) \qquad \text{and} \qquad \psi \circ U_t = \hat S_t \circ \psi \quad \forall t \in \R,\]
where the orthogonal shift operators \({\hat S_t: L^2(\R,\cM) \to L^2(\R,\cM)}\) are given by
\[(\hat S_tf)(x) = f(x-t), \qquad t \in \R, f \in L^2(\R,\cM).\]
\end{theorem}
\newpage
\begin{proof}
It is easy to see that triples of the form \((\cE,\cE_+,U) \cong (L^2(\R,\cM),L^2(\R_+,\cM),\hat S_t)\) are outgoing, so we now assume that the triple \((\cE,\cE_+,U)\) is outgoing and show that it is of this form. First, we notice that \((\cE,\cE_+,U)\) is outgoing, if and only if \((\cE_\C,(\cE_+)_\C,U_\C)\) is outgoing. Therefore, by the Lax--Phillips Theorem (\fref{thm:LaxPhillipsComplex}), there exists a complex Hilbert space \(\tilde \cM\) and a unitary map
\[\psi: \cE_\C \to L^2(\R,\tilde \cM)\]
such that
\[\psi((\cE_+)_\C) = L^2(\R_+,\tilde \cM) \qquad \text{and} \qquad \psi \circ (U_t)_\C = \hat S_t \circ \psi \quad \forall t \in \R.\]
Identifying \(\cE\) with \(\R \otimes_\R \cE \subeq \cE_\C\) and the subspace \(\cE_+\) with \(\R \otimes_\R \cE_+\) it just remains to show that there exists a real subspace \(\cM \subeq \tilde \cM\) such that
\[\psi(\cE) = L^2(\R,\cM) \qquad \text{and} \qquad \psi(\cE_+) = L^2(\R_+,\cM).\]
For this, we notice that \(\cC_\cE\) commutes with the unitary one-parameter group \(((U_t)_\C)_{t \in \R}\), which implies that the operator \(\theta \coloneqq \psi \circ \cC_\cE \circ \psi^{-1}\) commutes with the unitary one-parameter group \((\hat S_t)_{t \in \R}\). On the other hand \(\cC_\cE\) commutes with the orthogonal projection onto \((\cE_+)_\C\), so \(\theta\) commutes with the orthogonal projection onto \(L^2(\R_+,\tilde \cM)\). This, by \fref{cor:AntiLinTrivial}, implies that there exists an anti-linear operator \(\sigma\) on \(\tilde \cM\) such that
\[(\theta f)(x) = \sigma f(x), \qquad x \in \R.\]
Since \(\cC_\cE\) and therefore \(\theta\) is an anti-linear involution, also \(\sigma\) must be an anti-linear involution, which implies that, setting
\[\cM \coloneqq \ker(\sigma - \textbf{1}),\]
we have \(\tilde \cM \cong \C \otimes_\R \cM\). Now we notice that
\[\ker(\cC_\cE - \textbf{1}) = \R \otimes_\R \cE \cong \cE \qquad \text{and} \qquad \ker(\cC_\cE - \textbf{1}) \cap (\cE_+)_\C = \R \otimes_\R \cE_+ \cong \cE_+.\]
Then, the statement follows from
\[\psi(\cE) = \psi \ker(\cC_\cE - \textbf{1}) = \ker(\theta - \textbf{1}) = L^2(\R,\ker(\sigma - \textbf{1})) = L^2(\R,\cM)\]
and
\[\psi(\cE_+) = \psi \left(\ker(\cC_\cE - \textbf{1}) \cap (\cE_+)_\C\right) = L^2(\R,\cM) \cap L^2(\R_+,\tilde \cM) = L^2(\R_+,\cM). \qedhere\]
\end{proof}
Instead of using the Lax--Phillips theorem as stated above, we want to use the Fourier transform to pass to momentum space:
\begin{definition}
Let \(\cM\) be a real Hilbert space. On the Hilbert space \(L^2(\R,\cM_\C)\) we define an involution \(\sharp\) by
\[f^\sharp(x) = \cC_\cM f(-x), \qquad x \in \R\]
and for any subspace \(H \subeq L^2(\R,\cM_\C)\), we set
\[H^\sharp \coloneqq \{f \in H: f^\sharp = f\}.\]
\end{definition}
Denoting by \(H^2(\C_+,\cM_\C) \cong H^2(\C_+) \otimes \cM_\C\) the Hardy space on the upper half-plane with values in \(\cM_\C\) (\fref{def:HardyScalarDef}, \fref{def:HardyVectorDef}), we can now formulate the momentum space version of the Lax--Phillips Theorem:
\begin{theorem}{\rm \textbf{(Lax--Phillips, momentum space version)}}\label{thm:LaxPhillips}
Let \(\cE\) be a real Hilbert space, \(\cE_+ \subeq \cE\) be a closed subspace and \((U_t)_{t \in \R} \subeq \U(\cE)\) be an orthogonal one-parameter group. Then \((\cE,\cE_+,U)\) is outgoing, if and only if there exists a real Hilbert space \(\cM\) and an orthogonal map
\[\psi: \cE \to L^2(\R,\cM_\C)^\sharp\]
such that
\[\psi(\cE_+) = H^2(\C_+,\cM_\C)^\sharp \qquad \text{and} \qquad \psi \circ U_t = S_t \circ \psi \quad \forall t \in \R,\]
where
\[(S_tf)(x) = e^{itx}f(x), \qquad t,x \in \R, f \in L^2(\R,\cM_\C)^\sharp.\]
\end{theorem}
\begin{proof}
We consider the unitary Fourier transform
\[\cF: L^2(\R,\cM_\C) \to L^2(\R,\cM_\C), \quad (\cF f)(x) = \frac 1{\sqrt{2\pi}} \int_\R e^{-ixp} f(p) \,dp.\]
Then one has
\[\cF L^2(\R,\cM_\C)^\sharp = L^2(\R,\cM) \qquad \text{and} \qquad \cF H^2(\C_+,\cM_\C)^\sharp = L^2(\R_+,\cM)\]
and
\[\cF \circ S_t = \hat S_t \circ \cF.\]
The statement then follows immediately from \fref{thm:frankenbach}.
\end{proof}

\subsection{Reflection positivity}\label{subsec:RefPos}
In this subsection, we consider reflection positive orthogonal/unitary one-parameter groups. They are defined as follows:
\begin{definition}\label{def:RefPosDef}(\cite[Def.~2.1.1]{NO18})
Let \(\cE\) be a real or complex Hilbert space, \(\cE_+ \subeq \cE\) be a closed subspace and \(\theta \in \U(\cE)\). Then the triple \((\cE,\cE_+,\theta)\) is called \textit{reflection positive}, if
\[\braket*{\xi}{\theta \xi} \geq 0 \qquad \forall \xi \in \cE_+.\]
Now let \((U_t)_{t \in \R}\) be a strongly continuous orthogonal/unitary one-parameter group on \(\cE\). We say that the quadruple \((\cE,\cE_+,U,\theta)\) is a \textit{reflection positive orthogonal/unitary one-parameter group}, if \((\cE,\cE_+,\theta)\) is a reflection positive Hilbert space and
\[\theta U_t = U_{-t} \theta \quad \forall t \in \R \qquad \text{and} \qquad U_t \cE_+ \subeq \cE_+ \quad \forall t \in \R_+.\]
\end{definition}
The goal of this subsection is to provide a normal form for outgoing reflection positive orthogonal one-parameter groups, i.e. reflection positive orthogonal one-parameter groups \((\cE,\cE_+,U,\theta)\) for which the triple \((\cE,\cE_+,U)\) is outgoing. For this, we need the following definitions:
\begin{definition}
Let \(\cM\) be a real Hilbert space. We define involutions \(\sharp\) and \(\flat\) on the space \(L^\infty\left(\R,B(\cM_\C)\right)\) by
\begin{equation*}
f^\sharp \left(x\right) \coloneqq \cC_\cM f\left(-x\right) \cC_\cM \qquad \text{and} \qquad f^\flat \left(x\right) \coloneqq f\left(-x\right)^*
\end{equation*}
and set
\begin{equation*}
L^\infty\left(\R,\U(\cM_\C)\right) \coloneqq \left\lbrace f \in L^\infty\left(\R,B(\cM_\C)\right) : f \cdot f^* = f^* \cdot f = \textbf{1} \right\rbrace.
\end{equation*}
and
\begin{equation*}
L^\infty\left(\R,\U(\cM_\C)\right)^\sharp \coloneqq \left\lbrace f \in L^\infty\left(\R,\U(\cM_\C)\right) : f^\sharp = f\right\rbrace.
\end{equation*}
and
\begin{equation*}
L^\infty\left(\R,\U(\cM_\C)\right)^{\sharp,\flat} \coloneqq \left\lbrace f \in L^\infty\left(\R,\U(\cM_\C)\right) : f^\sharp = f^\flat = f\right\rbrace.
\end{equation*}
\end{definition}
\begin{definition}\label{def:Hh}
Let \(\cK\) be a complex Hilbert space. We define
\[R: L^2(\R,\cK) \to L^2(\R,\cK), \quad (Rf)(x) = f(-x).\]
Further, for a function \(h \in L^\infty\left(\R,B(\cK)\right)\), we denote the corresponding multiplication operator on \(L^2(\R,\cK)\) by \(M_h\) and define
\[\theta_h \coloneqq M_h R \in B(L^2(\R,\cK))\]
and
\[H_h \coloneqq P_+ \theta_h P_+ \in B(H^2(\C_+,\cK)),\]
where \(P_+\) denotes the orthogonal projection onto \(H^2(\C_+,\cK)\).
\end{definition}
Given these definitions, we have the following theorem:
\begin{thm}{\rm (\cite[Prop.~4.1.7, Thm.~4.1.11, Lem.~4.3.25]{Sc23})}\label{thm:refPos}
Let \(\cE\) be a real Hilbert space, \(\cE_+ \subeq \cE\) be a closed subspace, \((U_t)_{t \in \R} \subeq \U(\cE)\) be an orthogonal one-parameter group and \(\theta \in \U(\cE)\) be an involution. Then \((\cE,\cE_+,U,\theta)\) is an outgoing reflection positive orthogonal one-parameter group, if and only if there exists a real Hilbert space \(\cM\), a function \(h \in L^\infty\left(\R,\U(\cM_\C)\right)^{\sharp,\flat}\) and an orthogonal map
\[\psi: \cE \to L^2(\R,\cM_\C)^\sharp\]
such that
\[\psi(\cE_+) = H^2(\C_+,\cM_\C)^\sharp \qquad \text{and} \qquad \psi \circ U_t = S_t \circ \psi \quad \forall t \in \R\]
and
\[\psi \circ \theta = \theta_h \circ \psi \qquad \text{and} \qquad H_h \geq 0.\]
\end{thm}
\begin{proof}
In view of the Lax--Phillips Theorem (\fref{thm:LaxPhillips}), it suffices to classify reflection positive orthogonal one-parameter groups of the form
\[(L^2(\R,\cM_\C)^\sharp,H^2(\C_+,\cM_\C)^\sharp,S,\theta).\]
Now, using that \(R\) is an involution satisfying
\[R S_t = S_{-t} R \qquad \forall t \in \R,\]
the condition
\[\theta S_t = S_{-t} \theta \qquad \forall t \in \R\]
is equivalent to
\[(\theta R) S_t = S_t (\theta R) \qquad \forall t \in \R,\]
which, by \fref{prop:SCommutant}, is equivalent to \(\theta R = M_h\) for some function \(h \in L^\infty\left(\R,B(\cM_\C)\right)\), i.e. to
\[\theta = M_h R = \theta_h.\]
Since \(\theta = \theta_h\) defines an orthogonal operator on \(L^2(\R,\cM_\C)^\sharp\), the function \(h\) must have unitary values, i.e. \(h \in L^\infty\left(\R,\U(\cM_\C)\right)\) and satisfy \(h = h^\sharp\). Further, one has
\[(\theta_h \theta_h f)(x) = h(x)h(-x)f(x),\]
which shows that \(\theta_h\) is a unitary involution, if and only if \(h(-x) = h(x)^*\) for almost every \(x \in \R\), i.e. \(h = h^\flat\). This shows that
\[\theta S_t = S_{-t} \theta \qquad \forall t \in \R\]
is satisfied, if and only if there exists a function \(h \in L^\infty\left(\R,\U(\cM_\C)\right)^{\sharp,\flat}\) such that \(\theta = \theta_h\).

Finally, using that
\[L^2\left(\R,\cM_\C\right) = \C \cdot L^2\left(\R,\cM_\C\right)^\sharp \qquad \text{and} \qquad H^2\left(\C_+,\cM_\C\right) = \C \cdot H^2\left(\C_+,\cM_\C\right)^\sharp,\]
the reflection positivity condition
\[\braket*{f}{\theta_h f} \geq 0 \qquad \forall f \in H^2\left(\C_+,\cM_\C\right)^\sharp\]
is equivalent to
\[\braket*{f}{\theta_h f} \geq 0 \qquad \forall f \in H^2\left(\C_+,\cM_\C\right),\]
which, in turn, is equivalent to
\[H_h = P_+ \theta_h P_+ \geq 0. \qedhere\]
\end{proof}
\begin{remark}
The operators \(H_h\) appearing in this theorem are so-called Hankel operators. This motivates our construction of symbols for positive Hankel operators in the following section.
\end{remark}

\newpage
\section[Constructing symbols for positive Hankel operators]{\hspace{-0.09em}Constructing symbols for positive Hankel operators}\label{sec:Hankel}
In this section, we will consider positive Hankel operators and their symbols as introduced in the following:
\begin{definition}{\rm (\cite[Def.~1.3, Thm.~3.5]{ANS22})}
Let \(\cK\) be a complex Hilbert space. An operator \(H \in B(H^2\left(\C_+,\cK\right))\) is called \textit{\(\cK\)-Hankel operator}, if
\[H\tilde S_t = \tilde S_t^*H \qquad \forall t\in \R_+,\]
where
\[\tilde S_t \coloneqq P_+ S_t P_+ \in B(H^2\left(\C_+,\cK\right)).\]
\end{definition}
Examples of such \(\cK\)-Hankel operators are given by the operators \(H_h\) introduced in \fref{def:Hh}:
\begin{lemma}{\rm (\cite[Lem.~1.4]{ANS22})}
Let \(\cK\) be a complex Hilbert space and \(h \in L^\infty(\R,B(\cK))\). Then the operator \(H_h\) is a \(\cK\)-Hankel operator.
\end{lemma}
\begin{proof}
For \(t \in \R_+\) one has
\[S_t H^2(\C_+,\cK) \subeq H^2(\C_+,\cK),\]
so, for every \(f \in H^2(\C_+,\cK)\), one has
\[P_+ S_t P_+ = S_t P_+.\]
This yields
\[P_+ S_{-t} = (S_t P_+)^* = (P_+ S_t P_+)^* = (P_+ S_t P_+)^*P_+ = \tilde S_t^* P_+,\]
so
\[H\tilde S_t = P_+ \theta_h P_+ S_t P_+ = P_+ \theta_h S_t P_+ = P_+ S_{-t} \theta_h P_+ = \tilde S_t^* P_+ \theta_h P_+ = \tilde S_t^*H. \qedhere\]
\end{proof}
This lemma motivates the following definition:
\begin{definition}
Let \(\cK\) be a complex Hilbert space and \(H\) be a \(\cK\)-Hankel operator. A function \(h \in L^\infty(\R,B(\cK))\) is called a \textit{symbol} for the Hankel operator \(H\), if \(H = H_h\).
\end{definition}
In this section, we will explicitly construct some symbols for Hankel operators under the assumption that the Hankel operator is positive. Many of the results and proofs in this section are adaptations of those in \cite{ANS22}, where the scalar-valued case \(H^2\left(\C_+\right)\) was considered. In order to extend their results to the vector-valued case, it is necessary to partially redo some of their arguments. We do so replacing scalars by vectors and operators, while carefully handling the additional difficulties that arise in this setting.

\subsection{Positive Hankel operators and Carleson measures}
In this subsection, we will consider positive Hankel operators and see how they are connected to a certain class of measures. Given a complex Hilbert space \(\cK\), in the following, we will identify \(L^2\left(\R,\cK\right)\) with the Hilbert space completation of \(L^2\left(\R,\C\right) \otimes_\C \cK\) and for \(f \in L^2\left(\R,\C\right)\) and \(v \in \cK\) we shortly write \(fv\) for \(f \otimes v \in L^2\left(\R,\cK\right)\). The measures we will consider are the following:
\begin{definition}
Let \(\cK\) be a complex Hilbert space and let
\[\cS(\cK)_+ \coloneqq \{A \in B(\cK): A = A^*, A \geq 0\}.\]
An \(\cS(\cK)_+\)-valued measure \(\mu\) on \(\R_+\) is called a \textit{\(\cK\)-Carleson measure}, if
\[H^2\left(\C_+,\cK\right)^2 \ni (f,g) \mapsto \int_{\R_+} \braket*{f(i\lambda)}{d\mu(\lambda) g(i\lambda)}\]
defines a continuous sesquilinear form on \(H^2\left(\C_+,\cK\right)\). In this case we write \(H_\mu\) for the operator in \(B(H^2\left(\C_+,\cK\right))\) with
\[\braket*{f}{H_\mu g} = \int_{\R_+} \braket*{f(i\lambda)}{d\mu(\lambda) g(i\lambda)} \qquad \forall f,g \in H^2\left(\C_+,\cK\right).\]
\end{definition}
The following proposition relates Carleson measures to positive Hankel operators. It is an operator-valued version of the results provided in \cite[Sec.~3.2]{ANS22}.
\begin{prop}\label{prop:Carleson}
Let \(\cK\) be a complex Hilbert space and \(\mu\) be a \(\cK\)-Carleson measure. Then, the map \(\mu \mapsto H_\mu\) defines a bijection between the set of Carleson measures and the set of positive Hankel operators.
\end{prop}
\begin{proof}
We first show that, in fact, the operators \(H_\mu\) are positive Hankel operators. For every \(\cK\)-Carleson measure \(\mu\) and every \(t \in \R_+\) one has
\begin{align*}
\braket*{f}{H_\mu \tilde S_t g} &= \int_{\R_+} \braket*{f(i\lambda)}{d\mu(\lambda) e^{-t\lambda}g(i\lambda)}
\\&= \int_{\R_+} \braket*{e^{-t\lambda}f(i\lambda)}{d\mu(\lambda) g(i\lambda)} = \braket*{\tilde S_t f}{H_\mu g} \qquad \forall f,g \in H^2\left(\C_+,\cK\right),
\end{align*}
which implies \(H_\mu \tilde S_t = \tilde S_t^* H_\mu\) and therefore \(H_\mu\) is a \(\cK\)-Hankel operator. The positivity follows immediately from the fact that the measure \(\mu\) has values in the positive operators.

We now show that the map \(\mu \mapsto H_\mu\) is surjective. So let \(H \in B\left(H^2\left(\C_+,\cK\right)\right)\) be a positive \(\cK\)-Hankel operator. For every \(f,g \in H^2(\C_+)\) we have
\[|\braket*{fv}{H gw}| \leq \left\lVert H \right\rVert \cdot \left\lVert f \right\rVert \cdot \left\lVert g \right\rVert \cdot \left\lVert v \right\rVert \cdot \left\lVert w \right\rVert,\]
so there exists a bounded operator \(H^{(f,g)} \in B(\cK)\) with
\begin{equation}\label{eq:HfgNorm}
\left\lVert H^{(f,g)} \right\rVert \leq \left\lVert H \right\rVert \cdot \left\lVert f \right\rVert \cdot \left\lVert g \right\rVert
\end{equation}
such that
\[\braket*{fv}{H gw} = \braket*{v}{H^{(f,g)}w} \qquad \forall v,w \in \cK.\]
This equation, together with the Hankel condition \(H\tilde S_t = \tilde S_t^* H\), implies that
\begin{equation}\label{eq:HfgIdent}
H^{(f,\tilde S_t g)} = H^{(\tilde S_t f,g)} \qquad \forall t \in \R_+.
\end{equation}
Now, for every trace class operator \(A \in B_1(\cH)\), by \fref{eq:HfgNorm}, we have
\[\left|\tr\left(AH^{(f,g)}\right)\right| \leq \left\lVert A \right\rVert_1 \cdot \left\lVert H \right\rVert \cdot \left\lVert f \right\rVert \cdot \left\lVert g \right\rVert,\]
so there exists a bounded operator \(H^A \in B(H^2(\C_+))\) such that
\[\tr\left(AH^{(f,g)}\right) = \braket*{f}{H^A g} \qquad \forall f,g \in H^2(\C_+).\]
This, together with \fref{eq:HfgIdent}, implies that \(H^A\) is a \(\C\)-Hankel operator. If further \(A \in B_1(\cH) \cap \cS(\cK)_+\), by construction, we have \(H^A \geq 0\). Then, by \cite[Sec.~3.2]{ANS22}, there exists a measure \(\mu^A\) such that
\[\braket*{f}{H^A g} = \int_{\R_+} \overline{f(i\lambda)} g(i\lambda) d\mu^A(\lambda) \qquad \forall f,g \in H^2(\C_+).\]
Further, by construction, the map
\[B_1(\cH) \cap \cS(\cK)_+ \ni A \mapsto \mu^A\]
is a monoid homomorphism, which, by \cite[Prop.~I.7, Thm.~I.10]{Ne98}, implies that there exists a \(\cS(\cK)_+\)-valued measure \(\mu\) on \(\R_+\) such that, for every \(f,g \in H^2(\C_+)\), one has
\[\tr\left(AH^{(f,g)}\right) = \braket*{f}{H^A g} = \tr\left(A\int_{\R_+} \overline{f(i\lambda)} g(i\lambda) d\mu(\lambda)\right) \qquad \forall A \in B_1(\cH) \cap \cS(\cK)_+,\]
which implies
\[\braket*{fv}{H gw} = \braket*{v}{H^{(f,g)}w} = \braket*{v}{\int_{\R_+} \overline{f(i\lambda)} g(i\lambda) d\mu(\lambda) w} \qquad \forall v,w \in \cK.\]
Therefore, we have
\[\braket*{f}{H g} = \int_{\R_+} \braket*{f(i\lambda)}{d\mu(\lambda) g(i\lambda)} \qquad \forall f,g \in H^2\left(\C_+,\cK\right).\]
For the injectivity, we notice that, given two \(\cK\)-Carleson measures \(\mu,\nu\) such that \(H_\mu = H_\nu\), we have
\[\int_{\R_+} \overline{f(i\lambda)} g(i\lambda) \braket*{v}{d\mu(\lambda)v} = \braket*{fv}{H_\mu gv} = \braket*{fv}{H_\nu gv} = \int_{\R_+} \overline{f(i\lambda)} g(i\lambda) \braket*{v}{d\nu(\lambda)v}\]
for every \(f,g \in H^2\left(\C_+\right)\). This, by \cite[Sec.~3.2]{ANS22}, implies that
\[\braket*{v}{\mu v} = \braket*{v}{\nu v}\]
for every \(v \in \cK\) and therefore \(\mu = \nu\).
\end{proof}
This proposition allows us to define:
\begin{definition}
Let \(\cK\) be a complex Hilbert space. Given a positive \(\cK\)-Hankel operator \(H \in B(H^2\left(\C_+,\cK\right))\), we write \(\mu_H\) for the unique \(\cK\)-Carleson measure with \(H = H_{\mu_H}\), i.e. the map \(H \mapsto \mu_H\) is the inverse of the map \(\mu \mapsto H_\mu\).
\end{definition}

\subsection{Constructing symbols using Carleson measures}
In this subsection, we want to construct symbols for positive Hankel operators from the perspective of Carleson measures. So, given a Carleson measure \(\mu\), we want to construct a symbol for the Hankel operator \(H_\mu\). In this construction, two types of functions will play a central role. The first one are the functions \(Q_\xi\) defined as follows:
\begin{definition}
Let \(\xi \in \C_+\). Then, we define the function
\[Q_\xi: \C_+ \to \C, \quad z \mapsto \frac 1{2\pi} \cdot \frac i{z-\overline{\xi}}.\]
\end{definition}
\begin{remark}
By \fref{prop:H2KernelFunction}, for every \(\xi \in \C_+\), one has \(Q_\xi \in H^2(\C_+)\) satisfying
\[\braket*{Q_\xi}{f} = f(\xi) \qquad \forall f \in H^2(\C_+).\]
\end{remark}
The second type of functions that play a crucial role in the construction of symbols for the Hankel operator \(H_\mu\) are the functions \(\cN_\mu\) defined like this:
\begin{definition}\label{def:kappa}{\rm (cf. \cite[Def.~3.9]{ANS22})}
Let \(\cK\) be a complex Hilbert space and let \(\mu\) be a \(\cK\)-Carleson measure. We define the function \(\cN_\mu: \C \setminus i(-\infty,0] \to B(\cK)\) by
\[\cN_\mu(z) = \frac 1\pi \int_{\R_+} \frac{1}{\lambda-iz} - \frac{\lambda}{1+\lambda^2} \,d\mu(\lambda).\]
(For the existence of the integral, see \fref{rem:NRI}(a).) Further, we set
\[\cR_\mu \coloneqq \Re \,\cN_\mu \coloneqq \frac 12 \left(\cN_\mu + \cN_\mu^*\right) \qquad \text{and} \qquad \cI_\mu \coloneqq \Im \,\cN_\mu \coloneqq \frac 1{2i} \left(\cN_\mu - \cN_\mu^*\right).\]
\end{definition}
\begin{remark}\label{rem:NRI}
\begin{enumerate}[\rm (a)]
\item The integral used for the definition of the function \(\cN_\mu\) does in fact exist. For this, we notice that, for every \(v \in \cK\), one has \(Q_i v \in H^2(\C_+,\cK)\). This implies that, for every \(v,w \in \cK\), one has
\begin{align*}
\braket*{v}{\int_{\R_+} \frac 1{(2\pi)^2} \cdot \frac 1{(1+\lambda)^2} \,d\mu(\lambda) w} &= \braket*{v}{\int_{\R_+} |Q_i(i\lambda)|^2 \,d\mu(\lambda) w} = \braket*{Q_i v}{H_\mu Q_i w} 
\\&\leq \left\lVert H_\mu\right\rVert \cdot \left\lVert Q_i v\right\rVert \cdot \left\lVert Q_i w\right\rVert = \left\lVert H_\mu\right\rVert \cdot \left\lVert Q_i \right\rVert^2 \cdot \left\lVert v\right\rVert \cdot \left\lVert w\right\rVert,
\end{align*}
so the integral
\[\int_{\R_+} \frac 1{(1+\lambda)^2} \,d\mu(\lambda)\]
exists and defines a bounded operator in \(B(\cK)\). Now, for \(z \in \C \setminus i(-\infty,0]\) the function
\[\R_+ \ni \lambda \mapsto \frac{(1+\lambda iz)(1+\lambda)^2}{(\lambda-iz)(1+\lambda^2)}\]
is bounded since the denominator has no zeroes on \(\R_+\) and the limits for \(\lambda \downarrow 0\) and \(\lambda \to \infty\) exist and are \(\frac 1{-iz}\) and \(iz\) respectively. This implies the existence of the integral
\begin{align*}
\int_{\R_+} \frac{(1+\lambda iz)(1+\lambda)^2}{(\lambda-iz)(1+\lambda^2)} \cdot \frac 1{(1+\lambda)^2} \,d\mu(\lambda) &= \int_{\R_+} \frac{1+\lambda iz}{(\lambda-iz)(1+\lambda^2)} \,d\mu(\lambda)
\\&= \int_{\R_+} \frac{1}{\lambda-iz} - \frac{\lambda}{1+\lambda^2} \,d\mu(\lambda)
\end{align*}
and therefore the definition of the function \(\cN_\mu\) is valid.
\item For \(x+iy \in \C \setminus i(-\infty,0]\) one has
\[\cN_\mu(x+iy) = \frac 1\pi\int_{\R_+} \frac{1}{\lambda+y-ix} - \frac{\lambda}{1+\lambda^2} \,d\mu(\lambda) = \frac 1\pi \int_{\R_+} \frac{\lambda+y+ix}{(\lambda+y)^2+x^2} - \frac{\lambda}{1+\lambda^2} \,d\mu(\lambda).\]
This implies
\begin{equation}\label{eq:RForm}
\cR_\mu(x+iy) = \frac 1\pi \int_{\R_+} \frac{\lambda+y}{(\lambda+y)^2+x^2} - \frac{\lambda}{1+\lambda^2} \,d\mu(\lambda)
\end{equation}
and
\begin{equation}\label{eq:IForm}
\cI_\mu(x+iy) = \frac 1\pi \int_{\R_+} \frac{x}{(\lambda+y)^2+x^2} \,d\mu(\lambda).
\end{equation}
\item The function
\[\C_+ \ni z \mapsto \cN(-iz)\]
is an operator-valued version of a so-called Pick-function (cf. \cite[Thm.~3.4]{Si19}).
\end{enumerate}
\end{remark}
Before we can use the functions \(\cR_\mu\) and \(\cI_\mu\) to construct symbols for Hankel operators, we need to investigate some of their basic properties:
\begin{lemma}\label{lem:kappaSymmetry}
Let \(\cK\) be a complex Hilbert space and let \(\mu\) be a \(\cK\)-Carleson measure. Then, for every \(x \in \R^\times\), one has
\[\cR_\mu(-x) = \cR_\mu(x) \qquad \text{and} \qquad \cI_\mu(-x) = -\cI_\mu(x)\]
and
\[0 \leq \sgn(x) \cI_\mu(x).\]
\end{lemma}
\begin{proof}
For \(x \in \R^\times\), by equation \fref{eq:RForm}, we get
\[\cR_\mu(-x) = \frac 1\pi \int_{\R_+} \frac{\lambda}{\lambda^2+x^2} - \frac{\lambda}{1+\lambda^2} \,d\mu(\lambda) = \cR_\mu(x)\]
and by equation \fref{eq:IForm}, we get
\[\cI_\mu(-x) = \frac 1\pi \int_{\R_+} \frac{-x}{\lambda^2+x^2} \,d\mu(\lambda) = -\cI_\mu(x).\]
Further, by equation \fref{eq:IForm}, one has
\[\sgn(x) \cI_\mu(x) = \sgn(x) \frac 1\pi \int_{\R_+} \frac{x}{\lambda^2+x^2} \,d\mu(\lambda) = \frac 1\pi \int_{\R_+} \frac{|x|}{\lambda^2+x^2} \,d\mu(\lambda) \geq 0. \qedhere\]
\end{proof}
The idea now is to use some variation of the function \(\cN_\mu\) as a symbol for the positive Hankel operator \(H_\mu\). Unfortunately, in general, the function \(\cN_\mu\) is not bounded, but the following lemma shows that at least its imaginary part is bounded and its real part does not grow faster than logarithmically:
\begin{lemma}\label{lem:kappaEst}{\rm (cf. \cite[Thm.~4.1]{ANS22})}
Let \(\cK\) be a complex Hilbert space and let \(\mu\) be a \(\cK\)-Carleson measure. Further, let \(\alpha \coloneqq \left\lVert H_\mu\right\rVert\). Then, for every \(z \in \C_+ \cup \R^\times\), one has
\[\left\lVert \cI_\mu(z)\right\rVert \leq 2 \alpha \qquad \text{and} \qquad \left\lVert \cR_\mu(z)\right\rVert \leq \frac 8\pi \alpha |\log|z|| + \alpha.\]
\end{lemma}
\begin{proof}
For every function \(f \in H^2(\C_+)\) and \(v \in \cK\) we have
\[\braket*{v}{\int_{\R_+} \left|f(i\lambda)\right|^2 \,d\mu(\lambda)v} = \braket*{fv}{H_\mu fv} \leq \alpha \left\lVert f\right\rVert^2 \cdot \left\lVert v\right\rVert^2 = \alpha \int_\R \left|f(t)\right|^2 \,dt \cdot \braket*{v}{\textbf{1}v},\]
so
\begin{equation}\label{eq:muLineEstimate}
\int_{\R_+} \left|f(i\lambda)\right|^2 \,d\mu(\lambda) \leq \alpha \int_\R \left|f(t)\right|^2 \,dt \cdot \textbf{1}.
\end{equation}
Therefore, for \(\rho \in \R_+\), one has
\begin{align}\label{eq:rhoPiAlphaEstimate}
\int_{\R_+} \frac{\rho}{(\lambda + \rho)^2} \,d\mu(\lambda) &= \rho \int_{\R_+} \left|\frac{i}{i\lambda+i\rho}\right|^2 \,d\mu(\lambda) = \rho (2\pi)^2 \int_{\R_+} \left|Q_{i\rho}(i\lambda)\right|^2 \,d\mu(\lambda) \notag
\\&\overset{\fref{eq:muLineEstimate}}{\leq} \rho(2\pi)^2\alpha \int_\R \left|Q_{i\rho}(t)\right|^2 \,dt \cdot \textbf{1} = \alpha \int_\R \frac{\rho}{\rho^2+t^2} \,dt \cdot \textbf{1} = \pi\alpha \cdot \textbf{1},
\end{align}
using \fref{lem:Integrals}(a) in the last step. Further, considering the holomorphic square root
\[\sqrt{\cdot}: \C \setminus (-\infty,0] \to \C\]
with \(\sqrt{\R_+} = \R_+\), the function
\[f: \C_+ \to \C, \quad z \mapsto 2\pi \frac{\sqrt{-iz}}{z+i}\]
is bounded and holomorphic, so \(f \in H^\infty(\C_+)\) (\fref{def:HardyScalarDef}), which, by \fref{lem:HardyBasic}(d), implies that
\[f \cdot Q_{i\rho} \in H^\infty(\C_+) \cdot H^2(\C_+) = H^2(\C_+).\]
Therefore
\begin{align}\label{eq:alphaLogEstimate}
&\int_{\R_+} \frac{\lambda|1-\rho^2|}{(\lambda^2+\rho^2)(1+\lambda^2)} \,d\mu(\lambda) \leq 4|1-\rho^2| \int_{\R_+} \frac{\lambda}{(\lambda + \rho)^2(\lambda+1)^2} \,d\mu(\lambda) \notag
\\&\qquad= 4|1-\rho^2| \int_{\R_+} \left|\frac{\sqrt{-i(i\lambda)}}{(i\lambda + i\rho)(i\lambda+i)}\right|^2 \,d\mu(\lambda) = 4|1-\rho^2| (2\pi)^2 \int_{\R_+} \left|f(i\lambda)Q_{i\rho}(i\lambda)\right|^2 \,d\mu(\lambda) \notag
\\&\qquad\leq 4\alpha|1-\rho^2| (2\pi)^2 \int_\R \left|f(t)Q_{i\rho}(t)\right|^2 \,dt \cdot \textbf{1} = 4\alpha|1-\rho^2| \int_\R \left|\frac{\sqrt{-it}}{(t + i\rho)(t+i)}\right|^2 \,dt \cdot \textbf{1} \notag
\\&\qquad= 4\alpha|1-\rho^2| \int_\R \frac{|t|}{(t^2 + \rho^2)(t^2+1)} \,dt \cdot \textbf{1} = 8\alpha \left|\int_{\R_+} \frac{t(\rho^2-1)}{(t^2 + \rho^2)(t^2+1)} \,dt\right| \cdot \textbf{1} \notag
\\&\qquad= 8 \alpha |\log(\rho)|,
\end{align}
using \fref{lem:Integrals}(h) in the last step. For \(x+iy \in \C_+ \cup \R^\times\), using equation \fref{eq:IForm} and writing \(|A| = \sqrt{AA^*}\) for \(A \in B(\cK)\), this yields
\begin{align*}
|\cI_\mu(x+iy)| &= \frac 1\pi \int_{\R_+} \frac{|x|}{(\lambda+y)^2+x^2} \,d\mu(\lambda) \leq \frac 1\pi \int_{\R_+} \frac{|x|}{\lambda^2+x^2} \,d\mu(\lambda)
\\&\leq \frac 2\pi \int_{\R_+} \frac{|x|}{(\lambda + |x|)^2} \,d\mu(\lambda) \overset{\fref{eq:rhoPiAlphaEstimate}}{\leq} 2\alpha \cdot \textbf{1},
\end{align*}
so
\[\left\lVert \cI_\mu(x+iy)\right\rVert \leq 2 \alpha.\]
Further, by equation \fref{eq:RForm}, we have
\begin{align*}
|\cR_\mu(x+iy)| &= \left|\frac 1\pi \int_{\R_+} \frac{\lambda+y}{(\lambda+y)^2+x^2} - \frac{\lambda}{1+\lambda^2} \,d\mu(\lambda)\right|
\\&= \left|\frac 1\pi \int_{\R_+} \frac{-y\lambda^2 + \lambda(1-x^2-y^2) + y}{((\lambda+y)^2+x^2)(1+\lambda^2)} \,d\mu(\lambda)\right|
\\&\leq \frac 1\pi \int_{\R_+} \frac{y(1+\lambda^2) + \lambda|1-x^2-y^2|}{((\lambda+y)^2+x^2)(1+\lambda^2)} \,d\mu(\lambda)
\\&= \frac 1\pi \int_{\R_+} \frac{\lambda|1-x^2-y^2|}{((\lambda+y)^2+x^2)(1+\lambda^2)} \,d\mu(\lambda) + \frac 1\pi \int_{\R_+} \frac{y}{(\lambda+y)^2+x^2} \,d\mu(\lambda)
\\&\leq \frac 1\pi \int_{\R_+} \frac{\lambda|1-x^2-y^2|}{(\lambda^2+y^2+x^2)(1+\lambda^2)} \,d\mu(\lambda) + \frac 1\pi \int_{\R_+} \frac{y}{(\lambda+y)^2} \,d\mu(\lambda)
\\&\overset{\fref{eq:rhoPiAlphaEstimate}}{\leq} \frac 1\pi \int_{\R_+} \frac{\lambda|1-|x+iy|^2|}{(\lambda^2+|x+iy|^2)(1+\lambda^2)} \,d\mu(\lambda) + \alpha \cdot \textbf{1}
\\&\overset{\fref{eq:alphaLogEstimate}}{\leq} \frac 8\pi \alpha \left|\log|x+iy|\right| \cdot \textbf{1} + \alpha \cdot \textbf{1}
\end{align*}
so
\[\left\lVert \cR_\mu(x+iy)\right\rVert \leq \frac 8\pi \alpha |\log|x+iy|| + \alpha. \qedhere\]
\end{proof}
\begin{prop}\label{prop:almostHinfty}
Let \(\cK\) be a complex Hilbert space and let \(\mu\) be a \(\cK\)-Carleson measure. Then, for every \(\xi \in \C_+\) and \(v,w \in \cK\), one has
\[\braket*{v}{\cN_\mu w} \cdot Q_{\xi} \in H^2(\C_+).\]
\end{prop}
\begin{proof}
By \fref{lem:kappaEst}, setting \(\alpha \coloneqq \left\lVert H_\mu\right\rVert\), for every \(z \in \C_+\), we have
\[\left\lVert \cN_\mu(z)\right\rVert \leq \left\lVert \cR_\mu(z)\right\rVert + \left\lVert \cI_\mu(z)\right\rVert \leq \frac 8\pi \alpha |\log|z|| + \alpha + 2\alpha = 3\alpha + \frac 4\pi \alpha |\log(|z|^2)|.\]
Therefore, for \(x+iy \in \C_+\), we get
\begin{align}\label{eq:NNormEstimate}
\left\lVert \cN_\mu(x+iy)\right\rVert^2 &\leq \left(3\alpha + \frac 4\pi \alpha \left|\log\left(x^2+y^2\right)\right|\right)^2 \notag
\\&= (3\alpha)^2 + \frac {24}\pi\alpha^2 \left|\log\left(x^2+y^2\right)\right| + \left(\frac{4\alpha}\pi\right)^2 \left|\log\left(x^2+y^2\right)\right|^2.
\end{align}
Now, for \(x+iy \in \C_+\), if \(x^2+y^2\leq 1\), we have
\[\left|\log\left(x^2+y^2\right)\right| = -\log(x^2+y^2) \leq -\log(x^2) = -2 \log|x|\]
and if \(x^2+y^2\geq 1\) with \(x^2 \geq y^2\), we have
\[\left|\log\left(x^2+y^2\right)\right| = \log(x^2+y^2) \leq \log(2x^2) = \log(2) + 2 \log|x|\]
and if \(x^2+y^2\geq 1\) with \(x^2 \leq y^2\), we have
\begin{align*}
\left|\log\left(x^2+y^2\right)\right| &= \log(x^2+y^2) \leq \log(2y^2) \leq \log\left(\frac{81}{24}y^2\right)
\\&= \log\left(\frac{\left(3\sqrt y\right)^4}{4!}\right) \leq \log\left(\exp\left(3\sqrt y\right)\right) = 3 \sqrt{y}.
\end{align*}
In any case, we get
\[\left|\log\left(x^2+y^2\right)\right| \leq \max\left\{\log(2) + 2 |\log|x||, 3 \sqrt{y}\right\} \leq \log(2) + 2 |\log|x|| + \frac 32 (y+1)\]
and
\[\left|\log\left(x^2+y^2\right)\right|^2 \leq \max\left\{\left(\log(2) + 2 |\log|x||\right)^2, (3 \sqrt{y})^2\right\} \leq (\log(2) + 2 |\log|x||)^2 + 9y.\]
Together with \fref{eq:NNormEstimate} this implies that there are constants \(A,B,C,D \in \R_+\) such that
\begin{align*}
\left\lVert \cN_\mu(x+iy)\right\rVert^2 &\leq A + B |\log|x|| + C |\log|x||^2 + Dy.
\end{align*}
This implies that, for every \(\lambda \in \R_+\), one has
\begin{align*}
&\sup_{y>0} \int_\R \left\lVert \cN_\mu(x+iy)\right\rVert^2 |Q_{i\lambda}(x+iy)|^2 \,dx
\\&\qquad\leq \frac 1{(2\pi)^2} \sup_{y>0} \int_\R \frac{A + B |\log|x|| + C |\log|x||^2 + Dy}{(\lambda+y)^2+x^2} \,dx
\\&\qquad\leq \frac 1{(2\pi)^2} \left[\int_\R \frac{A + B |\log|x|| + C |\log|x||^2}{\lambda^2+x^2} \,dx + D \sup_{y>0} \int_\R \frac{y}{y^2+x^2} \,dx\right]
\\&\qquad\leq \frac 1{(2\pi)^2} \left[\int_\R \frac{A + B |\log|x|| + C |\log|x||^2}{\lambda^2+x^2} \,dx + D\pi\right] < \infty
\end{align*}
and therefore, for \(v,w \in \cK\), one has
\begin{align*}
&\sup_{y>0} \int_\R \left|\braket*{v}{\cN_\mu(x+iy)w} \cdot Q_{i\lambda}(x+iy)\right|^2 \,dx 
\\&\qquad\leq \left\lVert v\right\rVert^2 \cdot \left\lVert w\right\rVert^2 \cdot \sup_{y>0} \int_\R \left\lVert \cN_\mu(x+iy)\right\rVert^2 |Q_{i\lambda}(x+iy)|^2 \,dx < \infty,
\end{align*}
so
\[\braket*{v}{\cN_\mu w} \cdot Q_{i\lambda} \in H^2(\C_+).\]
Now, for any \(\xi \in \C_+\), we have
\[\frac{Q_\xi(z)}{Q_{i\lambda}(z)} = \frac{z+i\lambda}{z-\overline{\xi}}, \qquad
\text{so} \qquad \frac{Q_\xi}{Q_{i\lambda}} \in H^\infty(\C_+).\]
This, by \fref{lem:HardyBasic}(d), implies that
\[\braket*{v}{\cN_\mu w} \cdot Q_{\xi} = \frac{Q_\xi}{Q_{i\lambda}} \cdot \braket*{v}{\cN_\mu w} \cdot Q_{i\lambda} \in H^\infty(\C_+) \cdot H^2(\C_+) = H^2(\C_+). \qedhere\]
\end{proof}
\begin{cor}\label{cor:kadj0}{\rm (cf. \cite[Eq.~(30)]{ANS22})}
Let \(\cK\) be a complex Hilbert space and let \(\mu\) be a \(\cK\)-Carleson measure. Then, for every \(z,\xi \in \C_+\) and \(v,w \in \cK\), one has \(\braket*{Q_z v}{\cN_\mu^*R Q_\xi w} = 0\).
\end{cor}
\begin{proof}
By \fref{lem:kappaSymmetry} one has \(\cN_\mu^*R = R\cN_\mu\), so
\[\braket*{Q_z v}{\cN_\mu^*R Q_\xi w} = \braket*{Q_z v}{R\cN_\mu Q_\xi w} = \braket*{RQ_z}{\braket*{v}{\cN_\mu w} \cdot Q_\xi} = 0,\]
using in the last step that \(RQ_z \in H^2(\C_-)\) and \(\braket*{v}{\cN_\mu w} \cdot Q_\xi \in H^2(\C_+)\) by \fref{prop:almostHinfty}.
\end{proof}
With this preparation, we can prove the following generalization of \cite[Thm.~4.1]{ANS22}:
\begin{thm}\label{thm:SchoberOpValued}
Let \(\cK\) be a complex Hilbert space and let \(\mu\) be a \(\cK\)-Carleson measure. Then \(i \cdot \cI_\mu\) is a symbol for the Hankel operator \(H_\mu\).
\end{thm}
\begin{proof}
For \(z,\xi \in \C_+\) we have
\begin{equation}\label{eq:QReflectionIdentity}
Q_\xi(-\overline{z}) = \frac{1}{2\pi} \frac{i}{-\overline{z}-\overline{\xi}} = \frac{1}{2\pi} \frac{-i}{\overline{z}-(-\overline{\xi})} = \overline{\frac{1}{2\pi} \frac{i}{z-(-\xi)}} = \overline{Q_{-\overline{\xi}}(z)}.
\end{equation}
Further, since \(\overline{Q_{-\overline{\xi}}} \in H^2(\C_-)\), we have
\begin{equation}\label{eq:Q0Identity}
\int_\R \overline{Q_z(x)}\overline{Q_{-\overline{\xi}}(x)} \,dx = \braket*{Q_z}{\overline{Q_{-\overline{\xi}}}} = 0.
\end{equation}
We therefore get
\begin{align*}
\braket*{Q_z}{\cN_\mu R Q_\xi} &= \int_\R \overline{Q_z(x)}Q_\xi(-x) \cdot \frac 1\pi \int_{\R_+} \frac{1}{\lambda-ix} - \frac{\lambda}{1+\lambda^2} \,d\mu(\lambda) \,dx
\\&\overset{\fref{eq:QReflectionIdentity}}{=} \frac 1\pi \int_\R \overline{Q_z(x)}\overline{Q_{-\overline{\xi}}(x)}\int_{\R_+} 2\pi Q_{i\lambda}(x) - \frac{\lambda}{1+\lambda^2} \,d\mu(\lambda) \,dx
\\&= \frac 1\pi \int_{\R_+} \int_\R \overline{Q_z(x)}\overline{Q_{-\overline{\xi}}(x)} \left[2\pi Q_{i\lambda}(x) - \frac{\lambda}{1+\lambda^2}\right] \,dx \,d\mu(\lambda)
\\&\overset{\fref{eq:Q0Identity}}{=} \frac 1\pi \int_{\R_+} \int_\R \overline{Q_z(x)}\overline{Q_{-\overline{\xi}}(x)} \cdot 2\pi Q_{i\lambda}(x) \,dx \,d\mu(\lambda)
\\&= 2 \int_{\R_+} \braket*{Q_zQ_{-\overline{\xi}}}{Q_{i\lambda}} \,d\mu(\lambda) = 2 \int_{\R_+} \overline{Q_z(i\lambda)Q_{-\overline{\xi}}(i\lambda)} \,d\mu(\lambda)
\\&\overset{\fref{eq:QReflectionIdentity}}{=} 2 \int_{\R_+} \overline{Q_z(i\lambda)}Q_\xi(i\lambda) \,d\mu(\lambda) = 2 \braket*{Q_z}{H_\mu Q_\xi}.
\end{align*}
This implies that for \(z,\xi \in \C_+\) and \(v,w \in \cK\), one has
\begin{align*}
\braket*{Q_z v}{(i \cdot \cI_\mu) R Q_\xi w} &= \frac 1{2} \left[\braket*{Q_z v}{\cN_\mu R Q_\xi w} - \braket*{Q_z v}{\cN_\mu^*R Q_\xi w}\right]
\\&= \frac 1{2} \left[2 \braket*{Q_z v}{H_\mu Q_\xi w} - 0\right] = \braket*{Q_z v}{H_\mu Q_\xi w},
\end{align*}
using \fref{cor:kadj0} in the penultimate step. Since
\[\overline{\spann \left\{Q_z v: z \in \C_+, v \in \cK\right\}} = H^2(\C_+,\cK)\]
by \fref{lem:HardyBasic}(a) and both operators \(H_\mu\) and \(P_+ (i \cdot \cI_\mu) R P_+\) are bounded by \fref{lem:kappaEst}, we get that
\[H_\mu = P_+ (i \cdot \cI_\mu) R P_+ = H_{i \cdot \cI_\mu}. \qedhere\]
\end{proof}
\begin{example}{\rm (\cite[Ex.~5.3.7]{Sc23})}\label{ex:2Lebesgue}
Given a complex Hilbert space \(\cK\), we consider the \mbox{\(B(\cK)\)-valued} measure \(\mu\) on \(\R_+\) defined by
\[d\mu(\lambda) = 2 \cdot \,d\lambda \cdot \textbf{1}.\]
That \(\mu\) in fact is a \(\cK\)-Carleson measure, follows with \cite[Thm.~3.7]{ANS22} from
\[\int_0^x \frac{2 \,d\lambda}{1+\lambda^2} \leq \int_0^x 2 \,d\lambda = 2x \qquad \text{and} \qquad \int_{\frac 1x}^\infty \frac{2 \,d\lambda}{1+\lambda^2} \leq \int_{\frac 1x}^\infty \frac{2 \,d\lambda}{\lambda^2} = 2x.\]
We first show that
\[\cN_{\mu}(z) = -\frac 2\pi \log(-iz) \cdot \textbf{1}, \qquad z \in \C \setminus i(-\infty,0],\]
where
\[\log: \C \setminus (-\infty,0] \to \C\]
denotes the main branch of the complex logarithm with \(\log(\R_+) \subeq \R\). For this, we notice that, for \(z \in \C \setminus (-\infty,0]\), we have
\begin{align*}
\int_{\R_+} \frac{1}{\lambda+z} - \frac{1}{\lambda+1} \,d\lambda &= \left[\log(\lambda+z) - \log(\lambda+1)\right]_0^\infty
\\&= \left[\log\left(\frac{\lambda+z}{\lambda+1}\right)\right]_0^\infty = \log(1) - \log(z) = -\log(z).
\end{align*}
Further, we have
\begin{align*}
\int_{\R_+} \frac{1}{\lambda+1} - \frac{\lambda}{1+\lambda^2} \,d\lambda &= \left[\log(\lambda+1) - \frac 12\log(1+\lambda^2)\right]_0^\infty
\\&= \frac 12 \left[\log((\lambda+1)^2) - \log(1+\lambda^2)\right]_0^\infty
\\&= \frac 12 \left[\log\left(\frac{(\lambda+1)^2}{1+\lambda^2}\right)\right]_0^\infty = \frac 12 \left[\log(1) - \log(1)\right] = 0.
\end{align*}
For \(z \in \C \setminus i(-\infty,0]\) this yields
\begin{align*}
\cN_{\mu}(z) &= \frac 2\pi \int_{\R_+} \frac{1}{\lambda+(-iz)} - \frac{\lambda}{1+\lambda^2} \,d\lambda \cdot \textbf{1}
\\&= \frac 2\pi \left[\int_{\R_+} \frac{1}{\lambda+(-iz)} - \frac{1}{\lambda+1} \,d\lambda + \int_{\R_+} \frac{1}{\lambda+1} - \frac{\lambda}{1+\lambda^2} \,d\lambda\right] \cdot \textbf{1}
\\&=\frac 2\pi \left[-\log(-iz) + 0\right] \cdot \textbf{1} = -\frac 2\pi \log(-iz) \cdot \textbf{1}.
\end{align*}
For \(x \in \R^\times\) this yields
\begin{align}\label{eq:NMuExplicit}
\cN_{\mu}(x) &= -\frac 2\pi \log(-ix) \cdot \textbf{1} = -\frac 2\pi \left[\log(-i\sgn(x)) + \log|x|\right] \cdot \textbf{1} \notag
\\&= -\frac 2\pi \left[-\sgn(x) i\frac \pi 2 + \log |x|\right] \cdot \textbf{1} = -\frac 2\pi\log |x| \cdot \textbf{1} + i \cdot \sgn(x) \cdot \textbf{1},
\end{align}
so we get \(\cI_{\mu} = \sgn \cdot \textbf{1}\). Therefore, by \fref{thm:SchoberOpValued} the function
\[i \cdot \cI_{\mu} = i \cdot \sgn \cdot \textbf{1}\]
is a symbol for the Hankel operator \(H_\mu\).
\end{example}

\subsection{Constructing more symbols using projections}\label{subsec:moreSymbols}
We now want to construct a more general class of symbols than the ones in \fref{thm:SchoberOpValued}. For this, we need the following definition:
\begin{definition}\label{def:beta}
Let \(\cK\) be a complex Hilbert space and let \(\mu\) be a \(\cK\)-Carleson measure. Given a projection \(p \in B(\cK)\) and an operator \(C \in B((\textbf{1}-p)\cK,p\cK)\), we define
\[\beta(\mu,p,C): \R^\times \to B(\cK), \quad x \mapsto \begin{pmatrix*} i \,p\cI_\mu(x)p & C + p\cR_\mu(x)(\textbf{1}-p) \\
C^* + (\textbf{1}-p)\cR_\mu(x)p & i \,(\textbf{1}-p)\cI_\mu(x)(\textbf{1}-p)\end{pmatrix*}.\]
\end{definition}
The following theorem shows that as long as the so-defined functions are bounded, they define symbols of the positive Hankel operator \(H_\mu\):
\begin{thm}\label{thm:SchoberWithProjections}
Let \(\cK\) be a complex Hilbert space and let \(\mu\) be a \(\cK\)-Carleson measure. Further let \(p \in B(\cK)\) be a projection and let \(C \in B((\textbf{1}-p)\cK,p\cK)\).
If \(p\cR_\mu(\textbf{1}-p)\) is bounded, then \(\beta(\mu,p,C)\) is a symbol for the Hankel operator \(H_\mu\).
\end{thm}
\begin{proof}
One has
\[\beta(\mu,p,C) - i \cdot \cI_\mu = \begin{pmatrix*} 0 & C + p\cN_\mu^*(\textbf{1}-p) \\
C^* + (\textbf{1}-p)\cN_\mu^*p & 0\end{pmatrix*}.\]
For \(z,\xi \in \C_+\) and \(v,w \in \cK\) this yields
\begin{align*}
&\braket*{Q_z v}{\left[\beta(\mu,p,C) - i \cdot \cI_\mu\right]R Q_\xi w} 
\\&\qquad= \braket*{Q_z pv}{\cN_\mu^*R Q_\xi (\textbf{1}-p)w} + \braket*{Q_z (\textbf{1}-p)v}{\cN_\mu^*R Q_\xi pw} + \braket*{Q_z v}{\begin{pmatrix*} 0 & C \\
C^* & 0\end{pmatrix*}R Q_\xi w}
\\&\qquad= 0 + 0 + 0 = 0,
\end{align*}
where the first two terms vanish by \fref{cor:kadj0} and the last term since
\[Q_z v \in H^2(\C_+,\cK) \qquad \text{and} \qquad \begin{pmatrix*} 0 & C \\
C^* & 0\end{pmatrix*}R Q_\xi w \in H^2(\C_-,\cK).\]
Since
\[\overline{\spann \left\{Q_z v: z \in \C_+, v \in \cK\right\}} = H^2(\C_+,\cK)\]
by \fref{lem:HardyBasic}(a), this implies that
\[P_+ \left[\beta(\mu,p,C) - i \cdot \cI_\mu\right]R P_+ = 0,\]
which is equivalent to
\[P_+ \beta(\mu,p,C) R P_+ = P_+ (i \cdot \cI_\mu) R P_+ = H_\mu,\]
using \fref{thm:SchoberOpValued} in the last step.
\end{proof}
\begin{remark}\label{rem:p01}
Choosing \(p = \textbf{1}\) or \(p = 0\), we get
\[\beta(\mu,\textbf{1},0) = \beta(\mu,0,0) = i \cdot \cI_\mu,\]
so \fref{thm:SchoberWithProjections} is a generalization of \fref{thm:SchoberOpValued}.
\end{remark}
\begin{example}\label{ex:2LebesgueBeta}
For the \(\cK\)-Carleson measure \(\mu\) on \(\R_+\) with \(d\mu(\lambda) = 2 \cdot \,d\lambda \cdot \textbf{1}\) (cf. \fref{ex:2Lebesgue}), by Equation \fref{eq:NMuExplicit}, we have
\[\cI_\mu(x) = \sgn(x) \cdot \textbf{1} \qquad \text{and} \qquad \cR_\mu(x) = -\frac 2\pi \log |x| \cdot \textbf{1}.\]
Thus, for every projection \(p \in B(\cK)\), we have \(p\cR_\mu (1-p) =0\) and therefore, for every \({C \in B((\textbf{1}-p)\cK,p\cK)}\), we have
\[\beta(\mu,p,C)= i \cdot \sgn \cdot \textbf{1}+ {\begin{pmatrix*} 0 & C \\ C^* & 0 \end{pmatrix*}}.\]
This example shows that, for \(p\cR_\mu (1-p)\) to be bounded, it is not necessary that \(\cR_\mu\) is bounded.
\end{example}
\newpage
The symbols \(\beta(\mu,p,C)\) are not just any symbols, but, by the following theorem, they are precisely the symbols \(h\) satisfying \(h(-\cdot) = h^* = -(2p-\textbf{1})h(2p-\textbf{1})\).
\begin{thm}\label{thm:SchoberProjUnique}
Let \(\cK\) be a complex Hilbert space, let \(\mu\) be a \(\cK\)-Carleson measure and let \(h \in L^\infty(\R,B(\cK))\) be a symbol for the Hankel operator \(H_\mu\). Then, for any projection \(p \in B(\cK)\), the following are equivalent:
\begin{enumerate}[\rm (a)]
\item \(h(-x) = h(x)^* = -(2p-\textbf{1})h(x)(2p-\textbf{1})\) for almost all \(x \in \R\).
\item There exists \(C \in B((\textbf{1}-p)\cK,p\cK)\) such that \(h = \beta(\mu,p,C)\).
\end{enumerate}
\end{thm}
\begin{proof}
\begin{itemize}
\item[\rm (a) \(\Rightarrow\) (b):] Writing
\[h = \begin{pmatrix*} a & b \\ c & d\end{pmatrix*} \qquad \text{and} \qquad 2p-\textbf{1} = \begin{pmatrix*} \textbf{1} & 0 \\ 0 & -\textbf{1}\end{pmatrix*}\]
for the decomposition \(\cK = p\cK \oplus (\textbf{1}-p)\cK\), we have
\[(2p-\textbf{1})h(2p-\textbf{1}) = \begin{pmatrix*} \textbf{1} & 0 \\ 0 & -\textbf{1}\end{pmatrix*} \begin{pmatrix*} a & b \\ c & d\end{pmatrix*} \begin{pmatrix*} \textbf{1} & 0 \\ 0 & -\textbf{1}\end{pmatrix*} = \begin{pmatrix*} a & -b \\ -c & d\end{pmatrix*},\]
so
\[\begin{pmatrix*} a^* & c^* \\ b^* & d^*\end{pmatrix*} = h^* = - (2p-\textbf{1})h(2p-\textbf{1}) = \begin{pmatrix*} -a & b \\ c & -d\end{pmatrix*}\]
and therefore \(a=-a^*\), \(c=b^*\) and \(d=-d^*\). This yields that
\[h = \begin{pmatrix*} a & b \\ b^* & d\end{pmatrix*}\]
with
\begin{equation}\label{eq:adAntiSymmetry}
a = -a^* \qquad \text{and} \qquad d = -d^*.
\end{equation}

Now, by \fref{thm:SchoberOpValued}, both functions \(a = php\) and \(i \cdot p\cI_{\mu}p\) are symbols for the positive Hankel operator \(pH_\mu p\), so, by \cite[Cor.~4.8]{Pa88}, we have
\[a-i \cdot p\cI_{\mu}p \in H^\infty(\C_-,B(p\cK))\]
(see \fref{def:HardyVectorDef}). This yields
\[a-i \cdot p\cI_{\mu}p \overset{\fref{eq:adAntiSymmetry}}{=} -\left(a-i \cdot p\cI_{\mu}p\right)^* \in H^\infty(\C_+,B(p\cK)).\]
By \fref{lem:HardyBasic}(c), this implies that there exists an operator \(K \in B(p\cK)\) with \({K=-K^*}\) such that
\[a-i \cdot p\cI_{\mu}p \equiv K.\]
Further, for almost every \(x \in \R\), one has
\[a(-x) = ph(-x)p = ph^*(x)p = a^*(x) \overset{\fref{eq:adAntiSymmetry}}{=} -a(x)\]
and
\[\cI_{\mu}(-x) = -\cI_{\mu}(x)\]
by \fref{lem:kappaSymmetry}. This yields
\[K = a(-x)-i \cdot p\cI_{\mu}(-x)p = -\left(a(x)-i \cdot p\cI_{\mu}(x)p\right) = -K,\]
so \(K=0\) and therefore
\[a = i \cdot p\cI_{\mu}p.\]
By the same argument, one sees that
\[d = i \cdot (\textbf{1}-p)\cI_{\mu}(\textbf{1}-p).\]
Now, by \fref{thm:SchoberOpValued}, both functions
\[h = \begin{pmatrix*} i \cdot p\cI_{\mu}p & b \\ b^* & i \cdot (\textbf{1}-p)\cI_{\mu}(\textbf{1}-p)\end{pmatrix*}\]
and
\[i \cdot \cI_{\mu} = \begin{pmatrix*} i \cdot p\cI_{\mu}p & i \cdot p\cI_{\mu}(\textbf{1}-p) \\ i \cdot (\textbf{1}-p)\cI_{\mu}p & i \cdot (\textbf{1}-p)\cI_{\mu}(\textbf{1}-p)\end{pmatrix*}\]
are symbols for the positive Hankel operator \(H_h\). So, by \cite[Cor.~4.8]{Pa88}, we have
\[H^\infty(\C_-,B(\cK)) \ni h - i \cdot \cI_{\mu} = \begin{pmatrix*} 0 & b - i \cdot p\cI_{\mu}(\textbf{1}-p) \\ b^* - i \cdot (\textbf{1}-p)\cI_{\mu}p & 0 \end{pmatrix*}\]
and therefore
\[(h - i \cdot \cI_{\mu})^* \in H^\infty(\C_+,B(\cK)).\]
This implies that the function
\begin{align*}
f &\coloneqq (h - i \cdot \cI_{\mu})^* - p \cN_{\mu}(\textbf{1}-p) - (\textbf{1}-p)\cN_{\mu}p
\\&\,= \begin{pmatrix*} 0 & b - p\cR_{\mu}(\textbf{1}-p) \\ b^* - (\textbf{1}-p)\cR_{\mu}p & 0 \end{pmatrix*}
\end{align*}
is holomorphic on \(\C_+\). Now, by
\[b(-x) = ph(-x)(\textbf{1}-p) = ph(x)^*(\textbf{1}-p) = b(x)\]
and \fref{lem:kappaSymmetry}, we have
\[f(-x) = f(x)\]
for almost every \(x \in \R\). Further, for every \(v,w \in \cK\) and every \(\lambda \in \R_+\), we have
\begin{align*}
\braket*{v}{fw} \cdot Q_{i\lambda} &= \braket*{v}{(h - i \cdot \cI_{\mu})^*w} \cdot Q_{i\lambda} - \braket*{pv}{\cN_{\mu}(\textbf{1}-p)w} \cdot Q_{i\lambda} - \braket*{(\textbf{1}-p)v}{\cN_{\mu}pw} \cdot Q_{i\lambda}
\\&\in H^\infty(\C_+) \cdot H^2(\C_+) + H^2(\C_+) + H^2(\C_+) = H^2(\C_+),
\end{align*}
using \fref{prop:almostHinfty} for the last two terms.
Therefore, by \fref{cor:SymConstOp}, the function \(f\) is constant, so there exists an operator \(C \in B((\textbf{1}-p)\cK,p\cK)\) such that
\[b - p\cR_{\mu}(\textbf{1}-p) = C,\]
which is equivalent to
\[b = C + p\cR_{\mu}(\textbf{1}-p).\]
We then have
\[h = \begin{pmatrix*} i \cdot p\cI_{\mu}p & C + p\cR_{\mu}(\textbf{1}-p) \\ C^* + (\textbf{1}-p)\cR_{\mu}p & i \cdot (\textbf{1}-p)\cI_{\mu}(\textbf{1}-p)\end{pmatrix*} = \beta(\mu,p,C).\]
\item[\rm (b) \(\Rightarrow\) (a):]
One has
\begin{align*}
&-(2p-\textbf{1})\beta(\mu,p,C)(x)(2p-\textbf{1})
\\&\qquad= -\begin{pmatrix*} \textbf{1} & 0 \\ 0 & -\textbf{1}\end{pmatrix*} \begin{pmatrix*} i \cdot p\cI_{\mu}(x)p & C + p\cR_{\mu}(x)(\textbf{1}-p) \\ C^* + (\textbf{1}-p)\cR_{\mu}(x)p & i \cdot (\textbf{1}-p)\cI_{\mu}(x)(\textbf{1}-p)\end{pmatrix*} \begin{pmatrix*} \textbf{1} & 0 \\ 0 & -\textbf{1}\end{pmatrix*}
\\&\qquad= \begin{pmatrix*} -i \cdot p\cI_{\mu}(x)p & C + p\cR_{\mu}(x)(\textbf{1}-p) \\ C^* + (\textbf{1}-p)\cR_{\mu}(x)p & -i \cdot (\textbf{1}-p)\cI_{\mu}(x)(\textbf{1}-p)\end{pmatrix*}
\\&\qquad= \beta(\mu,p,C)(x)^* = R\beta(\mu,p,C)(-x)
\end{align*}
using \fref{lem:kappaSymmetry} for the last equality. \qedhere
\end{itemize}
\end{proof}

\newpage
\section{Classifying standard and Borchers-type quadruples}\label{sec:Classification}
In this section, we will link the results of the previous sections to standard subspaces. Throughout this section, given a complex Hilbert space \(\cH\), by \(\cH^\R\) we will denote the set \(\cH\) interpreted as a real Hilbert space endowed with the real scalar product
\[\braket*{\cdot}{\cdot}_{\cH^\R} \coloneqq \Re \, \braket*{\cdot}{\cdot}_\cH.\]
We will then identify quadruples of the form \((\cH^\R,\sV,U,J_\sV)\) among reflection positive orthogonal one-parameter groups under the additional assumption that \((\cH^\R,\sV,U)\) is outgoing. We will see that these quadruples are the ones equivalent to
\[\left(L^2(\R,\cM_\C)^\sharp,H^2(\C_+,\cM_\C)^\sharp,S,\theta_{\beta(\mu,p,C)}\right)\]
for some real Hilbert space \(\cM\), some \(\cM_\C\)-Carleson measure \(\mu\), some projection \(p \in B(\cM_\C)^\sharp\) and some operator \(C \in B((\textbf{1}-p)\cM_\C,p\cM_\C)\) (cf. \fref{def:beta}, \fref{thm:apperingClassi}). We will further see that the quadruples arising from Borchers' Theorem (cf. \fref{thm:Borchers}, \fref{def:BorchersDef}) are the ones corresponding to the case \(d\mu(\lambda) = 2 \cdot d\lambda \cdot \textbf{1}\) and \(C = 0\) (cf. \fref{thm:BorchersClassi}, \fref{ex:2LebesgueBeta}).

\subsection{Standard quadruples}\label{subsec:standardQuad}
To link standard subspaces to the results of \fref{sec:RPOG}, we start with the following lemma:
\begin{lem}\label{lem:standIsReflPos}
Let \(\cH\) be a complex Hilbert space and let \(\sV \subeq \cH\) be a standard subspace. Further let \((U_t)_{t \in \R}\) be a strongly continuous unitary one-parameter group satisfying
\begin{equation}\label{eq:UConditions}
J_\sV U_t = U_{-t} J_\sV \quad \forall t \in \R \qquad \text{and} \qquad U_t \sV \subeq \sV \quad \forall t \in \R_+.
\end{equation}
Then the quadruple \((\cH^\R,\sV,U,J_\sV)\) is a reflection positive orthogonal one-parameter group.
\end{lem}
\begin{proof}
For every \(v \in \sV\) one has
\[J_\sV v = J_\sV T_\sV v = \Delta_\sV^{\frac 12} v\]
and therefore
\[\braket*{v}{J_\sV v}_\cH = \braket{v}{\Delta_\sV^{\frac 12} v}_\cH \geq 0,\]
which implies that also
\[\braket*{v}{J_\sV v}_{\cH^\R} = \Re \, \braket*{v}{J_\sV v}_\cH = \braket*{v}{J_\sV v}_\cH \geq 0.\]
This implies that \((\cH^\R,\sV,J_\sV)\) is a real reflection positive Hilbert space. That \((\cH^\R,\sV,U,J_\sV)\) is a reflection positive orthogonal one-parameter group then follows immediately from \fref{eq:UConditions}.
\end{proof}
This lemma shows that every quadruple of the form \((\cH^\R,\sV,U,J_\sV)\) satisfying \fref{eq:UConditions} is a reflection positive orthogonal one-parameter group. The question we want to answer in this subsection is which reflection positive orthogonal one-parameter groups are of this type. For this, we define:
\begin{definition}
Let \((\cE,\cE_+,U,\theta)\) be a reflection positive orthogonal one-parameter group. We say that \((\cE,\cE_+,U,\theta)\) is \textit{standard}, if there exists a complex Hilbert space \(\cH\), a standard subspace \(\sV \subeq \cH\) and a unitary one-parameter group \((\tilde U_t)_{t \in \R} \subeq \U(\cH)\) such that
\[(\cE,\cE_+,U,\theta) \cong (\cH^\R,\sV,\tilde U,J_\sV),\]
i.e. if there exists an orthogonal map \(\psi: \cE \to \cH^\R\) satisfying
\[\psi(\cE_+) = \sV, \qquad \psi \circ \theta = J_\sV \circ \psi \qquad \text{and} \qquad \psi \circ U_t = \tilde U_t \circ \psi \quad \forall t \in \R.\]
\end{definition}
We will answer the question which reflection positive orthogonal one-parameter groups are standard under the additional assumption of them being outgoing. By \fref{thm:refPos}, we know that all outgoing reflection positive orthogonal one-parameter groups are -- up to orthogonal equivalence -- precisely the ones of the form
\[\left(L^2(\R,\cM_\C)^\sharp,H^2(\C_+,\cM_\C)^\sharp,S,\theta_h\right)\]
for some real Hilbert space \(\cM\) and some \(h \in L^\infty\left(\R,\U(\cM_\C)\right)^{\sharp,\flat}\) with \(H_h \geq 0\) (\fref{def:Hh}). Therefore our question boils down to classifying those functions \(h \in L^\infty\left(\R,\U(\cM_\C)\right)^{\sharp,\flat}\) for which the above quadruple is standard. For this, we need the following definition:
\begin{definition}
Let \(\cM\) be a real Hilbert space. We set
\[B(\cM_\C)^\sharp \coloneqq \{A \in B(\cM_\C): \cC_\cM A \cC_\cM = A\} \qquad \text{and} \qquad \U(\cM_\C)^\sharp \coloneqq \U(\cM_\C) \cap B(\cM_\C)^\sharp\]
(cf. \fref{def:complexification}).
\end{definition}
One difficulty in classifying outgoing standard quadruples is that, by passing from \(\cH\) to \(\cH^\R\), the complex structure on \(\cH\) gets lost. If
\[(\cH^\R,\sV,\tilde U,J_\sV) \cong \left(L^2(\R,\cM_\C)^\sharp,H^2(\C_+,\cM_\C)^\sharp,S,\theta_h\right),\]
there exists a complex structure \(\cI\) on \(L^2(\R,\cM_\C)^\sharp\) such that \(H^2(\C_+,\cM_\C)^\sharp\) is a standard subspace in the complex Hilbert space \((L^2(\R,\cM_\C)^\sharp,\cI)\) and such that the operators \(S\) are complex linear with respect to \(\cI\). The following proposition deals with such complex structures:
\begin{prop}\label{prop:ComplexStrucImpl}
Let \(\cM\) be a real Hilbert space and let \(h \in L^\infty\left(\R,\U(\cM_\C)\right)^{\sharp,\flat}\). Further let \(\cI\) be a complex structure on \(L^2(\R,\cM_\C)^\sharp\) such that the following assertions hold:
\begin{enumerate}[\rm (i)]
\item The real Hilbert space \(\sV \coloneqq H^2(\C_+,\cM_\C)^\sharp\) is a standard subspace in the complex Hilbert space \(\cH \coloneqq (L^2(\R,\cM_\C)^\sharp,\cI)\) with \(J_\sV = \theta_h\).
\item For every \(t \in \R\) one has \(S_t \cI = \cI S_t\).
\end{enumerate}
Then there exists a function \(I \in L^\infty(\R,\U(\cM_\C))^\sharp\) and a projection \(p \in B(\cM_\C)^\sharp\) such that:
\begin{enumerate}[\rm (a)]
\item \(\cI = M_I\).
\item \(H_h > 0\) {\rm (see \fref{def:Hh})}.
\item \(I = h(2p-\textbf{1})\).
\end{enumerate}
\end{prop}
\begin{proof}
\begin{enumerate}[\rm (a)]
\item By assumption, we have
\[S_t \cI = \cI S_t \qquad \forall t \in \R,\]
which, by \fref{prop:SCommutant}, is equivalent to \(\cI = M_I\) for some \(I \in L^\infty(\R,B(\cM_\C))\). Since \(\cI\) is a complex structure on \(L^2(\R,\cM_\C)^\sharp\), we even have \(I \in L^\infty(\R,\U(\cM_\C))^\sharp\).
\item For every \(f \in \sV = H^2(\C_+,\cM_\C)^\sharp\) one has
\[\braket*{f}{H_h f}_{H^2} = \braket*{f}{\theta_h f}_{L^2} = \braket*{f}{J_\sV f}_{L^2} = \braket{f}{\Delta_\sV^{\frac 12} f}_{L^2} = \Re \,\braket{f}{\Delta_\sV^{\frac 12} f}_{\cH} = \braket{f}{\Delta_\sV^{\frac 12} f}_{\cH} \geq 0.\]
If
\[0 = \braket{f}{\Delta_\sV^{\frac 12} f}_\cH = \braket{\Delta_\sV^{\frac 14}f}{\Delta_\sV^{\frac 14} f}_\cH,\]
then \(\Delta_\sV^{\frac 14} f = 0\) and therefore
\[f = T_\sV f = J_\sV \Delta_\sV^{\frac 12} f = J_\sV \Delta_\sV^{\frac 14} \big(\Delta_\sV^{\frac 14} f\big) = 0.\]
We therefore have \(\braket*{f}{H_h f}_{H^2} > 0\) for every \(f \in H^2(\C_+,\cM_\C)^\sharp\). Since
\[\left(H^2(\C_+,\cM_\C)^\sharp\right)_\C \cong H^2(\C_+,\cM_\C)\]
this implies that \(H_h > 0\).
\item By \fref{prop:standardComplement}, for any standard subspace \(\sV \subeq \cH\), we have
\[J_{\sV}\sV = \sV' = i\sV^{\perp_\R},\]
where
\[\sV' \coloneqq \{v \in \cH : \left(\forall w \in \sV\right) \,\Im\braket*{v}{w} = 0\}.\]
Applying this to the standard subspace \(\sV = H^2(\C_+,\cM_\C)^\sharp\) in the complex Hilbert space \(\left(L^2(\R,\cM_\C)^\sharp,\cI\right)\), we get
\[M_I H^2(\C_-,\cM_\C)^\sharp = \cI\sV^{\perp_\R} = J_\sV \sV = \theta_h H^2(\C_+,\cM_\C)^\sharp = M_h H^2(\C_-,\cM_\C)^\sharp.\]
This, by \fref{lem:HardyBasic}(e), implies that \(I = hu\) for some \(u \in \U(\cM_\C)\). The fact that \(I,h \in L^\infty\left(\R,\U(\cM_\C)\right)^\sharp\) then implies that
\[\cC_\cM u \cC_\cM = u,\]
so \(u \in \U(\cM_\C)^\sharp\). Now, for \(x \in \R \otimes_\R \cM\), one has \(\cC_\cM x = x\) and therefore
\begin{align*}
\braket*{x}{u x} = \braket*{\cC_\cM x}{u \cC_\cM x} = \overline{\braket*{x}{\cC_\cM u \cC_\cM x}} = \overline{\braket*{x}{u x}} = \braket*{u x}{x} = \braket*{x}{u^* x}.
\end{align*}
Therefore, for \(z \in \C\) and \(v \in \cM\), we have
\begin{align*}
\braket*{z \otimes v}{u(z \otimes v)} = |z|^2 \braket*{1 \otimes v}{u(1 \otimes v)} = |z|^2 \braket*{1 \otimes v}{u^*(1 \otimes v)} = \braket*{z \otimes v}{u^*(z \otimes v)},
\end{align*}
so
\[\braket*{\xi}{u \xi} = \braket*{\xi}{u^* \xi}\]
for every \(\xi \in \cM_\C\). Since, by the polarization identity, an operator is uniquely determined by its values on the diagonal, this yields \(u = u^*\). This, together with the unitarity of \(u\) implies that the operator
\[p \coloneqq \frac 12(\textbf{1}+u) \in B(\cM_\C)^\sharp\]
is a projection and we have \(u = (2p - \textbf{1})\), so
\[I = hu = h(2p - \textbf{1}). \qedhere\]
\end{enumerate}
\end{proof}
With this preparation, we can prove our theorem classifying standard quadruples:
\begin{thm}{\rm \textbf{(Classification of outgoing standard quadruples)}}\label{thm:apperingClassi}
Let \(\cM\) be a real Hilbert space and let \(h \in L^\infty\left(\R,\U(\cM_\C)\right)^{\sharp,\flat}\). Then the following are equivalent:
\begin{enumerate}[\rm (a)]
\item The quadruple \(\left(L^2(\R,\cM_\C)^\sharp,H^2(\C_+,\cM_\C)^\sharp,S,\theta_h\right)\) is standard.
\item There exists a function \(I \in L^\infty(\R,\U(\cM_\C))^\sharp\) with \(I^2 = -\textbf{1}\) such that \({\sV \coloneqq H^2(\C_+,\cM_\C)^\sharp}\) is a standard subspace in the complex Hilbert space \(\cH \coloneqq (L^2(\R,\cM_\C)^\sharp,M_I)\) with \(J_\sV = \theta_h\). 
\item \(H_h > 0\) and there exists a projection \(p \in B(\cM_\C)^\sharp\) such that \(h^* = -(2p-\textbf{1})h(2p-\textbf{1})\).
\item \(H_h > 0\) and there exists a projection \(p \in B(\cM_\C)^\sharp\) and an operator \(C \in B((\textbf{1}-p)\cM_\C,p\cM_\C)\) such that \(h = \beta(\mu_{H_h},p,C)\).
\end{enumerate} 
\end{thm}
\begin{proof}
\begin{itemize}
\item[(a) \(\Leftrightarrow\) (b):] 
By definition, \(\left(L^2(\R,\cM_\C)^\sharp,H^2(\C_+,\cM_\C)^\sharp,S,\theta_h\right)\) is standard, if and only if there exists a complex structure \(\cI\) on \(L^2(\R,\cM_\C)^\sharp\) such that \({\sV \coloneqq H^2(\C_+,\cM_\C)^\sharp}\) is a standard subspace in the complex Hilbert space \((L^2(\R,\cM_\C)^\sharp,\cI)\) with \(J_\sV = \theta_h\) and such that, for every \(t \in \R\), the operator \(S_t\) is unitary. In this case, by \fref{prop:ComplexStrucImpl}(a), there exists a function \(I \in L^\infty(\R,\U(\cM_\C))^\sharp\) such that \(\cI = M_I\).
\item[(b) \(\Rightarrow\) (c):] By \fref{prop:ComplexStrucImpl}(b), we have
\(H_h > 0\). By \fref{prop:ComplexStrucImpl}(c) there exists a projection \(p \in B(\cM_\C)^\sharp\) such that \(I = h(2p-\textbf{1})\). This yields
\[-(2p-\textbf{1})h(2p-\textbf{1}) = -h^*h(2p-\textbf{1})h(2p-\textbf{1}) = -h^*II = h^*.\]
\item [(c) \(\Rightarrow\) (b):] We set
\[u \coloneqq 2p-\textbf{1} \in \U(\cM_\C)^\sharp \qquad \text{and} \qquad I \coloneqq hu \in L^\infty(\R,\U(\cM_\C))^\sharp.\]
Then
\[I^2 = huhu = h(2p-\textbf{1})h(2p-\textbf{1}) = -hh^* = -\textbf{1},\]
so \(M_I\) is a complex structure on \(L^2(\R,\cM_\C)^\sharp\).

In the following \(\braket*{\cdot}{\cdot}_\R\) denotes the real scalar product on the real Hilbert space \(L^2(\R,\cM_\C)^\sharp\) and \(\braket*{\cdot}{\cdot}_\C\) the complex scalar product on the complex Hilbert space \(\cH = (L^2(\R,\cM_\C)^\sharp,M_I)\) defined by
\[\braket*{f}{g}_\C \coloneqq \braket*{f}{g}_\R - i \braket*{f}{M_Ig}_\R, \qquad f,g \in L^2(\R,\cM_\C)^\sharp.\]
We want to show that \({\sV = H^2(\C_+,\cM_\C)^\sharp}\) is a standard subspace in the complex Hilbert space \(\cH\). So let \(f \in \sV \cap M_I\sV\). Then \(f = Ig\) for some \(g \in \sV\). This yields
\[\braket*{f}{H_h f}_{H^2} = \braket*{f}{\theta_h f}_\R = \braket*{Ig}{Iu Rf}_\R = \braket*{g}{uRf}_\R = 0,\]
where the last equality follows from \(g \in \sV = H^2(\C_+,\cM_\C)^\sharp\) and
\[uRf \in H^2(\C_-,\cM_\C)^\sharp = \left(H^2(\C_+,\cM_\C)^\sharp\right)^{\perp_\R}.\]
Since \(H_h > 0\), this implies \(f = 0\), so
\begin{equation}\label{eq:TrivialIntersection}
\sV \cap I\sV = \{0\}.
\end{equation}
Now let \(f \in (\sV+M_I\sV)^{\perp_\C} \subeq (\sV+M_I\sV)^{\perp_\R}\). Using that
\[\sV^{\perp_\R} = \left(H^2(\C_+,\cM_\C)^\sharp\right)^{\perp_\R} = H^2(\C_-,\cM_\C)^\sharp = R H^2(\C_+,\cM_\C)^\sharp = R \sV,\]
this yields
\[f \in (\sV+M_I\sV)^{\perp_\R} = \sV^{\perp_\R} \cap M_I\sV^{\perp_\R} = R \sV \cap M_I R \sV.\]
Now \(h = h^\flat\) implies \(M_hR = RM_{h^*}\), so
\[M_IR = M_huR = RM_{h^*}u = RM_{(Iu)^*}u = RuM_{I^*}u = -RuM_Iu.\]
Therefore, using that \(u\sV = \sV\), we have
\[f \in R \sV \cap M_I R \sV = R \sV \cap RuM_Iu \sV = Ru(u\sV \cap M_I u\sV) = Ru(\sV \cap M_I \sV) \overset{\fref{eq:TrivialIntersection}}{=} Ru\{0\} = \{0\},\]
i.e. \((\sV+M_I\sV)^{\perp_\C} = \{0\}\), which implies \(\overline{\sV + I\sV} = \cH\). Therefore \(\sV\) is a standard subspace in the complex Hilbert space \(\cH\).

Finally, we show that \(J_\sV = \theta_h\). We first notice that \(\theta_h\) is an anti-unitary involution on \(\cH\) since \(\theta_h\) is real linear and
\begin{equation*}
\theta_h M_I = \theta_h M_h u = \theta_h \theta_h Ru = Ru = uR = -M_I M_I uR = -M_I\theta_h.
\end{equation*}
Now, for the subspace \(\sV' = M_I\sV^{\perp_\R}\) (cf. \fref{prop:standardComplement}), we have
\begin{align}\label{eq:thetaHRelation}
\sV' &= M_I\sV^{\perp_\R} = M_I \left(H^2(\C_+,\cM_\C)^\sharp\right)^{\perp_\R} \notag
\\&= M_I R H^2(\C_+,\cM_\C)^\sharp = M_h u R \sV = (M_h R) (u\sV) = \theta_h \sV.
\end{align}
Further, for \(f \in \sV = H^2(\C_+,\cM_\C)^\sharp\), we have
\[\braket*{f}{M_I\theta_h f}_\R = \braket*{f}{IhR f}_\R = \braket*{f}{IIuR f}_\R = -\braket*{f}{uRf}_\R = 0,\]
using that \(uRf \in H^2(\C_-,\cM_\C)^\sharp\).
This yields
\begin{equation}\label{eq:thetaHPositive}
\braket*{f}{\theta_h f}_\C = \braket*{f}{\theta_h f}_\R - i\braket*{f}{M_I\theta_h f}_\R = \braket*{f}{\theta_h f}_\R = \braket*{f}{H_h f}_{H^2} \geq 0.
\end{equation}
Since by \cite[Prop.~2.1.9]{Lo08} \(J_\sV\) is the unique anti-unitary involution satisfying \fref{eq:thetaHRelation} and \fref{eq:thetaHPositive}, we therefore have \(\theta_h = J_\sV\).
\item[(c) \(\Leftrightarrow\) (d):] This follows immediately from \fref{thm:SchoberProjUnique}. \qedhere
\end{itemize}
\end{proof}

\subsection{Borchers-type quadruples}\label{subsec:BorchersQuad}
In this subsection, we will consider a class of examples of outgoing standard quadruples and see how they fit into the context of \fref{thm:apperingClassi}. The class we will consider are the unitary one-parameter groups appearing in Borchers' Theorem, which states the following:
\begin{thm}\label{thm:Borchers}{\rm \textbf{(Borchers' Theorem, one-particle)}\,(cf. \cite[Thm.~II.9]{Bo92} and \cite[Thm.~2.2.1]{Lo08})}
Let \(\cH\) be a complex Hilbert space and \(\sV \subeq \cH\) be a standard subspace. Further let \({(U_t)_{t \in \R} \subeq \U(\cH)}\) be a strongly continuous unitary one-parameter group such that
\[U_t \sV \subeq \sV \qquad \forall t \in \R_+.\]
If \(\pm \partial U > 0\) then the following commutation relations hold:
\[\Delta_\sV^{is} U_t \Delta_\sV^{-is} = U_{e^{\mp 2\pi s} t}, \qquad J_\sV U_t J_\sV = U_{-t} \qquad \forall s,t \in \R.\]
\end{thm}
We start by showing that the one-parameter groups described in Borchers' Theorem give rise to outgoing standard quadruples:
\begin{prop}\label{prop:BorchersOutgoing}{\rm (\cite[Cor.~6.3.8]{Sc23})}
Let \(\cH\) be a complex Hilbert space and \(\sV \subeq \cH\) be a standard subspace. Further let \({(U_t)_{t \in \R} \subeq \U(\cH)}\) be a strongly continuous unitary one-parameter group such that
\[U_t \sV \subeq \sV \qquad \forall t \in \R_+.\]
If \(\pm \partial U > 0\), then \((\cH^\R,\sV,U,J_\sV)\) is an outgoing standard quadruple.
\end{prop}
\begin{proof}
By \fref{lem:standIsReflPos}, the quadruple \((\cH^\R,\sV,U,J_\sV)\) is a reflection positive orthogonal one-parameter group, which by definition is standard. To show that \((\cH^\R,\sV,U)\) is outgoing, we first assume that \(\partial U > 0\). Then, by \fref{prop:standardPairNormalForm}, there exists a real Hilbert space \(\cM\) and a unitary operator
\[\psi: \cH \to (L^2(\R,\cM_\C)^\sharp,M_{i \cdot \sgn \cdot \textbf{1}})\]
such that
\[\psi(\sV) = H^2(\C_+,\cM_\C)^\sharp \qquad \text{and} \qquad \psi \circ U_t = S_t \circ \psi \quad \forall t \in \R.\]
That \((\cH,\sV,U,J_\sV)\) is outgoing then follows by the Lax--Phillips Theorem (\fref{thm:LaxPhillips}).

If \(-\partial U > 0\), we notice that
\[\tilde U_t \coloneqq J_\sV U_t J_\sV\]
defines a unitary one-parameter group with
\[\tilde U_t \sV' = J_\sV U_t \sV \subeq J_\sV \sV = \sV'\]
(cf. \fref{prop:standardComplement}) and
\[\partial \tilde U = \lim_{t \to 0} \frac 1{it} \left(J_\sV U_t J_\sV - \textbf{1}\right) = J_\sV \lim_{t \to 0} \frac 1{-it} \left(U_t - \textbf{1}\right) J_\sV = -J_\sV \partial U J_\sV > 0.\]
By the above discussion, we then know that \((\cH^\R,\sV',\tilde U)\) is outgoing. This implies
\[\bigcap_{t \in \R} U_t \sV = \bigcap_{t \in \R} J_\sV \tilde U_t J_\sV \sV = J_\sV \bigcap_{t \in \R} \tilde U_t \sV' = \{0\}\]
and
\[\overline{\bigcup_{t \in \R} U_t \sV} = \overline{\bigcup_{t \in \R} J_\sV \tilde U_t J_\sV \sV} = J_\sV \overline{\bigcup_{t \in \R} \tilde U_t \sV'} = \cH,\]
i.e. \((\cH,\sV,U,J_\sV)\) is outgoing
\end{proof}
We want to answer the question if there are outgoing standard quadruples that do not arise in this way from Borchers' Theorem, i.e. for which the infinitesimal generator of the unitary one-parameter group \((U_t)_{t \in \R}\) does not satisfy \(\pm \partial U > 0\). Phrased like this, the question is obviously answered positively, since one can just take direct sums of unitary one-parameter groups with strictly positive and strictly negative generators: We consider complex Hilbert spaces \(\cH_1,\cH_2\), standard subspaces \(\sV_1 \subeq \cH_1\) and \(\sV_2 \subeq \cH_2\) and unitary one-parameter groups
\[(U^{(1)}_t)_{t \in \R} \subeq \U(\cH_1) \qquad \text{and} \qquad (U^{(2)}_t)_{t \in \R} \subeq \U(\cH_2)\]
that satisfy
\[U^{(1)}_t \sV_1 \subeq \sV_1 \quad \text{and} \quad U^{(2)}_t \sV_2 \subeq \sV_2 \qquad \forall t \in \R_+\]
and
\[\partial U^{(1)} > 0 \qquad \text{and} \qquad -\partial U^{(2)} > 0.\]
Then \(\sV \coloneqq \sV_1 \oplus \sV_2\) is a standard subspace in the complex Hilbert space \(\cH \coloneqq \cH_1 \oplus \cH_2\) and defining the unitary one-parameter group \(U \coloneqq U^{(1)} \oplus U^{(2)}\), the quadruple \((\cH^\R,\sV,U,J_\sV)\) is outgoing standard by \fref{prop:BorchersOutgoing}. This motivates the following definition:
\begin{definition}\label{def:BorchersDef}
Let \(\cE\) be a real Hilbert space, \(\cE_+ \subeq \cE\) be a closed subspace, \(\theta \in \U(\cE)\) be an involution and \((U_t)_{t \in \R} \subeq \U(\cE)\) be a unitary one-parameter group. We say that \((\cE,\cE_+,U,\theta)\) is of \textit{Borchers-type}, if there exists a complex Hilbert space \(\cH\), a standard subspace \(\sV \subeq \cH\) and a strongly continuous unitary one-parameter group \((\tilde U_t)_{t \in \R} \subeq \U(\cH)\) as well as a projection \(p \in B(\cH)\) such that the following hold:
\begin{enumerate}[\rm(1)]
\item \(\displaystyle (\cE,\cE_+,U,\theta) \cong (\cH^\R,\sV,\tilde U,J_\sV)\)
\item \(\displaystyle \sV = p\sV \oplus (\textbf{1}-p) \sV\).
\item \(\displaystyle p \tilde U_t = \tilde U_t p \quad \forall t \in \R\).
\item \(\displaystyle \tilde U_t \sV \subeq \sV \quad \forall t \in \R_+\).
\item \(\displaystyle (\partial U)(2p-\textbf{1}) > 0\).
\end{enumerate}
\end{definition}
The above discussion shows that every quadruple \((\cE,\cE_+,U,\theta)\) that is of Borchers-type is also outgoing standard. The question we want to investigate is if there are outgoing standard quadruples that are not of Borchers-type.

We now formulate a theorem analogous to \fref{thm:apperingClassi} for the classification of Borchers-type quadruples:
\begin{thm}{\rm \textbf{(Classification of Borchers-type quadruples)}}\label{thm:BorchersClassi}
Let \(\cM\) be a real Hilbert space and let \(h \in L^\infty\left(\R,\U(\cM_\C)\right)^{\sharp,\flat}\). Then the following are equivalent:
\begin{enumerate}[\rm (a)]
\item The quadruple \(\left(L^2(\R,\cM_\C)^\sharp,H^2(\C_+,\cM_\C)^\sharp,S,\theta_h\right)\) is of Borchers-type.
\item There exists a projection \(p \in B(\cM_\C)^\sharp\) such that, for the function
\[I \coloneqq i \cdot \sgn \cdot (2p-\textbf{1}) \in L^\infty(\R,\U(\cM_\C))^\sharp,\]
the subspace \({\sV \coloneqq H^2(\C_+,\cM_\C)^\sharp}\) is standard in the complex Hilbert space \((L^2(\R,\cM_\C)^\sharp,M_I)\) with \(J_\sV = \theta_h\).
\item \(H_h > 0\) and there exists a projection \(p \in B(\cM_\C)^\sharp\) such that
\[h^* = -(2p-\textbf{1})h(2p-\textbf{1}) \qquad \text{and} \qquad ph = hp.\]
\item \(H_h > 0\) and \(h = i \cdot \cI_{\mu_{H_h}}\).
\item \(h = i \cdot \sgn \cdot \textbf{1}\).
\end{enumerate} 
\end{thm}
\newpage
\begin{proof}
\begin{itemize}
\item[(a) \(\Leftrightarrow\) (b):] By definition, \(\left(L^2(\R,\cM_\C)^\sharp,H^2(\C_+,\cM_\C)^\sharp,S,\theta_h\right)\) is of Borchers-type, if and only if there exists a complex structure \(\cI\) on \(L^2(\R,\cM_\C)^\sharp\) and a complex linear projection \(P \in B(L^2(\R,\cM_\C)^\sharp,\cI)\) such that the following hold:
\begin{enumerate}[\rm(1)]
\item The space \({\sV \coloneqq H^2(\C_+,\cM_\C)^\sharp}\) is a standard subspace in the complex Hilbert space \((L^2(\R,\cM_\C)^\sharp,\cI)\) with \(J_\sV = \theta_h\).
\item For every \(t \in \R\) the operator \(S_t\) is complex linear, i.e. \(S_t \cI = \cI S_t\).
\item \(\displaystyle \sV = P\sV \oplus (\textbf{1}-P) \sV\).
\item \(\displaystyle P S_t = S_t P \quad \forall t \in \R\).
\item \(\displaystyle (\partial S)(2P-\textbf{1}) > 0\).
\end{enumerate}
By \fref{prop:ComplexStrucImpl}(a), condition (2) is equivalent to having \(\cI = M_I\) for some function \(I \in L^\infty\left(\R,\U(\cM_\C)\right)^\sharp\). Similarly, by \fref{prop:SCommutant}, condition~(4) is equivalent to \(P = M_f\) for some function \(f \in L^\infty\left(\R,B(\cM_\C)\right)^\sharp\). Then, condition (3) is equivalent to
\[H^2(\C_+,\cM_\C)^\sharp = M_{2f-\textbf{1}}H^2(\C_+,\cM_\C)^\sharp,\]
which, by \fref{lem:HardyBasic}(e), is equivalent to \(2f-\textbf{1}\) being constant and therefore also to \(f\) being constant. Since \(P=M_f\) is a projection in \(B(L^2(\R,\cM_\C)^\sharp,M_I)\) there exists a projection \(p \in B(\cM_\C)^\sharp\) with \(pI = Ip\) such that \(P = p\). Using that
\[\partial S = \lim_{t \to 0} (M_I t)^{-1} (e^{it M_\mathrm{Id}}-\textbf{1}) = -M_I i M_\mathrm{Id},\]
condition (5) then reads
\[0 < (\partial S)(2P-\textbf{1}) = -M_I i M_\mathrm{Id}(2p-\textbf{1}).\]
This is equivalent to
\[0 < -I(x) ix (2p-\textbf{1}) = |x|(-i\sgn(x) I(x)(2p-\textbf{1}))\]
for almost every \(x \in \R\). Since \(-i\sgn(x) I(x)(2p-\textbf{1})\) is unitary, this is equivalent to
\[-i\sgn(x) I(x)(2p-\textbf{1}) = \textbf{1}\]
for almost every \(x \in \R\). This, in turn, is equivalent to
\[I = i \cdot \sgn \cdot (2p-\textbf{1}).\]
\item[(b) \(\Rightarrow\) (c):] By \fref{prop:ComplexStrucImpl}(b) we have
\(H_h > 0\). By \fref{prop:ComplexStrucImpl}(c) there exists a projection \(\tilde p \in B(\cM_\C)^\sharp\) such that \(I = h(2\tilde p-\textbf{1})\). This yields
\[-(2\tilde p-\textbf{1})h(2\tilde p-\textbf{1}) = -h^*h(2\tilde p-\textbf{1})h(2\tilde p-\textbf{1}) = -h^*II = h^*.\]
Further, for almost every \(x \in \R\), we have
\[I(-x) = i \cdot \sgn(-x) \cdot (2p-\textbf{1}) = -i \cdot \sgn(x) \cdot (2p-\textbf{1}) = -I(x),\]
and therefore
\[h(x)^* = h(-x) = I(-x)u = -I(x)u = -h(x),\]
which implies
\[(2\tilde p-\textbf{1})h(2\tilde p-\textbf{1}) = -h^* = h\]
and therefore
\[\tilde ph = h \tilde p.\]
\item[(c) \(\Rightarrow\) (d):] Using that \(ph = hp\), we get
\[h^* = -(2p-\textbf{1})h(2p-\textbf{1}) = -h(2p-\textbf{1})(2p-\textbf{1}) = -h = -(2 \cdot \textbf{1}-\textbf{1})h(2 \cdot \textbf{1}-\textbf{1}).\]
This, by \fref{thm:SchoberProjUnique}, implies that
\[h = \beta(\mu_{H_h},\textbf{1},0) = i \cdot \cI_{\mu_{H_h}}.\]
\item[(d) \(\Rightarrow\) (e):]
By \fref{lem:kappaSymmetry}, for almost every \(x \in \R\), one has
\[0 \leq \sgn(x) \cI_{\mu_{H_h}}(x) = - i \cdot \sgn(x) h(x).\]
Since \(- i \cdot \sgn(x) h(x)\) is a unitary operator, this implies
\[- i \cdot \sgn \cdot h = \textbf{1},\]
so
\[h = i \cdot \sgn \cdot \textbf{1}.\]
\item[(e) \(\Rightarrow\) (b):]
By \fref{ex:2Lebesgue}, for the function \(h = i \cdot \sgn \cdot \textbf{1}\), the measure \(\mu_{H_h}\) is twice the Lebesgue measure on \(\R_+\). This, by \fref{prop:StrictnessCond}, implies that \(H_h > 0\). Further, for any projection \(p \in B(\cM_\C)^\sharp\), we have
\begin{align*}
-(2p-\textbf{1})h(2p-\textbf{1}) &= -(2p-\textbf{1})(i \cdot \sgn \cdot \textbf{1})(2p-\textbf{1})
\\&= -i \cdot \sgn (2p-\textbf{1})(2p-\textbf{1}) = -i \cdot \sgn \cdot \textbf{1} = h^*.
\end{align*}
By \fref{thm:apperingClassi}(c)\(\Rightarrow\)(b) this implies that there exists a function \(I \in L^\infty(\R,\U(\cM_\C))^\sharp\) with \(I^2 = -\textbf{1}\) such that \({\sV \coloneqq H^2(\C_+,\cM_\C)^\sharp}\) is standard in the complex Hilbert space \((L^2(\R,\cM_\C)^\sharp,M_I)\) with \(J_\sV = \theta_h\). Finally, by \fref{prop:ComplexStrucImpl}(c), there exists a projection \(p \in B(\cM_\C)^\sharp\) such that
\[I = h(2p-\textbf{1}) = i \cdot \sgn \cdot (2p-\textbf{1}). \qedhere\]
\end{itemize}
\end{proof}
Our results can be summarized as follows: Up to orthogonal equivalence, all outgoing reflection positive orthogonal one-parameter groups are of the form
\[\left(L^2(\R,\cM_\C)^\sharp,H^2(\C_+,\cM_\C)^\sharp,S,\theta_h\right)\]
for some real Hilbert space \(\cM\) and some function \(h \in L^\infty\left(\R,\U(\cM_\C)\right)^{\sharp,\flat}\) with \(H_h \geq 0\). If one wants this quadruple to be outgoing standard or even of Borchers-type, one has to impose stronger conditions on the function \(h\), as shown in the following table:
\begin{center}
\addtolength\tabcolsep{1.5pt}
\begin{tabular}{|c||c|c|c|} 
\hline
\raisebox{20pt}{\phantom{M}}
\raisebox{-20pt}{\phantom{M}} & \bf \(\substack{\text{\footnotesize Outgoing reflection} \\ \text{\footnotesize positive orthogonal} \\ \text{\footnotesize one-parameter group}}\) & \bf \(\substack{\text{\footnotesize Outgoing standard} \\ \text{\footnotesize quadruple}}\) & \bf \(\substack{\text{\footnotesize Borchers-type} \\ \text{\footnotesize quadruple}}\) \\
\hline\hline
\raisebox{20pt}{\phantom{M}} \(\substack{\text{\footnotesize \(\mathbf h\)}}\) \raisebox{-20pt}{\phantom{M}} & \(\substack{\text{\footnotesize \(h \in L^\infty\left(\R,\U(\cM_\C)\right)^{\sharp,\flat}\)} \\ \text{\footnotesize with \(H_h \geq 0\)}}\) & \(\substack{\text{\footnotesize \(\beta(\mu,p,C)\)}}\) & \(\substack{\text{\footnotesize \(i \cdot \sgn \cdot \textbf{1}\)}}\) \\
\hline
\raisebox{20pt}{\phantom{M}} \bf \(\substack{\text{\footnotesize \(\mathbf I\)}}\) \raisebox{-20pt}{\phantom{M}} & \(\substack{\text{\footnotesize \slash}}\) & \(\substack{\text{\footnotesize \(\beta(\mu,p,C)\cdot (2p-\textbf{1})\)}}\) & \(\substack{\text{\footnotesize \(i \cdot \sgn \cdot (2p-\textbf{1})\)}}\) \\
\hline
\raisebox{20pt}{\phantom{M}}
\bf \(\substack{\text{\footnotesize \(\mathbf \mu\)}}\)
\raisebox{-20pt}{\phantom{M}} & \(\substack{\text{\footnotesize \slash}}\) & \(\substack{\text{\footnotesize \(\cM_\C\)-Carleson measure} \\ \text{\footnotesize with \(H_\mu > 0\)}}\) & \(\substack{\text{\footnotesize \(d\mu(\lambda) = 2 \cdot d\lambda \cdot \textbf{1}\)}}\)\\
\hline
\end{tabular}
\end{center}
This table shows that, if one wants to construct an outgoing standard quadruple that is not of Borchers-type, one has to find a Carleson measure \(\mu\) with \(H_\mu > 0\), a projection \(p \in B(\cM_\C)^\sharp\) and an operator \(C \in B((\textbf{1}-p)\cM_\C,p\cM_\C)\) such that \(\beta(\mu,p,C) \in L^\infty\left(\R,\U(\cM_\C)\right)^{\sharp,\flat}\). The question which Carleson measures \(\mu\) satisfy \(H_\mu > 0\) is treated in \fref{app:Positive}. The difficult part is finding \(\mu\), \(p\) and \(C\) such that the values of \(\beta(\mu,p,C)\) are unitary.

The following theorem shows that, if one wants to find outgoing standard quadruples that are not of Borchers-type, one has to look at least in \(\dim \cM \geq 2\):
\begin{thm}
Let \(\cM\) be a real Hilbert space with \(\dim \cM = 1\), i.e. \(\cM \cong \R\). Then a quadruple of the form
\[\left(L^2(\R,\cM_\C)^\sharp,H^2(\C_+,\cM_\C)^\sharp,S,\theta_h\right)\]
is outgoing standard, if and only if it is of Borchers-type.
\end{thm}
\begin{proof}
By \fref{thm:apperingClassi}(a)\(\Leftrightarrow\)(c) and \fref{thm:BorchersClassi}(a)\(\Leftrightarrow\)(c) it suffices to show that the condition \(ph = hp\) holds for every projection \(p \in B(\cM_\C)^\sharp\) and every \(h \in L^\infty\left(\R,\U(\cM_\C)\right)^{\sharp,\flat}\). This follows immediately from the fact that \(B(\cM_\C) \cong \C\) is a commutative algebra.
\end{proof}

\subsection{A non Borchers-type example}\label{subsec:exampleQuad}
In this subsection, we construct an example of a quadruple that is outgoing standard but not of Borchers-type. We start with the following definition:
\begin{definition}
Let \(t \in [0,1]\). We set
\[C_t \coloneqq \begin{pmatrix} \sqrt{\frac{1-t}2} & -\frac{\sqrt{1-t^2}}2 \\ \frac{\sqrt{1-t^2}}2 & \sqrt{\frac{1-t}2} \end{pmatrix} \in B(\C^2).\]
Further, for \(\lambda \in \R_+\), we define
\[A_t(\lambda) \coloneqq \frac{2}{1+\lambda^2} \cdot \begin{pmatrix} t+\lambda^2 & 0 \\ 0 & t+\lambda^2 \end{pmatrix} \in B(\C^2)\]
and
\[B_t(\lambda) \coloneqq \frac{2}{1+\lambda^2} \cdot \begin{pmatrix} \sqrt{(1-t)\lambda}(\lambda-1) & -\sqrt{1-t^2} \lambda \\ \sqrt{1-t^2} \lambda & \sqrt{(1-t)\lambda}(\lambda-1) \end{pmatrix} \in B(\C^2)\]
and set
\[\rho_t(\lambda) \coloneqq \begin{pmatrix} A_t(\lambda) & B_t(\lambda)^* \\ B_t(\lambda) & A_t(\lambda) \end{pmatrix} \in B(\C^4).\]
Finally, we define a (signed) \(B(\C^4)\)-valued measure \(\mu_t\) on \(\R_+\) by
\[d\mu_t(\lambda) = \rho_t(\lambda) \,d\lambda.\]
\end{definition}
The goal of this subsection is to prove that, for \(t \in \left(\frac 13,1\right]\), the measure \(\mu_t\) is a Carleson measure for a strictly positive Hankel operator and that, choosing the projection
\[p = \begin{pmatrix} \textbf{1}_{\C^2} & 0 \\ 0 & 0 \end{pmatrix} \in B(\C^4),\]
the function \(\beta(\mu_t,p,C_t)\) is a symbol for this Hankel operator and satisfies
\[\beta(\mu_t,p,C_t) \in L^\infty(\R,\U(\C^4))^{\sharp,\flat}.\]
\newpage
Once we have proved this, we will be able to apply \fref{thm:apperingClassi} and thereby obtain a class of examples of quadruples
\[\left(L^2(\R,\C^4)^\sharp,H^2(\C_+,\C^4)^\sharp,S,\theta_{\beta(\mu_t,p,C_t)}\right)\]
that are outgoing standard but in the case \(t \neq 1\) not of Borchers-type.
\begin{remark}
For \(t = 1\) one has \(C_1 = 0\) and \(\rho_1 \equiv 2 \cdot \textbf{1}\), i.e. \(d\mu_1(\lambda) = 2 \cdot d\lambda \cdot \textbf{1}\). Therefore, by \fref{ex:2LebesgueBeta}, we have
\begin{align*}
\beta(\mu_1,p,C_1) = i \cdot \sgn \cdot \textbf{1}.
\end{align*}
This, by \fref{thm:BorchersClassi}, implies that, for \(t = 1\), the quadruple
\[\left(L^2(\R,\C^4)^\sharp,H^2(\C_+,\C^4)^\sharp,S,\theta_{\beta(\mu_1,p,C_1)}\right)\]
is of Borchers-type.
\end{remark}
To see what happens for other values of \(t\), we start by proving that, for \(t \in \left(\frac 13,1\right]\), the measures \(\mu_t\) are Carleson measures for strictly positive Hankel operators:
\begin{prop}\label{prop:densityPos}
For all \(t \in \left(\frac 13,1\right]\) and \(\lambda \in \R_+\) the matrix \(\rho_t(\lambda)\) is strictly positive.
\end{prop}
\begin{proof}
One has
\[\left(\frac {1+\lambda^2}2\right)^2 B_t(\lambda) B_t(\lambda)^* = \left[\left(\sqrt{(1-t)\lambda}(\lambda-1)\right)^2 + \left(\sqrt{1-t^2} \lambda\right)^2\right] \cdot \textbf{1},\]
so
\begin{align*}
\left(t + \lambda^2\right)^2 - \left(\frac {1+\lambda^2}2\right)^2\left\lVert B_t(\lambda)\right\rVert^2 &= \left(t + \lambda^2\right)^2 - \left[\left(\sqrt{(1-t)\lambda}(\lambda-1)\right)^2 + \left(\sqrt{1-t^2} \lambda\right)^2\right]
\\&= \left[t^2 + \lambda^4 + 2t\lambda^2\right] - \left[(1-t)\lambda(\lambda-1)^2 + (1-t^2)\lambda^2\right]
\\&=\lambda^4 - (1-t)\lambda^3 + (1+t^2)\lambda^2 - (1-t)\lambda + t^2
\\&= (1+\lambda^2) \left[\lambda^2 + (t-1) \lambda + t^2\right]
\\&= (1+\lambda^2) \left[\left(\lambda + \frac{t-1}2\right)^2 + t^2 - \left(\frac{1-t}2\right)^2\right] > 0,
\end{align*}
using for the last estimate that
\[0 \leq \frac{1-t}2 < \frac{1-\frac 13}2 = \frac 13 < t,\]
so
\[t^2 - \left(\frac{1-t}2\right)^2 > 0.\]
We therefore have
\[\frac 2{1+\lambda^2} \cdot (t + \lambda^2) > \left\lVert B_t(\lambda)\right\rVert,\]
which implies that
\[\rho_t(\lambda) = \frac 2{1+\lambda^2} \cdot (t + \lambda^2) \textbf{1} + \begin{pmatrix} 0 & B_t(\lambda)^* \\ B_t(\lambda) & 0 \end{pmatrix}\]
is strictly positive.
\end{proof}
\begin{cor}\label{cor:mutBigger0}
For all \(t \in \left(\frac 13,1\right]\) \(\mu_t\) is a \(\C^4\)-Carleson measure with \(H_{\mu_t} > 0\).
\end{cor}
\begin{proof}
By \fref{prop:densityPos} the signed measure \(\mu_t\) is in fact a (positive) measure. To show that it is a \(\C^4\)-Carleson measure, we notice that the function \(\rho_t\) is bounded since it is continuous and the limits for \(\lambda \downarrow 0\) and \(\lambda \to \infty\) exist and are
\[\lim_{\lambda \downarrow 0} \rho_t(\lambda) = 2t \cdot \textbf{1} \qquad \text{and} \qquad \lim_{\lambda \downarrow 0} \rho_t(\lambda) = 2 \cdot \textbf{1}.\]
This implies that there exists \(K \in \R_+\) such that
\[\rho_t(\lambda) \leq K \cdot 2 \cdot \textbf{1} = K \cdot \rho_1(\lambda) \qquad \forall \lambda \in \R_+.\]
Then, for every \(f \in H^2(\C_+,\C^4)\), we have
\begin{align*}
\int_{\R_+} \braket*{f(i\lambda)}{d\mu_t(\lambda)f(i \lambda)} &= \int_{\R_+} \braket*{f(i\lambda)}{\rho_t(\lambda)f(i \lambda)} \,d\lambda \leq K \cdot \int_{\R_+} \braket*{f(i\lambda)}{\rho_1(\lambda)f(i \lambda)} \,d\lambda
\\&= K \cdot \int_{\R_+} \braket*{f(i\lambda)}{d\mu_1(\lambda)f(i \lambda)} \leq K \cdot \left\lVert H_{\mu_1} \right\rVert \cdot \left\lVert f\right\rVert^2,
\end{align*}
using that \(\mu_1\) is a \(\C^4\)-Carleson measure by \fref{ex:2Lebesgue}. This, by polarization, implies that
\[H^2\left(\C_+,\C^4\right)^2 \ni (f,g) \mapsto \int_{\R_+} \braket*{f(i\lambda)}{d\mu_t(\lambda) g(i\lambda)}\]
defines a continuous sesquilinear form on \(H^2\left(\C_+,\C^4\right)\), i.e. \(\mu_t\) is a \(\C^4\)-Carleson measure. That \(H_{\mu_t} > 0\) then follows from \fref{prop:StrictnessCond}(a).
\end{proof}
Now, we want to show that, for \(t \in \left(\frac 13,1\right]\), choosing the projection
\[p = \begin{pmatrix} \textbf{1}_{\C^2} & 0 \\ 0 & 0 \end{pmatrix} \in B(\C^4),\]
the function \(\beta(\mu_t,p,C_t)\) is a symbol for the Hankel operator \(H_{\mu_t}\) and satisfies
\[\beta(\mu_t,p,C_t) \in L^\infty(\R,\U(\C^4))^{\sharp,\flat}.\]
\begin{prop}\label{prop:betaExact}
Let
\[t \in \left(\frac 13,1\right] \qquad \text{and} \qquad p = \begin{pmatrix} \textbf{1}_{\C^2} & 0 \\ 0 & 0 \end{pmatrix} \in B(\C^4).\]
Then
\[\beta(\mu_t,p,C_t) \in L^\infty(\R,\U(\C^4))^{\sharp,\flat}\]
and \(\beta(\mu_t,p,C_t)\) is a symbol for the Hankel operator \(H_{\mu_t}\).
\end{prop}
\begin{proof}
By \fref{lem:Integrals}(b),(c), for \(x \in \R^\times\), we have
\begin{align*}
\frac i\pi \int_{\R_+} \frac{x}{x^2+\lambda^2} \cdot A_t(\lambda) \,d\lambda &= \sgn(x) \cdot \frac i\pi \int_{\R_+} \frac{|x|}{x^2+\lambda^2} \cdot A_t(\lambda) \,d\lambda
\\&= \frac 1{1+|x|}\begin{pmatrix} i\sgn(x)\left(t+|x|\right) & 0 \\ 0 & i\sgn(x)\left(t+|x|\right) \end{pmatrix}
\end{align*}
and by \fref{lem:Integrals}(d),(g) we get
\[C_t + \frac 1\pi \int_{\R_+} \left[\frac{\lambda}{x^2+\lambda^2} - \frac{\lambda}{1+\lambda^2}\right] \cdot B_t(\lambda) \,d\lambda = \frac 1{1+|x|} \begin{pmatrix} \sqrt{2(1-t)|x|} & -\sqrt{1-t^2} \\ \sqrt{1-t^2} & \sqrt{2(1-t)|x|} \end{pmatrix}.\]
We therefore have
\begin{align}\label{eq:betaExplicitCalc}
\beta(\mu_t,p,C_t)(x) = \frac 1{1+|x|} \cdot \begin{psmallmatrix} i\sgn(x)\left(t+|x|\right) & 0 & \sqrt{2(1-t)|x|} & \sqrt{1-t^2} \\ 0 & i\sgn(x)\left(t+|x|\right) & -\sqrt{1-t^2} & \sqrt{2(1-t)|x|} \\ \sqrt{2(1-t)|x|} & -\sqrt{1-t^2} & i\sgn(x)\left(t+|x|\right) & 0 \\ \sqrt{1-t^2} & \sqrt{2(1-t)|x|} & 0 & i\sgn(x)\left(t+|x|\right) \end{psmallmatrix}.
\end{align}
This immediately yields
\[\beta(\mu_t,p,C_t)(-x) = \beta(\mu_t,p,C_t)(x)^* = \cC_{\R^4}\beta(\mu_t,p,C_t)(x) \cC_{\R^4} \qquad \forall x \in \R^\times.\]
Also, by direct calculation, using that
\begin{align*}
\left(t+|x|\right)^2 + \sqrt{2(1-t)|x|}^2 + \sqrt{1-t^2}^2 &= \left(|x|^2 + 2t|x| + t^2\right) + 2(1-t)|x| + (1-t^2)
\\&= |x|^2 + 2|x| + 1 = (|x|+1)^2,
\end{align*}
one sees that \(\beta(\mu_t,p,C_t)(x)\) is unitary for every \(x \in \R^\times\). We therefore have
\[\beta(\mu_t,p,C_t) \in L^\infty(\R,\U(\C^4))^{\sharp,\flat}.\]
Finally, by \fref{thm:SchoberWithProjections}, \(\beta(\mu_t,p,C_t)\) is a symbol for the Hankel operator \(H_{\mu_t}\).
\end{proof}
With this preparation, we can finally provide a class of examples of quadruples that are outgoing standard but not of Borchers-type:
\begin{thm}\label{thm:standardNotBorchers}
Let
\[t \in \left(\frac 13,1\right) \qquad \text{and} \qquad p = \begin{pmatrix} \textbf{1}_{\C^2} & 0 \\ 0 & 0 \end{pmatrix} \in B(\C^4).\]
Then the quadruple
\[\left(L^2(\R,\C^4)^\sharp,H^2(\C_+,\C^4)^\sharp,S,\theta_{\beta(\mu_t,p,C_t)}\right)\]
is outgoing standard but not of Borchers-type.
\end{thm}
\begin{proof}
By \fref{cor:mutBigger0} we have
\[H_{\mu_t} > 0\]
and by \fref{prop:betaExact} we have
\[\beta(\mu_t,p,C_t) \in L^\infty(\R,\U(\C^4))^{\sharp,\flat}\]
and that \(\beta(\mu_t,p,C_t)\) is a symbol for the Hankel operator \(H_{\mu_t}\). By \fref{thm:apperingClassi} this implies that the quadruple
\[\left(L^2(\R,\C^4)^\sharp,H^2(\C_+,\C^4)^\sharp,S,\theta_{\beta(\mu_t,p,C_t)}\right)\]
is outgoing standard.

On the other hand, for \(t \neq 1\), by equation \fref{eq:betaExplicitCalc}, one has 
\[\beta(\mu_t,p,C_t) \neq i \cdot \sgn \cdot \textbf{1},\]
so, by \fref{thm:BorchersClassi}, the quadruple
\[\left(L^2(\R,\C^4)^\sharp,H^2(\C_+,\C^4)^\sharp,S,\theta_{\beta(\mu_t,p,C_t)}\right)\]
is not of Borchers-type.
\end{proof}

\newpage
\appendix
\section{Strict positivity of Hankel operators}\label{app:Positive}
In this section, given a Carleson measure \(\mu\), we want to know under which conditions the corresponding Hankel operator \(H_\mu\) is strictly positive, i.e. \(H_\mu > 0\). Throughout this section, given a measure \(\mu\), let
\[\mu = \mu_{ac} + \mu_{sc} + \mu_{pp}\]
by the decomposition of \(\mu\) into its absolutely continuous part, its singular continuous part and its pure point part. We start with some sufficient conditions for \(H_\mu\) to be strictly positive:
\begin{prop}\label{prop:StrictnessCond}
Let \(\cK\) be a complex Hilbert space and let \(\mu\) be a \(\cK\)-Carleson measure. Further, assume that one of the following conditions is fulfilled:
\begin{enumerate}[\rm (a)]
\item There exists a subset \(E \subeq \R_+\) with positive Lebesgue-measure such that the Radon-Nikodym derivative
\[\frac{d\mu_{ac}}{d\lambda}(\lambda) \in B(\cK)\]
is strictly positive for almost every \(\lambda \in E\).
\item There exists a countable set \(N \subeq \R_+\) with
\[\sum_{\lambda \in N} \frac \lambda{(\lambda+1)^2} = \infty\]
such that
\[\mu_{pp}(\{\lambda\}) \in B(\cK)\]
is strictly positive for every \(\lambda \in N\).
\end{enumerate}
Then \(H_\mu > 0\).
\end{prop}
\begin{proof}
\begin{enumerate}[\rm (a)]
\item Assume that (a) holds and let \(f \in H^2(\C_+,\cK)\) such that \(\braket*{f}{H_\mu f} = 0\). Then
\begin{align*}
0 &= \braket*{f}{H_\mu f} = \int_{\R_+} \braket*{f(i\lambda)}{d\mu(\lambda)f(i\lambda)} \geq \int_{E} \braket*{f(i\lambda)}{d\mu_{ac}(\lambda)f(i\lambda)}
\\&= \int_{E} \braket*{f(i\lambda)}{\frac{d\mu_{ac}}{d\lambda}(\lambda)f(i\lambda)} \,d\lambda \geq 0,
\end{align*}
so
\[\int_{E} \braket*{f(i\lambda)}{\frac{d\mu_{ac}}{d\lambda}(\lambda)f(i\lambda)} \,d\lambda = 0.\]
This implies that
\[\braket*{f(i\lambda)}{\frac{d\mu_{ac}}{d\lambda}(\lambda)f(i\lambda)} = 0\]
for almost every \(\lambda \in E\). Since \(\frac{d\mu_{ac}}{d\lambda}(\lambda)\) is strictly positive for every \(\lambda \in E\), this implies that
\[f(i\lambda) = 0\]
for almost every \(\lambda \in E\). Since the zeros of non-zero holomorphic functions are discrete and therefore a Lebesgue-null set, this implies that \(f = 0\). We therefore have
\[\braket*{f}{H_\mu f} > 0\]
for every \(f \in H^2(\C_+,\cK) \setminus \{0\}\) and therefore \(H_\mu > 0\).
\item Assume that (b) holds and let \(f \in H^2(\C_+,\cK)\) such that \(\braket*{f}{H_\mu f} = 0\). Then
\begin{align*}
0 &= \braket*{f}{H_\mu f} = \int_{\R_+} \braket*{f(i\lambda)}{d\mu(\lambda)f(i\lambda)} \geq \int_{N} \braket*{f(i\lambda)}{d\mu_{pp}(\lambda)f(i\lambda)}
\\&= \sum_{\lambda \in N} \braket*{f(i\lambda)}{\mu_{pp}(\{\lambda\})f(i\lambda)} \geq 0,
\end{align*}
so
\[\sum_{\lambda \in N} \braket*{f(i\lambda)}{\mu_{pp}(\{\lambda\})f(i\lambda)} = 0.\]
This implies that
\[\braket*{f(i\lambda)}{\mu_{pp}(\{\lambda\})f(i\lambda)} = 0 \qquad \forall \lambda \in N.\]
Since \(\mu_{pp}(\{\lambda\})\) is strictly positive for every \(\lambda \in N\), this implies that
\[f(i\lambda) = 0 \qquad \forall \lambda \in N.\]
By Szeg{\H o}'s Theorem, for every \(F \in H^2(\C_+,\cK) \setminus \{0\}\), we have
\[\sum_{z \in F^{-1}(\{0\})} \frac{\Im(z)}{\Re(z)^2 + (\Im(z)+1)^2} < \infty\]
(cf. \cite[Thm.~5.13(ii)]{RR94}). But for the function \(f\), since \(iN \subeq f^{-1}(\{0\})\), we get
\[\sum_{z \in f^{-1}(\{0\})} \frac{\Im(z)}{\Re(z)^2 + (\Im(z)+1)^2} \geq \sum_{\lambda \in N} \frac{\Im(i\lambda)}{\Re(i\lambda)^2 + (\Im(i\lambda)+1)^2}= \sum_{\lambda \in N} \frac{\lambda}{(\lambda+1)^2} = \infty,\]
so \(f = 0\). We therefore have
\[\braket*{f}{H_\mu f} > 0\]
for every \(f \in H^2(\C_+,\cK) \setminus \{0\}\) and therefore \(H_\mu > 0\).
\end{enumerate}
\end{proof}
\begin{example}\label{ex:conditionFails}
\begin{enumerate}[\rm (a)]
\item We consider the \(B(\C^2)\)-valued measure \(\mu\) defined by
\[d\mu(\lambda) = \left[\chi_{(0,1)}(\lambda) \cdot \begin{pmatrix} 1 & 0 \\ 0 & 0 \end{pmatrix} + \chi_{[1,\infty)}(\lambda) \cdot \begin{pmatrix} 0 & 0 \\ 0 & 1 \end{pmatrix} \right] \,d\lambda.\]
Then \(\mu_{pp} = 0\) and clearly, for every \(\lambda \in \R_+\), the operator
\[\frac{d\mu_{ac}}{d\lambda}(\lambda)\]
is not strictly positive, so neither (a) nor (b) are fulfilled. Now let
\[f = \begin{pmatrix} f_1 \\ f_2 \end{pmatrix} \in H^2(\C_+,\C^2)\]
with \(\braket*{f}{H_\mu f} = 0\). Then
\begin{align*}
0 &= \braket*{f}{H_\mu f} = \int_{\R_+} \braket*{f(i\lambda)}{d\mu(\lambda)f(i\lambda)} = \int_{(0,1)} |f_1(i\lambda)|^2 \,d\lambda + \int_{[1,\infty)} |f_2(i\lambda)|^2 \,d\lambda.
\end{align*}
This implies that
\[f_1(i\lambda) = 0\]
for almost every \(\lambda \in (0,1)\)
and
\[f_2(i\lambda) = 0\]
for almost every \(\lambda \in [1,\infty)\).
This yields \(f_1 = 0\) and \(f_2 = 0\), so \(f = 0\). We therefore have
\[\braket*{f}{H_\mu f} > 0\]
for every \(f \in H^2(\C_+,\C^2) \setminus \{0\}\) and therefore \(H_\mu > 0\).

This example shows that, for the Hankel operator \(H_\mu\) to be strictly positive, it is not necessary that either condition (a) or (b) of \fref{prop:StrictnessCond} are fulfilled, they are just sufficient.

\item The last example suggests that, instead of the existence of a subset \(E \subeq \R_+\) with positive Lebesgue measure such that the Radon-Nikodym derivative of \(d\mu_{ac}\) is strictly positive on \(E\), one should rather assume that, for every \(v \in \cK \setminus \{0\}\), there exists a subset \(E_v \subeq \R_+\) with positive Lebesgue-measure such that the Radon-Nikodym derivative of \(\braket*{v}{d\mu_{ac}v}\) is strictly positive on \(E_v\). The following example though shows that this is not sufficient:

We consider the \(B(\C^2)\)-valued measure defined by
\[d\mu(\lambda) = \left|\begin{pmatrix} 1 \\ -\frac 1{1+\lambda} \end{pmatrix}\right\rangle \left\langle\begin{pmatrix} 1 \\ -\frac 1{1+\lambda} \end{pmatrix}\right| \,d\lambda.\]
Then, for every
\[v = \begin{pmatrix} v_1 \\ v_2 \end{pmatrix} \in \C^2,\]
we have
\[\braket*{v}{\frac{d\mu}{d\lambda}(\lambda)v} = \left|\braket*{\begin{pmatrix} 1 \\ -\frac 1{1+\lambda} \end{pmatrix}}{v}\right|^2 = \left|v_1 - \frac{v_2}{1+\lambda}\right|^2.\]
Therefore, for \(v \neq 0\), we have
\[\braket*{v}{\frac{d\mu}{d\lambda}(\lambda)v} > 0\]
for almost every \(\lambda \in \R_+\).

On the other hand we consider the function \(f \in H^2(\C_+,\C^2)\) given by
\[f(z) = \begin{pmatrix} \left(\frac{i}{i+z}\right)^2 \\ \frac{i}{i+z} \end{pmatrix}.\]
Then
\[f(i\lambda) = \begin{pmatrix} \left(\frac{i}{i+i\lambda}\right)^2 \\ \frac{i}{i+i\lambda}\end{pmatrix} = \begin{pmatrix} \left(\frac{1}{1+\lambda}\right)^2 \\ \frac{1}{1+\lambda} \end{pmatrix} = \left(\frac{1}{1+\lambda}\right)^2 \cdot \begin{pmatrix} 1 \\ 1+\lambda \end{pmatrix}.\]
Therefore
\begin{align*}
\braket*{f}{H_\mu f} &= \int_{\R_+} \braket*{f(i\lambda)}{\frac{d\mu}{d\lambda}(\lambda)f(i\lambda)} \,d\lambda 
\\&= \int_{\R_+} \left|\left(\frac{1}{1+\lambda}\right)^2 \cdot \braket*{\begin{pmatrix} 1 \\ -\frac 1{1+\lambda} \end{pmatrix}}{\begin{pmatrix} 1 \\ 1+\lambda \end{pmatrix}}\right|^2 \,d\lambda
\\&= \int_{\R_+} \left|\left(\frac{1}{1+\lambda}\right)^2 \cdot 0\right|^2 \,d\lambda = 0,
\end{align*}
which implies that \(H_\mu\) is not strictly positive.
\end{enumerate}
\end{example}
After seeing some sufficient conditions for \(H_\mu\) to be strictly positive, we provide a condition for \(H_\mu\) to be not strictly positive:
\begin{prop}\label{prop:strictnecessary}
Let \(\cK\) be a complex Hilbert space and let \(\mu\) be a \(\cK\)-Carleson measure. If there exists a countable set \(N \subeq \R_+\) with
\[\sum_{\lambda \in N} \frac \lambda{(\lambda+1)^2} < \infty\]
such that the operator \(\mu(\R_+ \setminus N)\) is not strictly positive, then \(H_\mu \ngtr 0\).
\end{prop}
\begin{proof}
Let \(N \subeq \R_+\) be a countable set with
\[\sum_{\lambda \in N} \frac \lambda{(\lambda+1)^2} < \infty\]
such that
\begin{equation}\label{eq:NZeroSet}
\braket*{v}{\mu(\R_+ \setminus N)v} = 0
\end{equation}
for some vector \(v \in \cK \setminus \{0\}\). Then, for the subset \(iN \subeq \R_+\), we have
\[\sum_{z \in iN} \frac{\Im(z)}{\Re(z)^2 + (\Im(z)+1)^2} = \sum_{\lambda \in N} \frac{\Im(i\lambda)}{\Re(i\lambda)^2 + (\Im(i\lambda)+1)^2} = \sum_{\lambda \in N} \frac \lambda{(\lambda+1)^2} < \infty,\]
so, by \cite[Thm.~5.13(ii)]{RR94}, we can form the Blaschke product
\[B(z) \coloneqq \left(\prod_{\substack{\lambda \in N \\ \lambda \leq 1}} \frac{z-i\lambda}{z+i\lambda}\right) \cdot \left(\prod_{\substack{\lambda \in N \\ \lambda > 1}} -\frac{z-i\lambda}{z+i\lambda}\right),\]
for which one has \(B \in H^\infty(\C_+) \setminus \{0\}\) and
\[B(i\lambda) = 0 \qquad \forall \lambda \in N.\]
Now, for every function \(f \in H^2(\C_+)\), one has \(B \cdot f \otimes v \in H^2(\C_+,\cK)\) and
\begin{align*}
\braket*{(B \cdot f \otimes v)}{H_\mu (B \cdot f \otimes v)} &= \int_{\R_+} |B(i\lambda)|^2 \cdot |f(i\lambda)|^2 \cdot \braket*{v}{d\mu(\lambda)v}
\\&\overset{\fref{eq:NZeroSet}}{=} \int_{N} |B(i\lambda)|^2 \cdot |f(i\lambda)|^2 \cdot \braket*{v}{d\mu(\lambda)v}
\\&= \sum_{\lambda \in N} \mu(\{\lambda\}) \cdot |B(i\lambda)|^2 \cdot |f(i\lambda)|^2
\\&= \sum_{\lambda \in N} \mu(\{\lambda\}) \cdot 0 \cdot |f(i\lambda)|^2 = 0.
\end{align*}
This shows that \(H_\mu\) is not strictly positive.
\end{proof}
Even though, for general Hilbert spaces \(\cK\), the converse implication in the above proposition does not hold (see \fref{ex:conditionFails}(b)), equivalence holds in the one-dimensional case \(\cK = \C\):
\begin{cor}
Let \(\mu\) be a \(\C\)-Carleson measure. Then \(H_\mu \ngtr 0\), if and only if there exists a countable set \(N \subeq \R_+\) with
\[\sum_{\lambda \in N} \frac \lambda{(\lambda+1)^2} < \infty\]
such that
\[\mu(\R_+ \setminus N) = 0.\]
\end{cor}
\begin{proof}
\begin{itemize}
\item [``\(\Leftarrow\)'':] This follows immediately from \fref{prop:strictnecessary}.
\item [``\(\Rightarrow\)'':] Assume that \(H_\mu\) is not strictly positive, so there exists \(f \in H^2(\C_+) \setminus \{0\}\) such that
\[\braket*{f}{H_\mu f} = 0.\]
We now assume that \(\tilde \mu \coloneqq \mu_{ac} + \mu_{sc} \neq 0\). Then there exists a compact subset \(K \subeq \R_+\) such that
\[\tilde \mu(K) > 0.\]
Since \(iK \subeq \C_+\) is compact and \(f \neq 0\) is holomorphic, the set
\[iK \cap f^{-1}(\{0\})\]
is finite. Since \(\tilde \mu\) has no pure point part, for every \(\lambda \in \R_+\), we have
\[\lim_{\epsilon \downarrow 0} \tilde \mu(U_\epsilon(\lambda)) = 0\]
and therefore also
\[\lim_{\epsilon \downarrow 0} \tilde \mu\left(\bigcup_{\lambda \in iK \cap f^{-1}(\{0\})} U_\epsilon(\lambda)\right) = 0\]
We therefore can choose \(\epsilon > 0\) sufficiently small such that, for the compact set
\[C \coloneqq K \setminus \left(\bigcup_{\lambda \in iK \cap f^{-1}(\{0\})} U_\epsilon(\lambda)\right),\]
one has
\[\tilde \mu(C) > 0.\]
Since \(iC\) is compact and \(|f|^2\) is continuous, the minimum
\[c \coloneqq \min_{\lambda \in C} |f(i\lambda)|^2\]
exists, but, by construction, the function \(f\) has no zeros on \(iC\), so \(c > 0\). This implies that
\begin{align*}
\braket*{f}{H_\mu f} &= \int_{\R_+} \left|f(i\lambda)\right|^2 \,d\mu(\lambda) \geq \int_{C} \left|f(i\lambda)\right|^2 \,d\tilde \mu(\lambda) \geq \int_{C} c \,d\tilde \mu(\lambda) = c \cdot \tilde \mu(C) > 0,
\end{align*}
in contradiction to our assumption \(\braket*{f}{H_\mu f} = 0\). Therefore
\[0 = \tilde \mu = \mu_{ac} + \mu_{sc},\]
so
\[\mu = \mu_{pp}.\]
Since
\[\int_{\R_+} \frac 1{1+\lambda} \,d\mu_{pp}(\lambda) = \int_{\R_+} \frac 1{1+\lambda} \,d\mu(\lambda) < \infty,\]
we know that \(\mu = \mu_{pp}\) is supported on a countable set, so there exists a countable set \(N \subeq \R_+\) such that
\[\mu(\R_+ \setminus N) = 0\]
and
\[\mu(\{n\}) > 0 \qquad \forall n \in N.\]
If 
\[\sum_{n \in N} \frac n{(n+1)^2} = \infty,\]
then, by \fref{prop:StrictnessCond}(b), we would have that \(H_\mu\) is strictly positive in contradiction to our assumption. This implies that
\[\sum_{n \in N} \frac n{(n+1)^2} < \infty. \qedhere\]
\end{itemize}
\end{proof}

\newpage
\section{Hardy spaces and multiplication operators}

\subsection{Hardy spaces on half-planes}
In this subsection, we will introduce the concept of Hardy spaces and collect some basic facts about them. We will restrict ourselves mostly to Hardy spaces on the upper half-plane \(\C_+\), but everything can be done analogously for the lower half-plane \(\C_-\). We start with the definition of a Hardy space:
\begin{definition}\label{def:HardyScalarDef}
For a holomorphic function \(f: \C_+ \to \C\), we set
\[\left\lVert f\right\rVert_{H^p} \coloneqq \sup_{y > 0} \left(\int_\R \left|f(x+iy)\right|^p \,dx\right)^{\frac 1p} \quad \text{for } 1 \leq p < \infty\]
and
\[\left\lVert f\right\rVert_{H^\infty} \coloneqq \sup_{z \in \C_+} \left|f(z)\right|.\]
Now, for \(1 \leq p \leq \infty\), we define the \textit{Hardy space} \(H^p(\C_+)\) by
\[H^p(\C_+) \coloneqq \{f: \C_+ \to \C \text{\,holomorphic} : \left\lVert f\right\rVert_{H^p} < \infty\}.\]
\end{definition}
\begin{theorem}{\rm (\cite[Cor.~5.24, Sec.~5.29]{RR94})}
For \(1 \leq p \leq \infty\), the Hardy space \(H^p(\C_+)\) is a Banach space. Further, \(H^2(\C_+)\) becomes a Hilbert space when endowed with the scalar product
\[\braket*{f}{g}_{H^2} \coloneqq \lim_{y \downarrow 0} \int_\R \overline{f(x+iy)} g(x+iy) dx, \quad f,g \in H^2(\C_+,\cK).\]
\end{theorem}
\begin{theorem}\label{thm:HpBoundaryValues}{\rm (\cite[Cor.~5.17, Thm.~5.19, Thm.~5.23, Sec.~5.29]{RR94})}
For \(1 \leq p \leq \infty\), for every \(f \in H^p(\C_+)\), the non-tangential limit
\[f_*(x) \coloneqq \lim_{\epsilon \downarrow 0} f(x+i\epsilon)\]
exists for almost every \(x \in \R\) and defines a function \(f_* \in L^p(\R,\C)\) with \(\left\lVert f_*\right\rVert_p = \left\lVert f\right\rVert_{H^p}\). In particular, for the scalar product on \(H^2(\C_+)\), one has
\[\braket*{f}{g}_{H^2} = \braket*{f_*}{g_*}_{L^2} \quad \forall f,g \in H^2(\C_+).\]
\end{theorem}
This theorem shows that we have an isometric embedding
\begin{equation*}
\iota_p: H^p\left(\C_+\right) \to L^p\left(\R,\C\right), \quad f \mapsto f_*.
\end{equation*}
From now on, we will identify \(H^p\left(\C_+\right)\) with \(\iota_p\left(H^p\left(\C_+\right)\right) \subeq L^p\left(\R,\C\right)\) and we will write \(f(x)\) instead of \(f_*(x)\) for \(f \in H^p\left(\C_+\right)\) and \(x \in \R\).

After considering Hardy spaces in general, we will now focus on the Hardy space \(H^2(\C_+)\). \(H^2(\C_+)\) as a Hilbert space is an example of a so-called reproducing kernel Hilbert space, defined as follows:
\begin{definition}
Let \(X\) be an arbitrary set and \(\cH\) be a Hilbert space that is a subset of the complex-valued functions on \(X\). We say that \(\cH\) is a \textit{reproducing kernel Hilbert space} if the linear functionals
\begin{equation*}
\ev_x: \cH \to \C, \quad f \mapsto f\left(x\right)
\end{equation*}
are continuous for every \(x \in X\). Then, by the Riesz Representation Theorem, for every \(x \in X\), there exists a function \(Q_x \in \cH\) such that
\begin{equation*}
f\left(x\right) = \ev_x\left(f\right) = \braket*{Q_x}{f}.
\end{equation*}
\end{definition}
\begin{lemma}\label{lem:evDense}
Let \(\cH\) be a reproducing kernel Hilbert space on a set \(X\). Then
\begin{equation*}
\overline{\mathrm{span}\left\lbrace Q_x: x \in X\right\rbrace} = \cH.
\end{equation*}
\end{lemma}
\begin{proof}
Let \(f \in \left\lbrace Q_x: x \in X\right\rbrace^\perp\). Then for every \(x \in X\) we have
\begin{equation*}
0 = \braket*{Q_x}{f} = f\left(x\right)
\end{equation*}
and therefore \(f=0\).
\end{proof}
\begin{prop}{\rm (\cite[Ex.~C.2]{ANS22}, \cite[Prop.~A.1.7]{Sc23})}\label{prop:H2KernelFunction}
The Hilbert space \(H^2\left(\C_+\right)\) is a reproducing kernel Hilbert space and its reproducing kernel functions are given by the Szegö kernel
\begin{equation*}
Q_w: \C_+ \to \C, \quad z \mapsto \frac{1}{2\pi} \cdot \frac{i}{z-\overline{w}}, \qquad w \in \C_+.
\end{equation*}
\end{prop}
After considering Hardy spaces of scalar-valued functions, we now want to introduce Hardy spaces with values in a Hilbert space or in the bounded operators on some Hilbert space respectively:
\begin{definition}(\cite[Def.~5.15]{RR94}, \cite[Def.~A.1.9]{Sc23})\label{def:HardyVectorDef}
Let \(\cK\) be a complex Hilbert space.
\begin{enumerate}[\rm (a)]
\item For a holomorphic function \(f: \C_+\to \cK\), we set
\[\left\lVert f\right\rVert_{H^2} \coloneqq \sup_{y > 0} \left(\int_\R \left\lVert f(x+iy)\right\rVert^2 \,dx\right)^{\frac 12}\]
and define
\[H^2(\C_+,\cK) \coloneqq \{f: \C_+\to \cK \text{\,holomorphic} : \left\lVert f\right\rVert_{H^2} < \infty\}.\]
\item For a holomorphic function \(f: \C_+ \to B(\cK)\), we set
\[\left\lVert f\right\rVert_{H^\infty} \coloneqq \sup_{z \in \C_+} \left\lVert f(z)\right\rVert\]
and define
\[H^\infty(\C_+,B(\cK)) \coloneqq \{f: \C_+ \to B(\cK) \text{\,holomorphic} : \left\lVert f\right\rVert_{H^\infty} < \infty\}.\]
\end{enumerate}
\end{definition}
We summarize some basic properties of vector/operator-valued Hardy spaces in the following lemma:
\begin{lemma}{\rm (\cite[Lem.~A.1.6, Prop.~A.1.7, Cor.~A.1.12, Prop.~A.1.13, Cor.~A.1.14]{Sc23})}\label{lem:HardyBasic} Let \(\cK\) be complex Hilbert space. Then the following hold:
\begin{enumerate}[\rm (a)]
\item \(\displaystyle \overline{\spann \left\{Q_z \cdot v: z \in \C_+, v \in \cK\right\}} = H^2(\C_+,\cK).\)
\item Let \(f \in H^\infty(\C_+,B(\cK))\) have self-adjoint boundary values, i.e. \(f(x) = f(x)^*\) for almost every \(x \in \R\). Then \(f\) is constant.
\item \(\displaystyle H^\infty(\C_+,B(\cK)) \cap H^\infty(\C_-,B(\cK)) = B(\cK) \cdot \textbf{1}.\)
\item Let \(f \in L^\infty(\R,B(\cK))\). Then, the inclusion \(M_fH^2(\C_+,\cK) \subeq H^2(\C_+,\cK)\) holds, if and only if \({f \in H^\infty(\C_+,B(\cK))}\).
\item Let \(f,g \in L^\infty(\R,B(\cK))\) such that \(f(x),g(x) \in \U(\cK)\) for almost all \(x \in \R\). Then, one has \(M_f H^2(\C_+,\cK) = M_g H^2(\C_+,\cK)\), if and only if there exists \(u \in \U(\cK)\) such that \(f = g \cdot u\).
\end{enumerate}
\end{lemma}
\begin{proof}
\begin{enumerate}[\rm (a)]
\item By \fref{lem:evDense} and \fref{prop:H2KernelFunction}, we have
\[\overline{\spann \left\{Q_z: z \in \C_+\right\}} = H^2(\C_+).\]
This yields
\[\overline{\spann \left\{Q_z \cdot v: z \in \C_+, v \in \cK\right\}} = \overline{\spann \left\{Q_z: z \in \C_+\right\}} \cdot \cK = H^2(\C_+) \cdot \cK = H^2(\C_+,\cK).\]
\item We start with the scalar-valued case. So let \(f \in H^\infty(\C_+)\) with \(f(x) \in \R\) for almost every \(x \in \R\). Since \(\exp: \C \to \C\) maps bounded sets to bounded sets, we have
\[g_\pm \coloneqq e^{\pm if} \in H^\infty(\C_+)\]
with \(|g_\pm(x)| = 1\) for almost every \(x \in \R\). This implies that
\[|e^{\pm if(z)}| = |g_\pm(z)| \leq \left\lVert g_\pm\right\rVert_\infty = 1 \qquad \forall z \in \C_+.\]
This yields
\[\Im(f(z)) \geq 0 \quad \text{and} \quad \Im(f(z)) \leq 0 \qquad \forall z \in \C_+\]
and therefore \(f(z) \in \R\) for every \(z \in \R\). This, by the Cauchy--Riemann equations, implies that \(f\) is constant.

We now consider the operator-valued case. So let \(f \in H^\infty(\C_+,B(\cK))\) with \(f(x) = f(x)^*\) for almost every \(x \in \R\). Then, for \(v \in \cK\), we define \(f_v \in H^\infty(\C_+)\) by
\[f_v(z) = \braket*{v}{f(z)v}, \qquad z \in \C_+.\]
Then \(f_v(x) \in \R\) for almost every \(x \in \R\), so by the above discussion \(f_v\) is constant. Then, by polarization, for every \(v,w \in \cH\), the map
\[\C_+ \ni z \mapsto \braket*{v}{f(z)w}\]
is constant, i.e. \(f\) is constant.
\item Let \(f \in H^\infty(\C_+,B(\cK)) \cap H^\infty(\C_-,B(\cK))\). Then \(f^* \in H^\infty(\C_+,B(\cK))\), so
\[\Re(f) \coloneqq \frac 12 (f+f^*) \in H^\infty(\C_+,B(\cK))\]
and
\[\Im(f) \coloneqq \frac 1{2i} (f-f^*) \in H^\infty(\C_+,B(\cK)).\]
Since both \(\Re(f)\) and \(\Im(f)\) have self-adjoint boundary values, by (b), they are constant. This implies that also \(f = \Re(f) + i \Im(f)\) is constant.
\item Let \(f \in H^\infty(\C_+,B(\cK))\) and \(g \in H^2(\C_+,\cK)\). Then obviously \(M_f g = f \cdot g\) is holomorphic. Further
\begin{align*}
\left\lVert M_f g\right\rVert_2^2 &= \left\lVert f \cdot g\right\rVert_2^2 = \sup_{y > 0} \int_\R \left\lVert f(x+iy) g(x+iy)\right\rVert^2 \,dx
\\&\leq \sup_{y > 0} \int_\R \left\lVert f\right\rVert_\infty^2 \cdot \left\lVert g(x+iy)\right\rVert^2 \,dx = \left\lVert f\right\rVert_\infty^2 \cdot \left\lVert g\right\rVert_2^2,
\end{align*}
so \(M_f g \in H^2(\C_+,\cK)\).

Now, for the converse, we first consider the scalar valued case. So let \(f \in L^\infty(\R,\C)\) with
\[M_fH^2(\C_+) \subeq H^2(\C_+).\]
This, in particular, implies \(f \cdot Q_i \in H^2(\C_+)\). Therefore, by \cite[Thm.~5.19(i)]{RR94}, we have
\[0 = \frac 1{2\pi i} \int_\R \frac{f(x) \cdot Q_i(x)}{x-\overline{z}} \,dx = \frac 1{4\pi^2} \int_\R \frac{f(x)}{(x-\overline{z})(x+i)} \,dx \quad \forall z \in \C_+,\]
which, by \cite[Sec.~3.1]{Ma09}, implies that \(f \in H^\infty(\C_+)\).

We now consider the operator-valued case, so let \(f \in L^\infty(\R,B(\cK))\) with
\[M_fH^2(\C_+,\cK) \subeq H^2(\C_+,\cK).\]
Then, for every \(v,w \in \cK\), defining
\[f_{v,w} \coloneqq \braket*{v}{f(\cdot)w} \in L^\infty(\R,\C),\]
we get
\[M_{f_{v,w}}H^2(\C_+) = \braket*{v}{M_f (H^2(\C_+) \otimes w)} \subeq \braket*{v}{M_fH^2(\C_+,\cK)} \subeq \braket*{v}{H^2(\C_+,\cK)} \subeq H^2(\C_+),\]
so \(f_{v,w} \in H^\infty(\C_+)\). This means that, for \(z \in \C_+\), we can define the map
\[F_z: \cK \times \cK \to \C, \quad (v,w) \mapsto f_{v,w}(z),\]
which is sesquilinear by construction. Further, for every \(z \in \C_+\), we have
\[|F_z(v,w)| = |f_{v,w}(z)| = \left| \braket*{v}{f(z)w}\right| \leq \left\lVert f\right\rVert_\infty \cdot \left\lVert v\right\rVert \cdot \left\lVert w\right\rVert \qquad \forall v,w \in \cK,\]
so there exists \(A(z) \in B(\cK)\) with \(\left\lVert A(z) \right\rVert \leq \left\lVert f\right\rVert_\infty\) such that
\[F_z(v,w) = \braket*{v}{A(z) w}, \qquad \forall v,w \in \cK.\]
Then
\[A: \C_+ \to B(\cK), \quad z \mapsto A(z)\]
is a bounded function, for which all the functions
\[\braket*{v}{A(\cdot)w} = f_{v,w}, \quad v,w \in \cK\]
are holomorphic. This, by \cite[Cor.~A.III.5]{Ne00}, implies that \(A\) is holomorphic and therefore \(A \in H^\infty(\C_+,B(\cK))\). Now, by construction, for every \(v,w \in \cK\) and almost every \(x \in \R\), we have
\[\braket*{v}{A(x)w} = f_{v,w}(x) = \braket*{v}{f(x)w},\]
which implies that \(f = A \in H^\infty(\C_+,B(\cK))\).
\item By multiplying with \(M_{g^{-1}}\) and \(M_{f^{-1}}\) respectively from the left, we get
\[M_{g^{-1}f} H^2(\C_+,\cK) = H^2(\C_+,\cK) \quad \text{and} \quad M_{f^{-1}g} H^2(\C_+,\cK) = H^2(\C_+,\cK).\]
By (d) this implies \(g^{-1}f,f^{-1}g \in H^\infty(\C_+,B(\cK))\). On the other hand
\[g^{-1}f = (f^{-1}g)^{-1} = (f^{-1}g)^* \in H^\infty(\C_-,B(\cK)),\]
so, by (c), we get
\[g^{-1}f \in H^\infty(\C_+,B(\cK)) \cap H^\infty(\C_-,B(\cK)) = B(\cK) \cdot \textbf{1}\]
and therefore, there exists \(u \in B(\cK)\) such that \(f = g \cdot u\). Finally, \(u \in \U(\cK)\) since \(f\) and \(g\) are unitary. \qedhere
\end{enumerate}
\end{proof}

\subsection{Holomorphic functions with symmetric boundary values}
In this subsection, we will see that holomorphic functions with symmetric boundary values that satisfy certain integrability conditions are constant:
\begin{prop}\label{prop:SymConst}
Let \(f: \C_+ \to \C\) be a holomorphic function such that, for almost every \(x \in \R\), the limit
\[f(x) \coloneqq \lim_{\epsilon \downarrow 0} f(x+i\epsilon)\]
exists and such that \(f(-x) = f(x)\) holds. If further \(f \cdot Q_{i\lambda} \in H^2(\C_+)\) for every \(\lambda \in \R_+\), then \(f\) is constant.
\end{prop}
\begin{proof}
Let \(\lambda,\mu \in \R_+\). Then
\[(f \cdot Q_{i\lambda}) \cdot Q_{i\mu} \in H^2(\C_+) \cdot H^\infty(\C_+) = H^2(\C_+)\] by \fref{lem:HardyBasic}(d) and we have
\begin{align}\label{eq:BigfiLCalc}
f(i\lambda) Q_{i\lambda}(i\lambda)Q_{i\mu}(i\lambda) &= \braket*{Q_{i\lambda}}{f \cdot Q_{i\lambda} \cdot Q_{i\mu}} \notag
\\&= \int_\R f(x) \cdot Q_{i\lambda}(x) \cdot Q_{i\mu}(x) \cdot \overline{Q_{i\lambda}(x)} \,dx \notag
\\&=\int_\R f(x) \cdot \frac i{2\pi(x+i\lambda)} \cdot \frac i{2\pi(x+i\mu)} \cdot \frac {-i}{2\pi(x-i\lambda)} \,dx \notag
\\&=\frac i{(2\pi)^3} \int_\R \frac{f(x)(x-i\mu)}{(x^2+\lambda^2)(x^2+\mu^2)} \,dx \notag
\\&=\frac i{(2\pi)^3} \int_\R x \cdot \frac{f(x)}{(x^2+\lambda^2)(x^2+\mu^2)} \,dx + \frac \mu{(2\pi)^3} \int_\R \frac{f(x)}{(x^2+\lambda^2)(x^2+\mu^2)} \,dx \notag
\\&=\frac \mu{(2\pi)^3} \int_\R \frac{f(x)}{(x^2+\lambda^2)(x^2+\mu^2)} \,dx,
\end{align}
where, for the last equality, we use that \(f(-x) = f(x)\) for almost every \(x \in \R\), so the integrand of the first integral is anti-symmetric and therefore the integral vanishes. Then using that
\begin{equation}\label{eq:QiMuValue}
Q_{i\mu}(i\lambda) = \frac{i}{2\pi(i\lambda+i\mu)} = \frac 1{2\pi(\lambda+\mu)}
\end{equation}
and therefore
\begin{equation}\label{eq:QiLValue}
Q_{i\lambda}(i\lambda) = \frac 1{2\pi(\lambda+\lambda)} = \frac 1{4\pi\lambda},
\end{equation}
for \(\mu \neq \lambda\), we get
\begin{align*}
f(i\lambda) &\overset{\fref{eq:BigfiLCalc}}{=} \frac 1{Q_{i\lambda}(i\lambda)Q_{i\mu}(i\lambda)} \cdot \frac \mu{(2\pi)^3} \int_\R \frac{f(x)}{(x^2+\lambda^2)(x^2+\mu^2)} \,dx
\\&\overset{\fref{eq:QiMuValue},\fref{eq:QiLValue}}{=} 2\pi(\lambda+\mu) \cdot 4\pi\lambda \cdot \frac \mu{(2\pi)^3} \int_\R \frac{f(x)}{\mu^2-\lambda^2} \cdot \left[\frac 1{x^2+\lambda^2} - \frac 1{x^2+\mu^2}\right] \,dx
\\&= \frac{\lambda\mu(\lambda+\mu)}{\pi(\mu^2-\lambda^2)} \int_\R \frac{f(x)}{x^2+\lambda^2} - \frac{f(x)}{x^2+\mu^2} \,dx
\\&= \frac{\lambda\mu}{\pi(\mu-\lambda)} \int_\R \frac{f(x)}{(x+i\lambda)(x-i\lambda)} - \frac{f(x)}{(x+i\mu)(x-i\mu)} \,dx
\\&= \frac{\lambda\mu(2\pi)^2}{\pi(\mu-\lambda)} \int_\R f(x)Q_{i\lambda}(x)\overline{Q_{i\lambda}(x)} - f(x)Q_{i\mu}(x)\overline{Q_{i\mu}(x)} \,dx
\\&= \frac{4\pi\lambda\mu}{\mu-\lambda} \left[\braket*{Q_{i\lambda}}{f \cdot Q_{i\lambda}} - \braket*{Q_{i\mu}}{f \cdot Q_{i\mu}}\right]
\\&= \frac{4\pi\lambda\mu}{\mu-\lambda} \left[f(i\lambda) \cdot Q_{i\lambda}(i\lambda) - f(i\mu) \cdot Q_{i\mu}(i\mu)\right]
\\&\overset{\fref{eq:QiLValue}}{=} \frac{4\pi\lambda\mu}{\mu-\lambda} \left[f(i\lambda) \cdot \frac 1{4\pi\lambda} - f(i\mu) \cdot \frac 1{4\pi\mu}\right]
\\&= f(i\lambda) \cdot \frac{\mu}{\mu-\lambda} - f(i\mu) \cdot \frac{\lambda}{\mu-\lambda}.
\end{align*}
Rearranging this equation, we get
\[f(i\mu) \cdot \frac{\lambda}{\mu-\lambda} = f(i\lambda) \cdot \frac{\mu}{\mu-\lambda} - f(i\lambda) = f(i\lambda) \cdot \frac{\lambda}{\mu-\lambda},\]
which is equivalent to \(f(i\mu) = f(i\lambda)\). This shows that \(f \big|_{i\R_+}\) is constant and therefore \(f\) is constant.
\end{proof}
\begin{cor}\label{cor:SymConstOp}
Let \(\cK\) be a complex Hilbert space and let \(f: \C_+ \to B(\cK)\) be a holomorphic function such that, for almost every \(x \in \R\), the limit
\[f(x) \coloneqq \lim_{\epsilon \downarrow 0} f(x+i\epsilon)\]
exists in the weak operator topology and such that \(f(-x) = f(x)\) holds. If further
\[\braket*{v}{fw} \cdot Q_{i\lambda} \in H^2(\C_+) \qquad \forall \lambda \in \R_+,v,w \in \cK,\]
then \(f\) is constant.
\end{cor}
\begin{proof}
By \fref{prop:SymConst}, for every \(v,w \in \C\), the holomorphic function
\[\braket*{v}{fw}: \C_+ \to \C\]
is constant. This implies that also \(f\) is constant.
\end{proof}

\subsection{Operators commuting with the shift}
In this subsection, we will investigate operators on \(L^2(\R,\cK)\) commuting with the shift-opera\-tors. We start with the following proposition:
\begin{prop}{\rm (\cite[Prop.~2.1.9]{Sc23})}\label{prop:SCommutant}
Let \(\cK\) be a complex Hilbert space and consider the unitary one-parameter group \({(S_t)_{t \in \R} \subeq \U(L^2(\R,\cK))}\) defined by
\[(S_tf)(x) = e^{itx}f(x), \qquad x \in \R.\]
Then
\[(S_\R)' \coloneqq \{A \in B(L^2(\R,\cK)): (\forall t \in \R)\, AS_t = S_t A\} = \{M_f: f \in L^\infty(\R,B(\cK))\}.\]
\end{prop}
\begin{proof}
Considering the functions \(e_t \in L^\infty(\R,\C)\) with \(e_t(x) = e^{itx}\) and identifying
\[L^2(\R,\cK) \cong L^2(\R,\C) \,\hat{\otimes}\, \cK,\]
we have that
\[S_t = M_{e_t} \otimes \textbf{1}_\cK, \quad t \in \R.\]
If \(f \in L^1(\R,\C)\) with
\[0 = \int_\R e_t(x) \cdot f(x) \,dx = \hat f(-t) \quad \forall t \in \R,\]
then \(\hat f = 0\) and therefore \(f = 0\), since the Fourier transform is injective. This implies that the functions \((e_t)_{t \in \R}\) span a weakly dense subspace in \(L^\infty(\R,\C)\). Therefore, the operators \((S_t)_{t \in \R}\) span a weakly dense subspace in
\[\{M_f: f \in L^\infty(\R,\C)\} \otimes \C\textbf{1}_\cK \subeq B(L^2(\R,\C)) \otimes B(\cK).\]
This, by \cite[Thm.~1]{RD75}, yields
\begin{align*}
(S_\R)' &= \left(\{M_f: f \in L^\infty(\R,\C)\} \otimes \C\textbf{1}_\cK \right)'
\\&= \left(\{M_f: f \in L^\infty(\R,\C)\}'  \otimes (\C\textbf{1}_\cK)'\right)'' = \left(\{M_f: f \in L^\infty(\R,\C)\}  \otimes B(\cK)\right)''
\end{align*}
using \cite[Cor.~2.9.3]{Sa71} in the last step. This implies the statement, since, identifying
\[L^2(\R,\cK) \cong L^2(\R,\C) \,\hat{\otimes}\, \cK,\]
by \cite[Thm.~1.22.13]{Sa71}, we have
\[\left(\{M_f: f \in L^\infty(\R,\C)\}  \otimes B(\cK)\right)'' = \{M_f: f \in L^\infty(\R,B(\cK))\}. \qedhere\]
\end{proof}
\begin{cor}\label{cor:LinTrivialHardy}
Let \(\cK\) be a complex Hilbert space and let \(A \in B(L^2(\R,\cK))\). Further let \(P_+\) be the orthogonal projection onto \(H^2(\C_+,\cK)\) and let
\[(S_tf)(x) = e^{itx}f(x), \qquad t,x \in \R, f \in L^2(\R,\cK).\]
Then \(A\) commutes both with \(P_+\) and \((S_t)_{t \in \R}\), if and only if there exists \(B \in B(\cK)\) such that
\[(Af)(x) = B f(x), \qquad x \in \R, f \in L^2(\R,\cK).\]
\end{cor}
\begin{proof}
By \fref{prop:SCommutant}, the operator \(A\) commutes with \((S_t)_{t \in \R}\), if and only if \(A = M_f\) for some function \(f \in L^\infty(\R,B(\cK))\). Now, \(M_f\) commutes with \(P_+\), if and only if
\[M_f H^2(\C_+,\cK) \subeq M_f H^2(\C_+,\cK) \qquad \text{and} \qquad M_f H^2(\C_-,\cK) \subeq M_f H^2(\C_-,\cK),\]
which, by \fref{lem:HardyBasic}(d) is equivalent to
\[f \in H^\infty(\C_+,B(\cK)) \cap H^\infty(\C_-,B(\cK)) = B(\cK) \cdot \textbf{1},\]
using \fref{lem:HardyBasic}(c) for the last equality.
\end{proof}
\begin{cor}\label{cor:AntiLinTrivial}
Let \(\cK\) be a complex Hilbert space and let \(A\) be a bounded (anti-)linear operator on \(L^2(\R,\cK)\). Further let \(P\) be the orthogonal projection onto \(L^2(\R_+,\cK)\) and let
\[(\hat S_tf)(x) \coloneqq f(x-t), \qquad t,x \in \R, f \in L^2(\R,\cK).\]
Then \(A\) commutes both with \(P\) and \((\hat S_t)_{t \in \R}\), if and only if there exists a bounded (anti-)linear operator \(B\) on \(\cK\) such that
\[(Af)(x) = B f(x), \qquad x \in \R, f \in L^2(\R,\cK).\]
\end{cor}
\begin{proof}
We first assume that \(A\) is linear and consider the Fourier transform
\[\cF: L^2(\R,\cK) \to L^2(\R,\cK), \quad (\cF f)(x) = \frac 1{\sqrt{2\pi}} \int_\R e^{-ixp} f(p) \,dp.\]
Then one has
\[\cF H^2(\C_+,\cK) = L^2(\R_+,\cK) \qquad \text{and} \qquad \cF \circ S_t = \hat S_t \circ \cF.\]
This implies that the operator \(\hat A \coloneqq \cF^{-1} \circ A \circ \cF\) commutes with \((S_t)_{t \in \R}\) and the projection onto \(H^2(\C_+,\cK)\) and therefore, by \fref{cor:LinTrivialHardy}, there exists \(B \in B(\cK)\) such that
\[(\hat Af)(x) = B f(x), \qquad x \in \R, f \in L^2(\R,\cK).\]
This immediately implies that
\[(Af)(x) = B f(x), \qquad x \in \R, f \in L^2(\R,\cK).\]
We now assume that \(A\) is anti-linear. Let \(\sigma\) be any anti-unitary involution on \(\cK\). We then define \(\tilde A \in B(L^2(\R,\cK))\) by
\[(\tilde Af)(x) = \sigma (Af)(x), \qquad x \in \R, f \in L^2(\R,\cK).\]
Then \(\tilde A\) is linear and commutes both with \(P\) and \((\hat S_t)_{t \in \R}\), so there exists \(\tilde B \in B(\cK)\) such that
\[(\tilde Af)(x) = \tilde B f(x), \qquad x \in \R, f \in L^2(\R,\cK).\]
Defining the anti-linear operator \(B \coloneqq \sigma \tilde B\), this implies
\[(Af)(x) = B f(x), \qquad x \in \R, f \in L^2(\R,\cK). \qedhere\]
\end{proof}

\newpage
\section{Standard subspaces and the affine group}
In this section, we collect some basic facts about standard subspaces and representations of the affine group. For a more detailed exposition of these two topics as well as their relation to each other, we refer to \cite{Lo08}. In the following, we present some of the results relevant to this paper.

\subsection{Positive energy representations of the affine group}\label{app:AffineGroup}
In this subsection, we provide some basic results of the representation theory of the affine group \(\mathrm{Aff}(\R)\). The affine group is an example of a so-called graded group, defined as follows:
\begin{definition}
A {\it graded group} is a pair \(\left(G,\epsilon_G\right)\) of a group \(G\) and a group homomorphism \(\epsilon_G:G \to \left\lbrace \pm 1 \right\rbrace \). We set \(G_{\pm}\coloneqq \epsilon_G^{-1}\left(\left\lbrace\pm 1\right\rbrace\right)\).
\end{definition}
\begin{example}
\begin{enumerate}[\rm (a)]
\item The affine group \(\mathrm{Aff}(\R) = \R \rtimes \R^\times\) becomes a graded group with the grading
\[\mathrm{Aff}(\R)_+ = \R \rtimes \R_+ \quad \text{and} \quad \mathrm{Aff}(\R)_- = \R \rtimes \R_-.\]
\item Given a complex Hilbert space \(\cH\) we consider the group \(\mathrm{AU}\left(\cH\right)\) of unitary and anti-unitary operators on \(\cH\). This becomes a graded group with the grading
\[\mathrm{AU}\left(\cH\right)_+ = \U(\cH) \quad \text{and} \quad \mathrm{AU}\left(\cH\right)_- = \mathrm{AU}\left(\cH\right) \setminus \U(\cH).\]
\end{enumerate}
\end{example}
Now that we have the concept of a graded group, we want to define what a representation of a graded group is:
\begin{definition}
An {\it anti-unitary representation} of a graded group \((G,\epsilon_G)\) on a Hilbert space \(\cH\) is a morphism of graded groups \({\rho:G \to \mathrm{AU}\left(\cH\right)}\).
\end{definition}
We are especially interested in so-called positive energy representations of the affine group, which are defined as follows:
\begin{definition}\label{def:PosEnergyRep}
\begin{enumerate}[\rm (a)]
\item Given a unitary representation \({\rho:\mathrm{Aff}(\R)_+ \to \mathrm{U}\left(\cH\right)}\) of the \(ax+b\)-group \(\mathrm{Aff}(\R)_+\), we say that \(\rho\) is a \textit{positive energy representation}, if, for the one-parameter group
\[U: \R \to \U(\cH), \quad t \mapsto \rho(t,1),\]
one has
\[0 \leq \partial U \coloneqq \lim_{t \to 0} \frac 1{it} (U_t - \textbf{1}).\]
We call \(\rho\) a \textit{strictly positive energy representation} if additionally \({\ker (\partial U) = \{0\}}\).
\item An anti-unitary representation \({\rho:\mathrm{Aff}(\R) \to \mathrm{AU}\left(\cH\right)}\) of the affine group is a (strictly) positive energy representation if the restriction \({\rho\big|_{\mathrm{Aff}(\R)_+}:\mathrm{Aff}(\R)_+ \to \mathrm{AU}\left(\cH\right)}\) is a (strictly) positive energy representation.
\end{enumerate}
\end{definition}
\begin{example}
Let \(\cM\) be a real Hilbert space. We set
\[\cH_\cM \coloneqq L^2(\R_+,\cM_\C).\]
Then the map
\[\rho_\cM: \mathrm{Aff}(\R) \to \mathrm{AU}(\cH_\cM)\]
with
\[(\rho_\cM(t,1)f)(x) \coloneqq e^{itx} f(x), \quad (\rho_\cM(0,e^s)f)(x) \coloneqq e^{\frac s2}f(e^s x) \qquad s,t \in \R\]
and
\[(\rho_\cM(0,-1)f)(x) \coloneqq i \cdot \cC_\cM f(x)\]
defines a strictly positive energy representation of the affine group \(\mathrm{Aff}(\R)\).
\end{example}
The following proposition shows that, in fact, every strictly positive energy representation of \(\mathrm{Aff}(\R)\) is of this form:
\begin{prop}\label{prop:RepAffineForm}
Let \((\rho,\cH)\) be a strictly positive energy representation of the affine group \(\mathrm{Aff}(\R)\). Then there exists a real Hilbert space \(\cM\) such that \((\rho,\cH)\) is equivalent to the anti-unitary representation \((\rho_\cM,\cH_\cM)\). Further, the representation \((\rho_\cM,\cH_\cM)\) is irreducible, if and only if \(\cM = \R\).
\end{prop}
\begin{proof}
By \cite[Thm.~1.6.1]{Lo08} there exists -- up to equivalence -- exactly one irreducible unitary representation of \(\mathrm{Aff}(\R)_+\) with strictly positive energy. By \cite[Sec.~2.4.1]{NO17} this representation is given by the restriction of \((\rho_\R,\cH_\R)\) to \(\mathrm{Aff}(\R)_+\). Further, by \cite[Thm.~1.6.1]{Lo08}, every unitary representation of \(\mathrm{Aff}(\R)_+\) with strictly positive energy is a multiple of this unique irreducible representation and therefore equivalent to the restriction of \((\rho_\cM,\cH_\cM)\) to \(\mathrm{Aff}(\R)_+\) for some real Hilbert space \(\cM\).

Now, by \cite[Thm.~1.6.3]{Lo08}, for every unitary representation of \(\mathrm{Aff}(\R)_+\), there exists -- up to equivalence -- just one unique extension to an anti-unitary representation of \(\mathrm{Aff}(\R)\) on the same Hilbert space. Therefore, every anti-unitary representation of \(\mathrm{Aff}(\R)\) is equivalent to \((\rho_\cM,\cH_\cM)\) for some real Hilbert space \(\cM\). Finally, by \mbox{\cite[Thm.~1.6.3]{Lo08}} the representation \(\rho_\cM\) is irreducible, if and only if its restriction to \(\mathrm{Aff}(\R)_+\) is irreducible, which is the case, if and only if \(\cM = \R\).
\end{proof}

\subsection{Standard subspaces}
In this subsection, we introduce standard subspaces and their modular objects and show how Borchers' Theorem links to positive energy representations of the affine group. We start with the definition of a standard subspace:
\begin{definition}{\rm (\cite[Sec.~2.1]{Lo08})}
Let \(\cH\) be a complex Hilbert space. A real subspace \(\sV \subeq \cH\) is called a \textit{standard subspace}, if
\[\sV \cap i\sV = \{0\} \qquad \text{and} \qquad \oline{\sV + i\sV} = \cH.\]
Given a standard subspace \(\sV \subeq \cH\), we consider the densely defined \textit{Tomita operator}
\[T_\sV : \sV+i\sV \to \cH, \quad x+iy \mapsto x-iy.\]
Then, by polar decomposition
\[T_\sV = J_\sV \Delta_\sV^{\frac 12}\]
one obtains an anti-unitary involution \(J_\sV\) and the positive densely defined operator \(\Delta_\sV\). We refer to this pair \((\Delta_\sV,J_\sV)\) as the \textit{pair of modular objects} of \(\sV\).
\end{definition}
The following proposition shows that standard subspaces always come in pairs:
\begin{prop}{\rm (\cite[Prop.~2.1.3, Thm.~2.1.4]{Lo08})}\label{prop:standardComplement}
Let \(\cH\) be a complex Hilbert space and let \(\sV \subeq \cH\) be a standard subspace. Then its symplectic complement
\[\sV' \coloneqq i \sV^{\perp_\R} = \{w \in \cH: (\forall v \in \sV)\Im \braket*{v}{w} = 0\}\]
is also a standard subspace satisfying:
\begin{enumerate}[\rm (a)]
\item \(\sV'' = \sV\).
\item \(J_{\sV'} = J_\sV\).
\item \(\Delta_{\sV'} = \Delta_\sV^{-1}\).
\item \(J_\sV \sV = \sV'\).
\end{enumerate}
\end{prop}
\newpage
A classical result in the context of standard subspaces is Borchers' Theorem, which states the following:
\begin{thm}{\rm \textbf{(Borchers' Theorem, one-particle)}\,(cf. \cite[Thm.~II.9]{Bo92} and \cite[Thm.~2.2.1]{Lo08})}
Let \(\cH\) be a complex Hilbert space and \(\sV \subeq \cH\) be a standard subspace. Further let \({(U_t)_{t \in \R} \subeq \U(\cH)}\) be a strongly continuous unitary one-parameter group such that
\[U_t \sV \subeq \sV \qquad \forall t \in \R_+.\]
If \(\pm \partial U > 0\) then the following commutation relations hold:
\[\Delta_\sV^{is} U_t \Delta_\sV^{-is} = U_{e^{\mp 2\pi s} t}, \qquad J_\sV U_t J_\sV = U_{-t} \qquad \forall s,t \in \R.\]
\end{thm}
In the following proposition, we show that the one-parameter groups described in Borchers' Theorem can be realized on the \(L^2\)-space on the real line by using our results about the representation theory of the affine group obtained in \fref{app:AffineGroup}:
\begin{prop}\label{prop:standardPairNormalForm}{\rm (\cite[Cor.~6.3.4]{Sc23})}
Let \(\cH\) be a complex Hilbert space and \(\sV \subeq \cH\) be a standard subspace. Further let \({(U_t)_{t \in \R} \subeq \U(\cH)}\) be a strongly continuous unitary one-parameter group such that
\[U_t \sV \subeq \sV \qquad \forall t \in \R_+.\]
If \(\partial U > 0\), then there exists a real Hilbert space \(\cM\) and a unitary operator
\[\psi: \cH \to (L^2(\R,\cM_\C)^\sharp,M_{i \cdot \sgn \cdot \textbf{1}})\]
such that
\[\psi(\sV) = H^2(\C_+,\cM_\C)^\sharp, \qquad \psi \circ J_\sV = \theta_{i \cdot \sgn \cdot \textbf{1}} \circ \psi, \qquad \psi \circ U_t = S_t \circ \psi \quad \forall t \in \R\]
and
\begin{equation*}
\big(\psi \Delta_\sV^{-\frac{is}{2\pi}}\psi^{-1} f\big)(x) = e^{\frac s2} f(e^s x) \quad \forall s \in \R.
\end{equation*}
\end{prop}
\begin{proof}
The commutation relations in Borchers' Theorem imply that the map
\[\rho: \mathrm{Aff}(\R) \to \mathrm{AU}(\cH)\]
with
\[\rho(t,1) = U_t, \quad \rho(0,e^s) = \Delta_\sV^{-\frac{is}{2\pi}} \quad \text{and} \quad \rho(0,-1) = J_\sV, \qquad s,t \in \R,\]
defines a strictly positive energy representation of the affine group \(\mathrm{Aff}(\R)\) (cf. \fref{app:AffineGroup}). Then, by \fref{prop:RepAffineForm}, without loss of generality, we can assume that, for some real Hilbert space \(\cM\), we have
\[\cH = L^2(\R_+,\cM_\C), \qquad U_t = S_t \quad \forall t \in \R, \qquad (J_\sV f)(x) = i \cdot \cC_\cM f(x)\]
and
\begin{equation}\label{eq:DeltaForm}
\big(\Delta_\sV^{-\frac{is}{2\pi}}f\big)(x) = e^{\frac s2} f(e^s x) \quad \forall s \in \R.
\end{equation}
Now, we consider the set
\[\Omega \coloneqq \left\{z \in \C: -\frac \pi 2 < \Im(z) < \frac{3\pi}2\right\}\]
and the functions \(Q_z \in L^2(\R_+,\C)\), \(z \in \C \setminus \R\), defined by
\[Q_z(x) = \frac 1{2\pi} \cdot \frac i{x-\overline{z}}, \qquad x \in \R\]
(cf. \fref{prop:H2KernelFunction}).

\noindent Then, for \(v \in \cM \cong \R \otimes_\R \cM\) and \(\lambda \in \R_+\), we define the maps
\[F_{\lambda,v}^\pm: \Omega \to \cH, \quad s \mapsto \exp\left(\mp \frac{is}{2\pi} \log(\Delta_\sV)\right) Q_{\pm i\lambda}v\]
and
\[G_{\lambda,v}^\pm: \Omega \to \cH, \quad s \mapsto e^{\pm \frac s2} Q_{\pm i\lambda}(e^{\pm s} \cdot)v.\]
The functions \(F_{\lambda,v}^\pm\) and \(G_{\lambda,v}^\pm\) are holomorphic and coincide on \(\R\) by Equation \fref{eq:DeltaForm}, which implies that they are equal. This especially implies that
\[\Delta_\sV^{\pm \frac 12}Q_{\pm i\lambda}v = F_{\lambda,v}^\pm(\pm i\pi) = G_{\lambda,v}^\pm(\pm i\pi) = \pm i R Q_{\pm i\lambda}v = \pm i \cC_\cM Q_{\pm i\lambda}v = \pm J_\sV Q_{\pm i\lambda}v\]
and therefore
\[J_\sV \Delta_\sV^{\pm \frac 12}Q_{\pm i\lambda}v = \pm Q_{\pm i\lambda}v,\]
so
\[H_\pm \coloneqq \overline{\spann\{Q_{\pm i\lambda}v: \lambda \in \R_+, v \in \cM\}} \subeq \ker\big(J_\sV \Delta_\sV^{\pm \frac 12} \mp \textbf{1}\big).\]
Then, using that
\[\ker\big(J_\sV \Delta_\sV^{\frac 12} - \textbf{1}\big) = \sV\]
and
\[\ker\big(J_\sV \Delta_\sV^{-\frac 12} + \textbf{1}\big) = i \ker\big(J_\sV \Delta_\sV^{-\frac 12} - \textbf{1}\big) =  i \ker\big(J_{\sV'} \Delta_{\sV'}^{\frac 12} - \textbf{1}\big) = i \sV' = \sV^{\perp_\R},\]
we get
\begin{equation}\label{eq:standardInclusion}
H_+ \subeq \sV \qquad \text{and} \qquad H_- \subeq \sV^{\perp_\R}.
\end{equation}
We now consider the unitary map
\[\Psi: L^2\left(\R_+,\cM_\C\right) \to (L^2(\R,\cM_\C)^\sharp,M_{i \cdot \sgn \cdot \textbf{1}}), \quad (\Psi f)(x) \coloneqq \begin{cases} \frac 1{\sqrt 2} f(x) & \text{if } x > 0 \\[0.75ex] \frac 1{\sqrt 2}  \cC_\cM f(-x) & \text{if } x < 0.\end{cases}\]
Then,
\[\Psi(H_\pm) = \overline{\spann\{Q_{\pm i\lambda}v: \lambda \in \R_+, v \in \cM\}}  = H^2(\C_\pm,\cM_\C)^\sharp,\]
where the inclusion \(\overline{\spann\{Q_{\pm i\lambda}v: \lambda \in \R_+, v \in \cM\}} \subeq H^2(\C_\pm,\cM_\C)^\sharp\)
follows by
\[(Q_{\pm i\lambda}v)^\sharp(z) = \overline{\frac 1{2\pi}\cdot \frac i{-z-\overline{\pm i \lambda}}} \cdot \cC_\cM v = \frac 1{2\pi}\cdot \frac i{z-\overline{\pm i \lambda}} \cdot v = (Q_{\pm i\lambda}v)(z)\]
and the equality, since, for \(f \in H^2(\C_\pm,\cM_\C)^\sharp \cap \overline{\spann\{Q_{\pm i\lambda}v: \lambda \in \R_+, v \in \cM\}}^\perp\), one has
\[0 = \braket*{\pm Q_{\pm i\lambda}v}{f} = \braket*{v}{f(\pm i\lambda)} \qquad \forall \lambda \in \R_+, v \in \cM\]
and therefore \(f\big|_{\pm \R_+} \equiv 0\), which implies \(f = 0\) (cf. \fref{lem:evDense}, \fref{lem:HardyBasic}(a)). Then, by \fref{eq:standardInclusion}, we have
\begin{align*}
H^2(\C_+,\cM_\C)^\sharp &= \Psi(H_+) \subeq \Psi(\sV) = \left(\Psi(\sV^{\perp_\R})\right)^{\perp_\R}
\\&\subeq \left(\Psi(H_-)\right)^{\perp_\R} = \left(H^2(\C_+,\cM_\C)^\sharp\right)^{\perp_\R} = H^2(\C_+,\cM_\C)^\sharp
\end{align*}
and therefore
\[\sV = H^2(\C_+,\cM_\C)^\sharp.\]
Finally, we have
\[(\Psi J_\sV \Psi^{-1} f)(x) = i \cdot \sgn(x) \cC_\cM f(x) = i \cdot \sgn(x) R f^\sharp(x) = \theta_{i \cdot \sgn \cdot \textbf{1}}f(x). \qedhere\]
\end{proof}

\newpage
\section{Evaluating some integrals}
In this section, we will evaluate some concrete integrals appearing in the proof of \fref{lem:kappaEst} and \fref{prop:betaExact}:
\begin{lem}\label{lem:Integrals}
The following equations hold for all \(x \in \R_+\):
\begin{enumerate}[\rm (a)]
\item \(\displaystyle \frac 2\pi \int_{\R_+} \frac{x}{x^2+\lambda^2} \,d\lambda = 1\).
\item \(\displaystyle \frac 2\pi \int_{\R_+} \frac{x}{x^2+\lambda^2} \cdot \frac{1}{1+\lambda^2} \,d\lambda = \frac 1{1+x}\).
\item \(\displaystyle \frac 2\pi \int_{\R_+} \frac{x}{x^2+\lambda^2} \cdot \frac{\lambda^2}{1+\lambda^2} \,d\lambda = \frac x{1+x}\).
\item \(\displaystyle \frac 2\pi \int_{\R_+} \left[\frac{\lambda}{x^2+\lambda^2} - \frac{\lambda}{1+\lambda^2}\right] \cdot \frac{\lambda}{1+\lambda^2} \,d\lambda = \frac 1{1+x} - \frac 12\).
\item \(\displaystyle \frac 2\pi \int_{\R_+} \frac{(\lambda - 1)\sqrt{\lambda}}{\lambda\left(x^2+\lambda^2\right)} \,d\lambda = \frac{\sqrt{2x}(x-1)}{x^2}\).
\item \(\displaystyle \frac 2\pi \int_{\R_+} \frac{\lambda}{x^2+\lambda^2} \cdot \frac{\sqrt{\lambda}(\lambda-1)}{1+\lambda^2} \,d\lambda = \sqrt 2 \cdot \frac{\sqrt{x}}{1+x}\).
\item \(\displaystyle \frac 2\pi \int_{\R_+} \left[\frac{\lambda}{x^2+\lambda^2} - \frac{\lambda}{1+\lambda^2}\right] \cdot \frac{\sqrt{\lambda}(\lambda-1)}{1+\lambda^2} \,d\lambda = \sqrt 2 \cdot \left[\frac{\sqrt{x}}{1+x} - \frac 12\right]\).
\item \(\displaystyle \int_{\R_+} \frac{\lambda(x^2-1)}{(\lambda^2 + x^2)(\lambda^2+1)} \,d\lambda = \log(x).\)
\end{enumerate}
\end{lem}
\begin{proof}
\begin{enumerate}[\rm (a)]
\item For every \(x \in \R_+\) one has
\begin{align*}
\frac 2\pi \int_{\R_+} \frac{x}{x^2+\lambda^2} \,d\lambda &= \frac 2\pi \int_{\R_+} \frac{x}{x^2+(xy)^2} x\,dy = \frac 2\pi \int_{\R_+} \frac{1}{1+y^2} \,dy 
\\&= \frac 2\pi \left[\arctan(y)\right]_0^\infty = \frac 2\pi \cdot \frac \pi{2} = 1.
\end{align*}
\item For every \(x \in \R_+ \setminus \{1\}\) one has
\begin{align*}
\frac 2\pi\int_{\R_+} \frac{x}{x^2+\lambda^2} \cdot \frac{1}{1+\lambda^2} \,d\lambda &= \frac 2\pi \frac x{1-x^2} \int_{\R_+} \frac{1}{x^2+\lambda^2} - \frac{1}{1+\lambda^2} \,d\lambda
\\&= \frac 1{1-x^2} \left[\frac 2\pi \int_{\R_+} \frac{x}{x^2+\lambda^2} \,d\lambda - x \cdot \frac 2\pi \int_{\R_+} \frac{1}{1+\lambda^2} \,d\lambda\right]
\\&\overset{\mathrm{(a)}}{=} \frac 1{1-x^2} \left[1 - x\right] = \frac 1{1+x}.
\end{align*}
By the continuity of both sides in \(x\), the equality also holds for \(x = 1\).
\item For every \(x \in \R_+\) one has
\begin{align*}
\frac 2\pi\int_{\R_+} \frac{x}{x^2+\lambda^2} \cdot \frac{\lambda^2}{1+\lambda^2} \,d\lambda &= \frac 2\pi\int_{\R_+} \frac{x}{x^2+\lambda^2} \cdot \left[1-\frac{1}{1+\lambda^2}\right] \,d\lambda
\\&=\frac 2\pi\int_{\R_+} \frac{x}{x^2+\lambda^2} \,d\lambda - \frac 2\pi\int_{\R_+} \frac{x}{x^2+\lambda^2} \cdot \frac{1}{1+\lambda^2} \,d\lambda
\\&\overset{\mathrm{(a)},\mathrm{(b)}}{=}1-\frac 1{1+x} = \frac x{1+x}.
\end{align*}
\item For every \(x \in \R_+\) one has
\begin{align*}
&\frac 2\pi \int_{\R_+} \left[\frac{\lambda}{x^2+\lambda^2} - \frac{\lambda}{1+\lambda^2}\right] \cdot \frac{\lambda}{1+\lambda^2} \,d\lambda
\\&\qquad\qquad= \frac 1x \cdot \frac 2\pi\int_{\R_+} \frac{x}{x^2+\lambda^2} \cdot \frac{\lambda^2}{1+\lambda^2} \,d\lambda - \frac 2\pi\int_{\R_+} \frac{1}{1+\lambda^2} \cdot \frac{\lambda^2}{1+\lambda^2} \,d\lambda
\\&\qquad\qquad\overset{\mathrm{(c)}}{=} \frac 1x \cdot \frac x{1+x} - \frac 1{1+1} = \frac 1{1+x} - \frac 12.
\end{align*}
\item One has
\begin{align*}
1+t^2\pm\sqrt{2}t = \frac 12 \left[2+2t^2\pm 2\sqrt{2}t\right] = \frac 12 \left[1+\left(\sqrt{2}t \pm 1\right)^2\right].
\end{align*}
This yields
\begin{align}\label{eq:firstBigIntegral}
\frac 12 \int_{\R_+} \frac{1}{1+t^2\pm \sqrt{2}t} \,dt &= \int_{\R_+} \frac{1}{1+\left(\sqrt{2}t \pm 1\right)^2} \,dt \notag
\\&= \left[\frac 1{\sqrt 2}\arctan\left(\sqrt{2}t \pm 1\right)\right]_0^\infty = \frac 1{\sqrt 2} \left[\frac \pi 2 - \arctan(\pm 1)\right].
\end{align}
We then have
\begin{align*}
1+t^4 &= \left(1+t^2\right)^2 - 2t^2 = \left(1+t^2+\sqrt{2}t\right)\left(1+t^2-\sqrt{2}t\right),
\end{align*}
so
\begin{align}\label{eq:secondBigIntegral}
\int_{\R_+} \frac{1+t^2}{1+t^4} \,dt &= \frac 12\int_{\R_+} \frac{1}{1+t^2+\sqrt{2}t} + \frac{1}{1+t^2-\sqrt{2}t} \,dt \notag
\\&\overset{\fref{eq:firstBigIntegral}}{=}\frac 1{\sqrt 2} \left[\frac \pi 2 - \arctan(1)\right] + \frac 1{\sqrt 2} \left[\frac \pi 2 - \arctan(-1)\right] = \frac{\pi}{\sqrt 2}.
\end{align}
For \(x \in \R_+\) one has
\begin{align*}
\int_{\R_+} \frac{xt^2-1}{1+t^4} \,dt &\overset{t = \frac 1y}{=} \int_{\R_+} \frac{x\frac{1}{y^2}-1}{1+\frac 1{y^4}} \frac 1{y^2} \,dy = \int_{\R_+} \frac{x-y^2}{y^4+1} \,dy,
\end{align*}
so
\begin{align}\label{eq:thirdBigIntegral}
\int_{\R_+} \frac{xt^2-1}{1+t^4} \,dt &= \frac 12 \left[\int_{\R_+} \frac{xt^2-1}{1+t^4} \,dt + \int_{\R_+} \frac{x-y^2}{y^4+1} \,dy\right] \notag
\\&= \frac 12 \int_{\R_+} \frac{(x-1)\left(1+t^2\right)}{1+t^4} \,dt \notag
\\&= \frac{x-1}2 \cdot \int_{\R_+} \frac{1+t^2}{1+t^4} \,dt \overset{\fref{eq:secondBigIntegral}}{=} \frac{x-1}2 \cdot \frac{\pi}{\sqrt 2}.
\end{align}
This yields
\begin{align*}
\frac 2\pi \int_{\R_+} \frac{(\lambda - 1)\sqrt{\lambda}}{\lambda\left(x^2+\lambda^2\right)} \,d\lambda &\overset{\lambda=xt^2}{=} \frac 2\pi \int_{\R_+} \frac{(xt^2 - 1)\sqrt{x}t}{xt^2\left(x^2+x^2t^4\right)} 2xt \,dt
\\&= \frac 2\pi \cdot \frac{2\sqrt{x}}{x^2} \cdot \int_{\R_+} \frac{xt^2-1}{1+t^4} \,dt
\\&\overset{\fref{eq:thirdBigIntegral}}{=}
\frac 2\pi \cdot \frac{2\sqrt{x}}{x^2} \cdot \frac{x-1}2 \cdot \frac{\pi}{\sqrt 2} = \frac{\sqrt{2x}(x-1)}{x^2}.
\end{align*}
\item For every \(x \in \R_+ \setminus \{1\}\) one has
\begin{align*}
&\frac 2\pi \int_{\R_+} \frac{\lambda}{x^2+\lambda^2} \cdot \frac{\sqrt{\lambda}(\lambda-1)}{1+\lambda^2} \,d\lambda
\\&\qquad\qquad= \frac 2\pi \int_{\R_+} \frac{\lambda^2\left(x^2-1\right)}{\left(x^2+\lambda^2\right)\left(1+\lambda^2\right)} \cdot \frac{\sqrt{\lambda}(\lambda-1)}{\lambda \left(x^2-1\right)} \,d\lambda
\\&\qquad\qquad= \frac 2\pi \int_{\R_+} \left[\frac{x^2}{x^2+\lambda^2} - \frac{1}{1+\lambda^2}\right] \cdot \frac{\sqrt{\lambda}(\lambda-1)}{\lambda \left(x^2-1\right)} \,d\lambda
\\&\qquad\qquad= \frac{x^2}{x^2-1} \cdot \frac 2\pi \int_{\R_+} \frac{(\lambda-1)\sqrt{\lambda}}{\lambda \left(x^2+\lambda^2\right)} \,d\lambda - \frac{1}{x^2-1} \cdot \frac 2\pi \int_{\R_+} \frac{(\lambda-1)\sqrt{\lambda}}{\lambda \left(1+\lambda^2\right)} \,d\lambda
\\&\qquad\qquad\overset{\mathrm{(e)}}{=} \frac{x^2}{x^2-1} \cdot \frac{\sqrt{2x}(x-1)}{x^2} - \frac{1}{x^2-1} \cdot \frac{\sqrt{2 \cdot 1}(1-1)}{1^2}
\\&\qquad\qquad=\frac{\sqrt{2x}(x-1)}{x^2-1} - 0 = \sqrt 2 \cdot \frac{\sqrt{x}}{1+x}.
\end{align*}
By the continuity of both sides in \(x\), the equality also holds for \(x = 1\).
\item For every \(x \in \R_+\), one has
\begin{align*}
&\frac 2\pi \int_{\R_+} \left[\frac{\lambda}{x^2+\lambda^2} - \frac{\lambda}{1+\lambda^2}\right] \cdot \frac{\sqrt{\lambda}(\lambda-1)}{1+\lambda^2} \,d\lambda
\\&\qquad\qquad= \frac 2\pi \int_{\R_+} \frac{\lambda}{x^2+\lambda^2} \cdot \frac{\sqrt{\lambda}(\lambda-1)}{1+\lambda^2} \,d\lambda - \frac 2\pi \int_{\R_+} \frac{\lambda}{1+\lambda^2} \cdot \frac{\sqrt{\lambda}(\lambda-1)}{1+\lambda^2} \,d\lambda
\\&\qquad\qquad\overset{\mathrm{(f)}}{=} \sqrt 2 \cdot \frac{\sqrt{x}}{1+x} - \sqrt 2 \cdot \frac{\sqrt 1}{1+1} = \sqrt 2 \cdot \left[\frac{\sqrt{x}}{1+x} - \frac 12\right].
\end{align*}
\item For every \(x \in \R_+\), one has
\begin{align*}
\int_{\R_+} \frac{\lambda(x^2-1)}{(\lambda^2 + x^2)(\lambda^2+1)} \,d\lambda &= -\frac 12 \int_0^\infty \frac{2\lambda}{\lambda^2 + x^2} - \frac{2\lambda}{\lambda^2+1} \,d\lambda
\\&= -\frac 12 \left[\log(\lambda^2+x^2) - \log(\lambda^2+1)\right]_0^\infty
\\&= -\frac 12 \left[\log\left(\frac{\lambda^2+x^2}{\lambda^2+1}\right)\right]_0^\infty = \frac 12 \log(x^2) = \log(x). \qedhere
\end{align*}
\end{enumerate}
\end{proof}

\vspace{1cm}
\noindent \textbf{Acknowledgment:} This work was conducted during my postdoctoral stay at the Instituto de Investigaciones en Matemáticas Aplicadas y en Sistemas at the National Autonomous University of Mexico in Mexico City and supported by the UNAM Postdoctoral Program (POSDOC) (Estancia posdoctoral realizada gracias al Programa de Becas Posdoctorales en la UNAM (POSDOC)) and the project CONACYT, FORDECYT-PRONACES 429825/2020 (Proyecto apoyado por el FORDECYT-PRONACES, PRONACES/429825), recently renamed project CF-2019 /429825. I would also like to thank Karl-Hermann Neeb and Maria Stella Adamo for their many useful comments on this work.

\vspace{1cm}
\noindent \textbf{Ethics declaration:} Not applicable.

\end{document}